\documentclass{article}
\usepackage{graphicx}
\usepackage{amsmath}
\usepackage{textcomp}
\usepackage{amsbsy}
\usepackage{latexsym}
\usepackage[mathscr]{eucal}
\usepackage{amssymb}
\usepackage{diagbox}
\usepackage{hhline}
\usepackage{booktabs}
\usepackage{xcolor}
\usepackage{tikz}
\usepackage{hyperref}
\usepackage{stmaryrd}
\usepackage{amsthm}

\usepackage{comment}
\usepackage{todonotes}
\usepackage{soul}
\usepackage{multirow}
\usepackage{multicol}
\usepackage{fullpage}

\newcommand\restr[2]{\ensuremath{\left.#1\right|_{#2}}}

\newtheorem{lemma}{Lemma}[section]

\newtheorem{theorem}[lemma]{Theorem}
\newtheorem{remark}[lemma]{Remark}

\def\RR{\rm \hbox{I\kern-.2em\hbox{R}}}
\def\NN{\rm \hbox{I\kern-.2em\hbox{N}}}
\def\ZZ{\rm {{\rm Z}\kern-.28em{\rm Z}}}
\def\CC{\rm \hbox{C\kern -.5em {\raise .32ex \hbox{$\scriptscriptstyle
|$}}\kern
-.22em{\raise .6ex \hbox{$\scriptscriptstyle |$}}\kern .4em}}

\def\<{\langle}
\def\>{\rangle}

\def\e{\varepsilon}

\def\cT{{\cal T}}

\def\cI{{\cal I}}

\def\cB{{\cal B}}
\def\cF{{\cal F}}

\def\cD{{\cal D}}
\def\cS{{\cal S}}
\def\cP{{\cal P}}

\def\cM{{\cal M}}

\def\Chi{\raise .3ex
\hbox{\large $\chi$}} 
\def\lsima{\hbox{\kern -.6em\raisebox{-1ex}{$~\stackrel{\textstyle<}{\sim}~$}}\kern -.4em}
\def\lsim{\hbox{\kern -.2em\raisebox{-1ex}{$~\stackrel{\textstyle<}{\sim}~$}}\kern -.2em}
\def\[{\Bigl [}
\def\]{\Bigr ]}
\def\({\Bigl (}
\def\){\Bigr )}
\def\[{\Bigl [}
\def\]{\Bigr ]}
\def\({\Bigl (}
\def\){\Bigr )}

\def\R{\mathbb{R}}

\def\T{{\relax\ifmmode I\!\!\hspace{-1pt}T\else$I\!\!\hspace{-1pt}T$\fi}}

\def\lsim{\raisebox{-1ex}{$~\stackrel{\textstyle<}{\sim}~$}}

  \def\NN{N}                  

\def\cZ{{\cal Z}}

\def\cU{{\cal U}}
\def\argmin{\mathop{\rm argmin}}

\def\cS{{\cal F}}

\def\bz{{\bf z}}

\def\cD{{\cal D}}
\def\cG{{\cal G}}
\def\cI{{\cal I}}

\def\cM{{\cal M}}
\def\cN{{\cal N}}
\def\cT{{\cal T}}
\def\cU{{\cal U}}

\def\cL{{\cal L}}
\def\cB{{\cal B}}
\def\cG{{\cal G}}

\def\cX{{\cal X}}
\def\cS{{\cal S}}
\def\cP{{\cal P}}

\def\cR{{\cal R}}
\def\cU{{\cal U}}

\def\bx{{\bf x}}

\def\bk{{\bf k}}

\def\bz{{\bf z}}

\def\argmin{\mathop{\rm argmin}}

\newcommand{\bh}{{\bf h}}
\newcommand{\by}{{\bf y}}

\newcommand{\be}{\begin{equation}}
\newcommand{\ee}{\end{equation}}
\newcommand{\bea}{$$ \begin{array}{lll}}
\newcommand{\eea}{\end{array} $$}

\newcommand{\beqn}{\begin{equation}}
\newcommand{\eeqn}{\end{equation}}

\makeatletter
\@addtoreset{equation}{section}
\makeatother

\newcommand\dist{\mathop{\rm dist}}

\newcommand\eref[1]{{\rm (\ref{#1})}}

\def\int{\intop\limits}

\newcommand\cH{{\cal H}}

\title{Convergence and error control of consistent PINNs for elliptic PDEs}
\author{Andrea Bonito, Ronald DeVore, Guergana Petrova, and Jonathan W. Siegel}

\date{V1: 06/13/24; V2:10/4/24}

\begin{document}

\maketitle
\begin{abstract}
    We provide an a priori analysis of a certain class of numerical methods, commonly referred to as {\it collocation  methods}, for solving elliptic boundary value problems.  They begin with information in the form of point values of the right side $f$ of such equations and point values of the boundary function $g$ and utilize only this information to numerically approximate the solution $u$ of the Partial Differential Equation. For such a method to provide an approximation to $u$ with guaranteed error bounds, additional assumptions on $f$ and $g$, called {\it model class assumptions}, are needed. We determine  the best error (in the energy norm) of approximating $u$, in terms of the total number of point samples,  under all Besov  class  model  assumptions  for the right hand side and boundary data.
    
    We then turn to the study of numerical procedures and analyze whether a proposed numerical procedure  (nearly) achieves the optimal recovery error.  In particular, we analyze numerical methods which generate
    the numerical approximation  to $u$ by minimizing 
 specified data driven loss
    functions  over a set $\Sigma$ which is either a finite dimensional linear space,  or more generally,  a finite dimensional manifold.   We show that the success of such a procedure depends critically on choosing a  data driven loss function that is consistent with the PDE and provides sharp error control.  Based on this analysis a  loss function $\cL^*$ is proposed.

We also address the recent methods of Physics Informed Neural Networks.
We prove  that 
     minimization of the new loss $\cL^*$ over restricted neural network spaces $\Sigma$ provides an optimal recovery  of the solution $u$, provided that the optimization problem can be numerically executed and   $\Sigma$ has  sufficient approximation capabilities. 
     We also analyze variants of $\cL^*$ which are more practical for implementation. 
     Finally, numerical examples illustrating the benefits of the proposed loss functions
     are given.
\end{abstract}

\section{Introduction}
\label{operator}
This paper is concerned with numerical methods for solving elliptic partial differential equations (PDEs).  Our primary goal is to provide a rigorous  analysis of rates of convergence for collocation methods, including   PINNs (Physics Informed Neural Networks), for solving such differential equations.

Let $\Omega$ be a bounded  domain  in the Euclidean space
 $\R^d$, $d \geq 2$,  and $\overline\Omega$ be its closure in $\R^d$.
Consider the elliptic Dirichlet problem
\begin{eqnarray}
    \label{main}
-\nabla\cdot (a({\bf x})\nabla u({\bf x}))=f({\bf x}),\quad {\bf x}\in \Omega,
\nonumber\\
u({\bf z})=g(\bz), \quad {\bf z}\in \partial \Omega,
\end{eqnarray}
where $\partial\Omega$ is the boundary of $\Omega$.  In order to prove the existence and uniqueness of a solution to \eref{main}, one needs to impose conditions on $f,g,a$,  and $\Omega$.  A  standard set of assumptions that guarantee the existence and uniqueness of a solution to  \eref{main} is  the following:
\vskip .1in
\noindent
{\bf A1:} $\Omega$ is a Lipschitz (graph) domain, i.e., $\Omega$ is an open connected set in $\R^d$ with a Lipschitz boundary in the sense described in \cite{Stein};
\vskip .1in
\noindent
{\bf A2:}  $ f\in H^{-1}(\Omega)$;
\vskip .1in
\noindent
{\bf A3:} $g\in H^{1/2}(\partial\Omega)$;
\vskip .1in
\noindent
{\bf A4:}
    the diffusion coefficient $a$ satisfies the coercivity condition
 \be 
 \label{diff}
 0<M_0\le a({\bf x})\le M,\quad {\bf x}\in \Omega,
 \ee 
  for some constants $M_0,M>0$.
Here the spaces $H^s(\Omega)$, $s\in\R$, are the Sobolev spaces of   order $s$ in $L_2(\Omega)$.   

Under the assumptions {\bf A1-A4}, the Lax-Milgram theorem (see e.g. \cite{gilbarg1977elliptic})   implies  that   \eref{main}
 has a unique solution $u\in H^1(\Omega)$ which satisfies
\be 
\label{perturb}
c\{\|f\|_{H^{-1}(\Omega)}+ \|g\|_{H^{1/2}(\partial\Omega)}\}\le \|u\|_{H^1(\Omega)}\le C\{\|f\|_{H^{-1}(\Omega)}+ \|g\|_{H^{1/2}(\partial\Omega)}\},
\ee
where the constants $c,C$ depend on $\Omega$, $M_0$ and $M$ \footnote{Throughout this paper, we use the notation $c,C,c_1,C_1$, etc, to denote constants. Unless they are absolute constants, we indicate the quantities on which they depend when
they occur.  The value of these constants can change at each occurrence. We also use the notation $A\asymp B$ to denote that $c_1A\leq B\leq c_2A$ with $c_1,c_2>0$.}.
Note that assumptions  {\bf A2-A3} could be replaced with other smoothness assumptions on $f$ and $g$, respectively, which would result in a theory similar to the one described in this paper.

In order to make our presentation as clear as possible, we want to avoid the technicalities of handling general domains and general diffusion coefficients $a$.
Accordingly, we restrict our presentation to the case $\Omega=(0,1)^d$ and  $a\equiv 1$.  This will allow us to concentrate on the novelties of PINNs and alternative collocation methods. 
We will comment on the extension of our results to more general domains $\Omega$ in  \S\ref{S:CR}.  

Note  that because of \eref{perturb}, we have that for every $v\in H^1(\Omega)$
\be
\label{equiv}
\|u-v\|_{H^1(\Omega)}\asymp \|f+\Delta v\|_{H^{-1}(\Omega)}+\|g-v\|_{H^{1/2}(\partial\Omega)},
\ee
 with absolute constants in the equivalence. 
Hence, the problem of finding the minimum of  the function $\cL_T$ 
(which we call theoretical loss function),
\be 
\label{Tloss}
\cL_T(v):=\|f+\Delta v \|_{H^{-1}(\Omega)}+ \|g-v\|_{H^{1/2}(\partial\Omega)},
 \ee
 over the whole of $H^1(\Omega)$ will have  $u$ as its unique solution, 
that is,
\be 
\label{firstmin}
u = \argmin_{v \in H^1(\Omega)} \cL_T(v).
\ee


\subsection{Numerical methods for solving \eref{main}}
\label{SS:numerical}
We are interested in numerical methods for solving \eref{main}.  The most prominent of these  are the Finite Element
Methods (FEMs) and the adaptive variations of these (AFEMs).  Over the last decades, a rigorous theory has been developed for FEMs and AFEMs which guarantees, a priori, the convergence of these numerical methods and even provides bounds on their rate of convergence under additional assumptions on $f$ and $g$.  These additional assumptions
stipulate extra smoothness of $f$ and $g$ than those assumed in {\bf A2-A3}.

A numerical method assumes that $f,g$, and $\Omega$ are known and suggests  a procedure for solving \eref{main}  based on that information.  The numerical  procedure generates a function $\hat u$ which is an approximation to $u$.
The efficiency of the numerical method is then evaluated in the following sense.  One chooses a norm $\|\cdot\|_X$
in which to measure performance and then seeks to quantify how fast the error $\|u-\hat u\|_X$ between the true solution $u$ and the output $\hat u$ of the numerical procedure  tends to zero in terms of  the complexity of the numerical procedure. The classical norm $\|\cdot\|_X$ used to measure error is the energy norm which corresponds to  choosing 
$X=H^1(\Omega)$.     Other function norms sometimes used to measure error correspond  to $X=L_p(\Omega)$, $1\le p\le \infty$.  We restrict our analysis of convergence and rates of convergence to the case $X=H^1(\Omega)$ in going further. We remark that our approach can in principle be applied to other error norms as well. In fact, the case of $X = H^{1/2}(\Omega)$ has been considered in a similar way in \cite{zeinhofer2023unified}.
 
An important question is how one can compare the performance of  different numerical methods for solving \eref{main}  in order  to give a fair competition between all possible numerical methods. This is typically done by assigning some form of complexity to the numerical method. In the case of (A)FEMs, this complexity can be measured in terms of the number $n$ of machine operations used to compute $\hat u_n=\hat u$,  and there  are theorems that
give a priori bounds for the error $\|u-\hat u_n\|_{H^1(\Omega)}$ under additional assumptions on the smoothness of $f$ and $g$ (see e.g. \cite{bonito2024adaptive}).
    Some alternatives to measuring the complexity $n$ in terms of machine  operations are   the following.  If the numerical approximation $\hat u$ to $u$ is  an element from a linear $n$-dimensional space $V_n$, then
the associated complexity is typically  assigned to be  the dimension $n$. 
This may not be directly converted to an appropriate number of machine operations because numerical stability becomes an issue. 
If the approximation $\hat u$ comes from a nonlinear manifold $\cM_n$, then one can use the number of parameters  $n$  determining  the manifold as a complexity measure.  However, in the latter  case, one has to impose extra conditions on the manifold (or the numerical procedure)  to prevent the use of space filling manifolds and thus avoid unstable numerical methods, see \cite{cohen2022optimal}.  In the absence of such restrictions, the numerical method may require an inordinate amount of computational resources to achieve a desired accuracy even in the case when the dimension $n$ is small. 

Another  issue to consider  is  in what sense  $f,g,\Omega$ are known/given to us.  One setting is to assume that we can query (ask questions about) $f,g,\Omega$,  and receive the answer to such queries. In this setting,  one has to decide which type of questions are allowed and whether these queries are answered free of charge or should the cost of asking/answering such queries be included in the numerical cost of the algorithm.    For example, in FEMs, the data used  are certain linear functionals  (inner products with the basis functions of the
Galerkin space) which are then utilized in  the FEM approximation.

In this paper, we are interested in collocation methods as described in the next section.
Accordingly, the linear functionals we 
consider are point evaluations of $f$ and $g$ at points from $\overline\Omega$.  In order for these to make sense, we assume that
$f$ and $g$ are continuous functions, i.e., $f\in C(\overline \Omega)$, $g\in C(\partial\Omega)$. 
The accuracy of how well $\hat u$ approximates $u$ will depend on two  components. The first is the number $m$ of point evaluations and their positioning.  The optimal accuracy that any numerical method can achieve from these $m$ pieces of information under the model class assumption on the function is called
the {\it optimal recovery rate}.  A central question in this paper is to determine
this optimal recovery rate under various 
model class assumptions, provided the $m$ data sites are optimally positioned.
This is the subject of \S\ref{S:OR}.

A second item in collocation methods is how one numerically uses the 
$m$ pieces of
information 
provided to create a numerical approximation to $u$.
If $\Sigma_n$ is  the underlying space   used to create $\hat u$  then
the accuracy of the numerical recovery will depend on $n$ and the choice of $\Sigma_n$. In other words, $\hat u$
can be viewed  as $\hat u=\hat u_{n,m}$.  If one fixes $m$, then one can study
how fast the accuracy of approximation $\|u-\hat u_{n,m}\|$ tends to the optimal recovery rate as $n\to\infty$.  This issue is discussed in \S\ref{S:NOR}. We remark that when neural network based methods such as PINNs are used to perform this numerical approximation, a critical component of the numerical algorithm becomes the optimization method used to train the neural network. This issue, which is specific to neural networks, is addressed in detail in \S \ref{numerical-section}.

In going further, we consider problem \eref{main} under the following assumptions:
\vskip .1in
\noindent
{\bf B1:} $\Omega=(0,1)^d$, $d\geq 2$;
\vskip .1in
\noindent
{\bf B2:}  $ f\in C(\overline\Omega)\subset H^{-1}(\Omega)$;
\vskip .1in
\noindent
{\bf B3:} $g\in H^{1/2}(\partial\Omega)\cap C(\partial \Omega)$;
\vskip .1in
\noindent
{\bf B4:}
$a\equiv 1$.
 
   Concerning the sense in which we know $\Omega$, one usually  considers  polyhedral domains whose vertices and edges are given to us.  As already noted, for simplicity, we assume that $\Omega=(0,1)^d$.  However, with some additional technicality, we could equally well start with polyhedral domain. The standard way of handling more general domains $\Omega$  is to first find a polyhedral domain $\widehat \Omega$ that approximates $\Omega$, solve the corresponding PDE on $\widehat\Omega$, and then analyze the effect of the approximation of $\Omega$ by $\widehat \Omega$.

\subsection{Collocation methods and PINNs}
\label{SS:PINN}

Recently, there has been  significant  interest in using neural networks (NNs) as a nonlinear manifold to generate the approximation  $\hat u$ to the solution $u$ of \eref{main}.  Let $\Sigma_n$ denote the set of outputs of a   neural network  with $n$ parameters  and fixed architecture.     Once $\Sigma_n$ is chosen,  the numerical procedure queries $f$ and $g$ and then uses this information to numerically construct a function $\hat u\in \Sigma_n$  which  serves as an approximation to $u$ in a specified norm $\|\cdot\|_X$.  The queries of $f$ and $g$ are  taken as point evaluations
at specified points from $\overline\Omega$,  thereby providing the values
\be\label{query}
{\bf f}=(f_1,\ldots,f_{\widetilde m}), \quad f_i:=f({\bf x}_i),\quad i=1,\dots,\widetilde m; \quad {\bf g}=(g_1,\ldots,g_{{\overline m}}), \qquad g_i:=g({\bf z}_i),\quad i=1,\dots,\overline m,
\ee 
   where   ${\bf x}_i\in \overline\Omega$, $i=1,\dots,\widetilde m$, and    ${\bf z}_i\in \partial \Omega$, $i=1,\dots,\overline m$,  are the query sites.  We refer to these points as {\it data sites} and denote  them  by 
   $$\cX:=\{\bx_1,\dots,\bx_{\widetilde m}\}, \quad  \cZ:=\{\bz_1,\dots,\bz_{\overline m}\}.
   $$  
 The performance of such a numerical method will depend on the numbers $\widetilde m,\overline m$, which we refer to as the {\it query budget},
 and  also on the positioning of these points. 

 In PINNs, the numerical procedure to find $\hat u$ is to search over $  \Sigma_n$ and find a $\hat u\in \Sigma_n$ which
  \lq fits the data\rq.  The most frequently used procedure (there are many variants) is to take  $\hat u$
 as one of the elements of the set 
 \be 
 \label{loss1}
 \hat u\, \in \,\argmin_{S\in\Sigma_n}\cL(S), 
 \ee
 where 
 \be\label{square-root-loss-definition}
    \cL(S):=\left[\frac{1}{\widetilde m}\sum_{i=1}^{\widetilde m} [\Delta S (\bx_i)+f(\bx_i)]^2\right]^{1/2}+\left[\frac{1}{\overline{m}} \sum_{i=1}^{\overline m} [S(\bz_i)-g(\bz_i)]^2\right]^{1/2}.
 \ee
 Note that this is equivalent to the square root of the original PINNs loss \cite{raissi2019physics}. Later in the paper we will consider weighted versions of $\cL$.
 The appropriateness of the loss $\cL$ and how $\widetilde m,\overline m$, and the data sites 
 $(\cX,\cZ)$ should be chosen are important  issues and  are one of the  focal points of  this paper. 
We remark that an analysis of other methodologies for solving \eqref{main} using neural networks, such as the deep Ritz method \cite{weinan2018deep} and finite neuron method \cite{xu2020finite}, can be found in \cite{lu2021priori,lu2022priori,siegel2023greedy,muller2022error,duan2022convergence,dondl2022uniform}.
 
The  PINNs numerical procedures fall into the broad class of {\it collocation methods}  for solving PDEs, i.e. numerical methods that use only values of $f$ and $g$ at specified data sites $(\cX,\cZ)$.  Such methods were studied in the last century (see e.g. \cite{AS,HRV,HHRV}) but became less popular with the advent of FEMs.   The novelty of PINNs is to use neural networks instead of  polynomials or splines  to build the approximation $\hat u$ from the  given data.  

Some assumptions are necessary for collocation methods to make sense.  Firstly, to ensure that point values of $f$  and $g$
 make sense, we assume that   $f$ and $g$  are continuous function on $\overline\Omega$ and $\partial\Omega$, respectively.  Secondly, to have any hope of proving a priori guarantees on convergence or rates of convergence
as $\widetilde m,\overline m,n\to\infty$, we   need at a minimum that $f$ comes from  a compact subset $\cF$ of $C(\overline{\Omega})$ and $g$ comes from  a compact subset
$\cG$ of $C(\partial \Omega)$.  This  in turn guarantees that $u$ is in a compact subset $\cU$ of $H^1(\Omega)$.

\subsection{Overview of this paper}
\label{SS:overview}
The idea of using neural networks to solve PDEs goes back to the last century \cite{dissanayake1994neural,lagaris1998artificial}, and has recently become extremely popular with the introduction of PINNs \cite{raissi2019physics}. Despite the increasing  usage of PINNs for numerically solving a wide range  of PDEs  (see for example,  \cite{karniadakis2021physics,cuomo2022scientific,cai2021physics,mao2020physics} and the references therein),
a satisfactory analysis of the convergence and performance of these methods has not yet been given. We refer to \cite{de2024numerical} and the references therein for an overview of the existing approaches. 
Some partial progress towards this goal has recently been made, see for instance \cite{gazoulis2023stability,mishra2023estimates,SDK,shin2023error,zeinhofer2023unified}.

For example, in \cite{SDK}, convergence is proved in the $C(\overline{\Omega})$-norm for a modified physics-informed neural network for elliptic PDEs under the assumption that the right hand side $f$ and boundary data $g$ are H\"older continuous. Under the additional assumption that the outputs of the neural network satisfy the boundary conditions exactly, convergence is obtained in $H^1(\Omega)$. In this situation, the boundary values are \textit{not enforced} via the loss function and must instead be  implemented  through the neural network architecture. We remark that although methods for exactly enforcing boundary conditions with neural networks have been proposed in \cite{lagaris1998artificial,lagaris2000neural,berg2018unified}, this approach is unable to rigorously handle arbitrary boundary conditions on general domains.

Further, in \cite{shin2023error} and \cite{zeinhofer2023unified}, an abstract framework for analyzing methods based upon residual minimization, such as PINNs, is presented. This framework has been used in a variety of follow up works to analyze PINNs (see \cite{bai2021physics,tang2023pinns,doumeche2023convergence} for instance). 

An investigation of the PINNs method  for solving a variety of equations, obtaining upper bounds on the rate of convergence with respect to the number of collocation points, has been obtained in \cite{mishra2023estimates}. In \cite{lu2021machine},  a mini-max analysis for various neural network PDE formulations (with built in zero boundary conditions) is given for the Poisson equation \textit{with zero boundary values} under the assumption that noisy values of  the right hand side of the PDE are measured at random points.

Our work differs from the previous approaches  in the following important ways:
\begin{itemize}
    \item  we consider general Sobolev and Besov smoothness assumptions on the RHS and boundary data (instead of just the Hilbert space setting as is typical in prior work) and consider error in the energy norm $H^1(\Omega)$ (in constrast to the weaker norms often considered in prior work); Besov and Sobolev regularity is the  standard staple of formulating regularity theorems for PDEs;  
    \item we obtain matching upper and lower bounds on the theoretically optimal convergence rate in terms of the number of collocations points;
    \item we show that a properly designed loss function can be used to obtain an a posteriori error estimate on the solution error for the elliptic PDE;
    \item based upon our  analysis, we propose new loss functions,  including a loss function  which correctly balances the boundary and domain error in the original PINNs formulation, and numerically demonstrate their improved performance on problems with (relatively) low regularity.
\end{itemize}
In particular, our analysis concentrates on the following questions:

\vskip .1in
\noindent 
{\bf Q1:}  Given the model class assumptions on $f$ and $g$, which query sites should be used and how large would we have to take $\widetilde m$ and $\overline m$ to 
 be able to reach a target accuracy $\varepsilon$ for the error   $\|u-\hat u\|_{H^1(\Omega)}$?
\vskip .1in
\noindent 
{\bf Q2:}  Given a budget of $m$ queries and given model class assumption on $f$ and $g$, what is the smallest error that can be achieved in the recovery of $u$?
This is called the {\it optimal recovery} accuracy.
\vskip .1in
\noindent 
{\bf Q3:}  If we use a NN space $\Sigma_n$ with $n$  parameters to build the approximation $\hat u$ to $u$, then  how large do we need to choose $n$ to achieve near optimal accuracy?  
\vskip .1in
\noindent
{\bf Q4:} If we use minimization of a data driven loss function to generate the numerical approximation $\hat u$, then what form should the loss function take  so  that small values of this loss function guarantee small values of the solution approximation error, up to the optimal recovery rate
?
\vskip .1in

We answer the above questions for all Besov space model classes and thereby establish such  an a priori analysis of collocation methods and in particular PINNs. In contrast to previous works, we consider the error in the $H^1(\Omega)$-norm, and treat the situation where the boundary values are enforced through the loss function. 
We assume that 
 $f\in\cF$, $g\in\cG$, where $\cF$ and $\cG$ are compact subsets of the space of continuous functions on the domains 
$\overline{\Omega}$ and $\partial \Omega$, respectively.  Such compact sets are typically described by smoothness
conditions.  In this paper, we measure smoothness   by membership in Besov classes.  For this reason, we begin in \S \ref{S:Besov} with the definition of Besov spaces and the properties we will need in the sequel.  We also
discuss there how well functions in Besov classes can be approximated (in various norms) by piecewise polynomials.  All of the results in that section are known and are therefore presented in a  summary form.  However, for completeness of the paper, and for the specific presentation of these results,  we give proofs of the results we need in the appendix.

We use membership in Besov spaces $B_q^s(L_p(\Omega))$ as the smoothness we impose on $f$ and $g$. 
 Once   such assumptions are placed on $f$ and $g$,  we  answer  questions
 {\bf Q1-Q2} in  \S\ref{S:OR}.
 Questions of this type are commonly referred to as {\it optimal recovery} (OR) of a function $u$ from   data. Since our data is given by point samples, this is also known as {\it optimal sampling}.  Our general setting is not  usually  addressed in the OR literature because our data is not point evaluations of the target function $u$, but rather of $f$ and $g$. Nevertheless, we consider the content  of \S\ref{S:OR} to be for the most part known in the sense that we are simply piecing together various known results and techniques such as those in 
     \cite{DNSI,dahlke2006optimal,KNS, NT}. 
     The paper \cite{vybiral2007sampling}, which studies optimal sampling with respect to the negative Sobolev norm $H^{-1}(\Omega)$, comes closest to presenting the optimal recovery results in the form we need.
     However, note that the cases $0<p,q<1$ for all $d$, or $p=1$, 
     $0<q\leq\infty$, $d=2$ are not covered in \cite{vybiral2007sampling}. Notably, we observe that the optimal sampling rates for Besov spaces in $H^{-1}(\Omega)$ coincide with the optimal sampling rates in a certain space $L_\gamma(\Omega)$, determined by the Sobolev embedding (see Theorem \ref{T:ORf} (iv)). 
 
   The results of \S \ref{S:OR} do not give a numerical method for finding $\hat u$ from the given data that provides the error bound $\|u-\hat u\|_X\le \e$.
 They simply establish the theoretical benchmark for the optimal  accuracy of  any numerical procedure for the recovery of $u$ based on the point samples of  $f$ and $g$, see the discussion at the beginning of \S\ref{S:NOR2}. 
 Our   analysis there  does not involve
 NNs or any other proposed collocation method.  NNs and PINNs   only enter the picture when one wants to
 transform the theoretical analysis into a numerical procedure that utilizes NNs to find an approximant to $u$.
We then further discuss in \S\ref{S:NOR2}   the use of classical numerical methods such as Finite Element Methods (FEMs) for optimal recovery.

In \S\ref{S:NOpt},  we turn to the question of using optimization to solve the PDE.
 We observe   that minimizing the theoretical  loss $\cL_T$ over a space $\Sigma_n$ will approximate well $u$, provided $\Sigma_n$ has sufficient approximation power.   However, this is not a numerical algorithm per se because it does not provide a numerical 
recipe for evaluating $\cL_T(S)$ when $S\in\Sigma_n$.

The remainder of this paper discusses possible numerical methods to achieve the optimal recovery rate,  i.e., to solve the PDE to the highest accuracy possible given only the information $({\bf f},{\bf g})$ (see \eqref{query}). 
 An important consequence of the analysis in \S\ref{S:NOpt}  is   that
the widely used  loss function $\cL$ may not be an appropriate discretization of the theoretical loss $\cL_T$. We introduce the loss function $\cL^*:H^1(\Omega)\to\R$ given by
 \be
 \label{correctloss}
 \cL^*(v)\!:= \!\!\left [\frac{1}{\widetilde m}\sum_{i=1}^{\widetilde m} |f(\bx_i)+\Delta v(\bx_i)|^\gamma \right ]^{1/\gamma}\!\!\!\!+
 \left [\frac{1}{\overline m^2}\sum_{\substack{i,j=1 \\ i\neq j}}^{\overline m}
 \frac{|[g-v](\bz_i)-[g-v](\bz_j)|^2}{|\bz_i-\bz_j|^d}\right ]^{1/2}\!\!\!  +\left[\frac{1}{\overline m}\sum_{j=1}^{\overline m}|g(\bz_j)-v(\bz_j)|^2\right]^{1/2},
 \ee
 and show that it correctly discretizes an upper bound on the theoretical loss $\cL_T$.
 Here $\gamma$ is the smallest number 
 (if this is possible) so that $L_\gamma(\Omega)$ embeds in $H^{-1}(\Omega)$. For example, $\gamma=\frac{2d}{d+2}$ in the case $d\geq 3$. The case $d=2$ is more complicated and we will discuss it later in the paper.
 
The  loss $\cL^*$ and its properties are given in \S\ref{S:NOpt}, \S\ref{S:discrete} and 
 \S\ref{S:dloss}.  A large component of its development centers on how to discretize $L_p$  (quasi-)norms and the $H^{1/2}$ norm.  Our analysis on this subject may
 be useful in the development of other methods for solving PDEs numerically.

 The first term    in the loss $\cL^*$ are shown to bound $\|f+\Delta v\|_{H^{-1}(\Omega)}$ and the second   and third terms are shown to bound $\|g-v\|_{H^{1/2}(\partial\Omega)}$.  These in turn give a bound for $\|u-v\|_{H^1(\Omega)}$ via
 \eref{equiv}.
 We go on to show how  near minimizers of the suggested loss $\cL^*$  give  an  a priori bound of the approximation error,  provided that the approximation method $\Sigma_n$ has suitable approximation properties. The a priori guarantee provides a near optimal recovery whenever $\Sigma_n$ provides sufficiently good approximation of the elements in the solution model class $\cU$.

 In \S\ref{S:implications}, we 
 analyze a variety of other possible loss functions which provide the same a posteriori error guarantees as the loss $\cL^*$. We call the use of any of these loss functions \emph{consistent PINNs}. 
 If the set $\Sigma_n$ is chosen large enough, all these losses will provide a near best optimal recovery (in terms of number of data points) of the solution $u$. 
We  prove that for any function $v\in H^1(\Omega)$, the losses suggested in \S\ref{S:implications} always provide an upper bound of the true error $\|u-v\|_{H^1(\Omega)}$  up to the optimal recovery rate for the solution model class $\cU$, provided that the approximant $v$ can be controlled in the $\|\cdot\|_{\cU}$-norm (see Theorem \ref{T:dloss}).   This means that these losses can be used as an a posteriori error estimator for any proposed numerical approximation $v$ to $u$. This error estimate can be used to check whether the output of a PINNs optimization achieves the desired accuracy $\e$. 

However, proving that a particular numerical method of optimization (such as gradient descent)
converges to a minimizer (or near minimizer) of the loss is a serious issue in optimization, which we do not
address here.

 We next develop in \S\ref{S:NOR} the  properties needed of a set  $\Sigma_n$ so that minimizing the loss $\cL^*$ over   $\Sigma_n$  results in a near optimal recovery of $u$.   
 It turns out that a certain restricted approximation property, guaranteed to hold
 when $\Sigma_n$ is sufficiently good at approximating the elements of solution model class $\cU$, is sufficient.
 In particular, we show that this is the case for suitable NN spaces $\Sigma_n$, provided $n$ is large  enough.

Finally, in \S\ref{numerical-section}, we numerically test the loss functions described in \S\ref{S:implications} on a two-dimensional Poisson equation. These results show that using consistent PINNs provides a significant improvement in the solution error, especially for problems with low regularity.

\subsection*{Acknowledgments}
We are grateful to Peter Binev, Albert Cohen, Wolfgang Dahmen, and Jinchao Xu for many insightful conversations about the material in this paper. We also thank the referees for their careful reading of the manuscript  and their insightful comments/suggestions which helped improve the paper.

This research was supported by the NSF Grants
DMS-2409807 (AB), DMS-2424305 (JWS), CCF-2205004 (JWS), and DMS 2134077 (RD and GP), and the MURI ONR Grant N00014-20-1-278 (RD, GP, and JWS).

\section{Besov spaces}
\label{S:Besov}
  
We start this section by recalling the definition of Besov spaces 
and  their properties.   
For the range  $s>0$, and $0<p,q\le\infty$,  the Besov space $B_q^s(L_p(\Omega))$, $\Omega=(0,1)^d$, $d\ge 2$,  is a space of functions with smoothness of order $s>0$ in $L_p(\Omega)$.  Here $q $
is a secondary index that gives a fine gradation of these spaces.    
The material in this section is taken for the most part from the papers \cite{DS,DP,DSmono} and the reader will have to refer
to those papers for some of the definitions and proofs.  Let us also mention that the univariate
case is covered in the book \cite{DL}.

If $r$ is a positive integer, $0<p\le \infty $ and $f\in L_p(\Omega)$, we define the modulus of smoothness $\omega_r(f,\cdot)_{p}$ of $f$ by
\be 
\label{moduli}
\omega_r(f,t)_{p}:=\omega_r(f,t,\Omega)_{p}:= \sup_{|{\rm \bf h}|\le t} \|\Delta_{\rm \bf h}^r(f,\cdot)\|_{L_p(\Omega_{r{\rm \bf h}})}, \quad t>0,
\ee
where 
\be 
\label{rdiff}
\Delta_{\rm \bf h}^r(f,\cdot):=(-1)^r\sum_{k=0}^r (-1)^k\binom{r}{k} f(\cdot+k{\rm \bf h}),
\ee 
is the $r$-th difference of $f$ for ${\rm \bf h}\in\R^d$ and $\Omega_{\rm \bf h}:=\{{\rm \bf x}\in\Omega: [{\rm \bf x},{\rm \bf x}+{\rm \bf h}]\subset \Omega\}$.  Here
$[{\rm \bf x},{\rm \bf x}+{\rm \bf h}]$, 
${\rm \bf x},{\rm \bf h}\in\R^d$, denotes the line segment in $\R^d$ between ${\rm \bf x}$ and ${\rm \bf x}+{\rm \bf h}$,  and $|{\rm \bf h}|$ denotes the Euclidean norm of ${\rm\bf h}$. If $s>0$ and $0<p\le  \infty$, then
$B_q^s(L_p(\Omega))$ is defined as the set of all functions in $L_p(\Omega)$ for which
\be 
\label{Bsemi}
|f|_{B_q^s(L_p(\Omega))}:= \left [\int_0^1 [t^{-s}\omega_r(f,t)_p]^q\frac{dt}{t}\right ]^{1/q}<\infty,\quad 0<q<\infty,
\ee 
where $r$ can be taken as any integer strictly bigger than $s$. When $q=\infty$ we replace the integral by a supremum in the definition. This is a (quasi-)semi-norm and we obtain the (quasi-)norm for $B_q^s(L_p(\Omega))$
by adding $\|f\|_{L_p(\Omega)}$ to it. An equivalent (quasi-)semi-norm is given by
\be 
\label{Bsemi1}
|f|_{B_q^s(L_p(\Omega))}\asymp \left [\sum_{k=0}^\infty [2^{ks}\omega_r(f,2^{-k})_p]^q \right ]^{1/q},\quad 0<q<\infty,
\ee 
with equivalency constants 
independent of $f$.
This is proved by discretizing the integral in \eref{Bsemi} and using the monotonicity of $\omega_r$
as a function of $t$.  When $q=\infty$, \eref{Bsemi1} uses the $\ell_\infty$ norm in place of the $\ell_q$ norm, i.e.,
\be
\label{Bsemi11}
|f|_{B_\infty^s(L_p(\Omega))}\asymp \sup_{k\ge 0}   2^{ks}\omega_r(f,2^{-k})_p.
\ee
While different choices of $r$ result in different (quasi-)semi-norms, the corresponding Besov (quasi-)norms are equivalent, provided $r>s$.
To fix matters, we define the Besov norm with the value of $r=r(s)$ as the smallest integer strictly larger than $\max\{s,1\}$.  It follows that $r(s)$ is always larger than   or equal to  $2$.  This choice is for notational convenience in the material that follows.

Let us make some remarks on these spaces. Consider the  role of $q$ in this definition.  If $q_2<q_1$, then $B_{q_2}^s(L_p(\Omega))\subset B_{q_1}^s(L_p(\Omega)) $, i.e., these spaces get smaller as $q$ gets smaller.  Thus, all 
$B_q^s(L_p(\Omega))$, $q>0$, are contained in $B_\infty^s(L_p(\Omega))$ once $s$ and $p$ are fixed. The effect of $q$ in the definition of the Besov spaces is subtle. We remark that when $p=q$ and $s$ is not an integer, the space $B^s_p(L_p(\Omega))$ is equivalent to the Sobolev space $W^s(L_p(\Omega))$ (see \cite{adams2003sobolev}).  

In this paper, the space $B_\infty^s(L_p(\Omega))$  is especially important, and accordingly, we use the abbreviated
notation
\be 
\label{Hsmoothness}
B_p^s:=B_p^s(\Omega):=B_\infty^s(L_p(\Omega))
\ee 
in going forward.  It follows that a function $f\in B_p^s(\Omega)$ if and only if
\be 
\label{member}
\omega_r(f,t)_{p}\le Mt^s,\quad t>0,
\ee
and the smallest $M$ for which \eref{member} is valid is the (quasi-)semi-norm $|f|_{B_p^s}$.
This space is commonly referred to as (generalized)   H\"older smoothness of order $s$ in $L_p$. It is important to note that when $s$ is a positive  integer, we take $r>s$ in its definition. Therefore the Besov space $B_p^s$ is not a Lipschitz space when $s$ is a positive integer.  In view of \eref{Bsemi11}, we have that
 a function $f$ is in $B_p^s$ if and only if
\be 
\label{member1}
\omega_r(f,2^{-k})_{L_p(\Omega)}\le M'2^{-ks},\quad k=0,1,\dots,
\ee
and the smallest $M'$ for which this is true is an equivalent (quasi-)semi-norm.

\subsection {The Sobolev spaces $H^s(\Omega)$}
\label{SS:Sobolevspaces}
Since the results of this paper heavily use the Sobolev spaces $H^s(\Omega)$, we give here a short review of these spaces  and list some of their properties. The classical Sobolev  spaces $H^s(\Omega)$, $s>0$,  are the Besov spaces 
$$
H^s(\Omega)=B^s_2(L_2(\Omega))\subset 
B^s_\infty(L_2(\Omega))=:B^s_2(\Omega).
$$
 
Two of the $H^s(\Omega)$ spaces play an important role in what follows, namely $H^{-1}(\Omega)$ and $H^{1/2}(\partial \Omega)$.  Let us recall  the definition of $H^{-1}(\Omega)$.
 The space $H_0^1(\Omega)$ is the collection  of functions in $H^1(\Omega)$ which vanish on the boundary of $\Omega$, i.e., it is the closure of smooth functions in $H^1(\Omega)$ which are identically zero on $\partial\Omega$. The space $H^{-1}(\Omega)$ is by definition the dual space of $H_0^1(\Omega)$ and is equipped with the norm
\be 
\label{H-1norm}
\|f\|_{H^{-1}(\Omega)}:=\sup_{\|v\|_{H_0^1(\Omega)}=1}\langle f,v\rangle_{H^{-1}\times H_0^1}.
\ee

There are two equivalent ways to describe the space $H^{1/2}(\partial \Omega)$.  The first is as the Besov space $B_2^{1/2}(L_2(\partial\Omega))$.  Here, one needs the concept of Besov spaces on manifolds.  In our case, the boundary of $\Omega$ is quite simple since it is
the union of the faces of $\Omega$.  This definition gives the norm in $H^{1/2}(\partial\Omega)$ through the modulus of smoothness.
 Using the averaged modulus of smoothness, see  \eref{avemod} in the Appendix, \S\ref{SS:polyapprox}, one finds
that the semi-norm of $g\in H^{1/2}(\partial\Omega)$,  is
\be 
\label{inorm}
|g|^2_{H^{1/2}(\partial\Omega)}:=\int_{\partial\Omega}\int_{\partial\Omega}\frac{|g(\bz)-g(\bz')|^2}{|\bz-\bz'|^d}\, d\bz\, d\bz',
\ee
which is commonly referred to as the {\it intrinsic semi-norm} on this space.  We obtain the norm on this space by adding $\|g\|_{L_2(\partial\Omega)}$ to this semi-norm.

The second way to describe $H^{1/2}(\partial\Omega)$ and its norm is through the trace operator $Tr=Tr_{\partial\Omega}$.  Recall that if $g$ is a continuous function on $\overline \Omega$ then the trace $Tr(g)$  is simply the restriction of $g$ to $\partial\Omega$. While the trace operator has an extension to certain functions that are not continuous, in our applications that appear later the function $g$ will always be continuous.  Using the trace operator one defines $H^{1/2}(\partial\Omega)$ as the trace of functions in $H^1(\Omega)$ and its {\it trace norm} is
\be 
\label{tH1/2}
\|g\|_{H^{1/2}(\partial \Omega)}:= \inf\{\|v\|_{H^1(\Omega)}:\ Tr(v)=g\}.
\ee 
It is well known that the trace norm and the intrinsic norms for $H^{1/2}(\Omega)$ are equivalent for Lipschitz domains and in particular for our case of $\Omega=(0,1)^d$ (see \cite{grisvard2011elliptic}). We use both of these norms in what follows while always making clear which norm is being employed.

\subsection{ Piecewise polynomial approximation and interpolation}
\label{SS:pwpapprox}

Recall that one can characterize membership in Besov spaces by piecewise polynomial approximation. We describe and prove such characterizations in the appendix.   An important additional fact is that the piecewise polynomials in such characterizations can be described by {\it interpolation}.   This allows us to generate near best piecewise polynomial approximations   to $f$ and $g$ by using   only the data $(f_i)$, respectively $(g_i)$.   

Let us begin with the cube $\Omega=(0,1)^d$ and discuss polynomial interpolation on $\overline\Omega$ which  
we will later rescale to any dyadic cube.   We fix  $r\in\mathbb N$,  $r> 1$ and let 
$$
G_r:=\left\{
\left(\frac{j_1}{r-1}, \ldots,\frac{j_d}{r-1}\right),\ j_i\in\{0,1,\ldots,r-1\}\right\}\subset [0,1]^d,
$$ be the tensor product grid
of  $r^d$ equally spaced points in $\overline\Omega$ . 
Consider the simplicial (Kuhn-Tucker)  decomposition
of $\overline\Omega$ into simplices $T^0$, 
\be
\label{simp}
\overline\Omega=\bigcup\overline T^0.
\ee
 Given any one of these simplices $T^0$, the number $n_r $ of grid points from $G_r$ in the simplex $\overline {T^0}$ is the same as the dimension  of $\cP_r:=\cP^d_r$, 
 where $\cP^d_r$ denotes the linear space of algebraic polynomials of order $r$ (total degree $r-1$), namely, 
$$
\cP^d_r:=\left\{\sum_{|{\bk}|_1<r}a_{\bk}\bx^{\bk}, \,\,a_{\bk}\in\R\right\}, 
\quad \hbox{where}\quad \bx^{\bk}:=x_1^{k_1}\cdots x_d^{k_d}, \quad \bk:=(k_1,\dots,k_d), \quad k_j\geq 0,
\quad |\bk|_1:=\sum_{j=1}^d k_j.
$$
Let us consider 
  the closed  simplex 
  $$
 \overline{T}_0:=\{{\bf x}=(x_1, \ldots,x_d)\in \R^d: \,x_1+\ldots+x_d\leq 1, \,\,x_i\geq 0,\,i=1, \ldots,d\}.
  $$
It is well known that polynomial interpolation using elements of $\cP_r$ at the  grid points in $G_r\cap   \overline{T_0}$ is well posed (see e.g. Section 4 in \cite{SX}  or Chapter 2 in \cite{ciarlet2002finite}).
Let $\phi_j:=\phi_{j,T_0}$, $j=1,\dots,n_r$,  be the Lagrange polynomial
basis for $\cP_r$ corresponding to the points in $G_r\cap \overline {T_0}$.  Then, the operator
\be 
\label{defL}
L_{T_0}(f):=\sum_{\bx_j\in G_r\cap \overline{T_0}} f(\bx_j)\phi_j,
\ee 
is a bounded projector from $C(\overline\Omega)$ onto $\cP_r$, i.e.,
\be 
\label{projnorm}
\|L_{T_0}(f)\|_{C(\overline{T_0})}\le \Lambda_r\|f\|_{C(\overline{T_0})},
\quad \hbox{where}\quad \Lambda_r:=\|\sum_{j=1}^{n_r}|\phi_{j,T_0}({\bx})|\|_{C(\overline{T_0})},
\ee 
is  the Lebesgue constant which depends only on $r$, $d$,  and $T_0$. 

For each $k\ge 0$, we define $\cD_k=\cD_k(\Omega)$ to be the  partition of $\Omega$ into  dyadic cubes $I$ of side length $2^{-k}$.  Here we take a dyadic cube $I\in\cD_k$  to be the tensor product
  of the dyadic intervals $[(j-1)2^{-k},j2^{-k})$, $1\le  j\leq 2^k$.
We  then define the tensor product grid $G_{k,r}$ of $\overline\Omega=[0,1]^d$,
 \be
 \label{Gd}
  G_{k,r}:=\{\bx_1, \ldots,\bx_{\widetilde m}\}\subset \overline \Omega, \quad \widetilde m=(r2^k)^d,
\ee
   with spacing $h=2^{-k}(r-1)^{-1}$.
  We fix $k\ge 0$, $I\in\cD_k$,  and a simplex  $T$ in the simplicial decomposition of $I$, and  denote by $\phi_{j,T}$, the rescaled polynomial $\phi_{j,T}:= \phi_{j,T_0} \circ F_T^{-1}$ where $F_T: \overline{T}_0 \rightarrow \overline{T}$ is an affine map from the reference simplex $\overline{T}_0$ to $\overline{T}$. Notice that $\phi_{j,T}$, $j=1, \ldots,n_r$, are the Lagrange polynomial
basis for $\cP_r$, corresponding to the points in $G_{k,r}\cap \overline T$, and thus we obtain the projector $L_T$ onto $\cP_r$ which  satisfies
\be 
\label{projnorm1}
\|L_T(f)\|_{C(\overline T)}\le \Lambda_r(T) \|f\|_{C(\overline T)}, \quad \hbox{where}\quad \Lambda_r(T):=\|\sum_{j=1}^{n_r}|\phi_{j,T}({\bx})|\|_{C(\overline{T})},\quad T\subset I\in\cD_k,\ k\ge 0.
\ee 
   Here $\Lambda_r=\Lambda_r(T)$,
   i.e. $\Lambda_r(T)$ does not depend on $T$ since
   for $\bx\in \overline T$ we have $\phi_{j,T}({\bx})=\phi_{j,T_0}(F^{-1}_T(\bx))=\phi_{j,T_0}(\by)$, where $\by=F^{-1}_T(\bx)\in \overline {T_0}$. Moreover, the interpolating polynomial $L_T(f)$ provides a good   approximation to $f$ on $T$.  For example, for    any polynomial $P\in\cP_r$, we have
\be 
\label{projapprox}
\|f-L_T(f)\|_{C(\overline T)}\le \|f-P\|_{C(\overline T)}+ \|L_T(f-P)\|_{C(\overline T)}\le (1+\Lambda_r)\|f-P\|_{C(\overline T)}.
\ee

Let us denote by $S^*_{k}(f)$ the piecewise polynomial interpolant 
\be 
\label{pwpinter1}
S^*_k(f):=S^*_{k,r}(f):=\sum_{T\in\cT_{k}}L_T(f)\chi_T, 
\ee 
where $\cT_{k}$ is the collection of all simplicies arising from the 
(Kuhn-Tucker) decomposition of the dyadic cubes in $\cD_k(\Omega)$  and 
$\chi_T$ is the characteristic function of $\overline T$. The function
$S^*_{k}(f)$  is well defined and   continuous on $\overline\Omega$
 since the number  of points from $G_{k,r}$ that are on a 
 shared face of any two simplicies from $\cT_k$ is the same as the dimension of the polynomial space $\cP^{d-1}_r$ .  

The following theorem (proved in the  Appendix, \S\ref{T2.1}) describes the approximation accuracy of  $S^*_{k}$.

\begin{theorem}
    \label{T:interapprox}
     Let $s>0$,  $0<p\leq \infty$,  be fixed with $s>d/p$.  Let $r=r(s)$ be the smallest integer strictly larger than $\max(s,1)$ 
     and let $S_k^*$ be the interpolation operator \eref{pwpinter1}.  Then, for any $f\in B_p^s(\Omega)=B_\infty^s(L_p(\Omega))$ and any $   \tau\ge p$, we have
     \be 
     \label{Tpwpinter}
     \|f-S_k^*(f)\|_{L_\tau(\Omega)}\le C|f|_{B_p^s(\Omega)} 2^{-k(s-d/p+d/\tau)},
     \ee
     with $C$ independent of $f$ and $k$.
\end{theorem}

We will also need similar results where the approximation takes place   in the $H^1(\Omega)=B_2^1(L_2(\Omega))$ norm. The following theorem, see the Appendix, \S\ref{T2.2} for its proof, holds.
\begin{theorem}
    \label{T:SkH1}
    Let  $d\ge 2$.  If $f\in B_q^s(L_p(\Omega))$ with $s>d/p$ and $0<p\le 2$, and $S^*_k(f)$ is as defined in \eref{pwpinter1}, then
    \be 
    \label{projapprox3}
\|f-S_k^*(f)\|_{H^1(\Omega)}\le  C|f|_{B_q^s(L_p(\Omega))} 2^{-k(s-1-d/p+d/2)},\quad k\ge 0.
\ee
\end{theorem}

There are many results concerning the embeddings of Besov spaces into other Besov spaces or $L_p$ spaces.  We are particularly interested in those when a Besov space is embedded into $C(\overline\Omega)$.  For this, we can use the following theorem proved in
the Appendix, \S\ref{T2.3}.

\begin{theorem}
\label{L:cont}
For every $f\in B_p^s(\Omega)$, $s>d/p$, $0<p\leq \infty$,  there is a continuous function   
$\widetilde f$ such that $f=\widetilde f$ a.e. In fact, $\widetilde f\in$ Lip $\alpha$, with $\alpha:=s-d/p$, that is, 
 \be 
\label{fcont}
\omega_r(\widetilde f,t)_{C(\overline\Omega)}\le C|f|_{B_p^s(\Omega)}t^{s-d/p}, \quad t>0.
\ee    
\end{theorem}

 \section{Optimal recovery}
 \label{S:OR}
 In this section, we study the question of
  whether the information $({\bf f},{\bf g})$ at the {\it fixed} data sites 
  $$\cX:=\{\bx_1, \ldots, \bx_{\widetilde m}\}, \quad 
  \cZ:=\{\bz_1,\ldots, \bz_{\overline m}\},
  $$
  is sufficient to determine   $u$   to a prescribed accuracy $\e$  in the $H^1(\Omega)$ norm. Questions of this type are well studied and referred to  as optimal recovery (OR). The answer to our specific  question depends on
 the numbers $\widetilde m,\overline m$, the positions of the points from $\cX$ and  $\cZ$, and the 
 smoothness assumptions we make on $f$ and $g$. 
  Thus, we require that $f,g$
 are in certain Besov spaces  that compactly embed into 
$C(\overline\Omega)$ 
  and  
 $C(\partial \Omega)$, respectively,  and are such that $u$ lies in a compact subset of $H^1(\Omega)$.

     Before presenting the results on optimal recovery in this section, let us remark that the results given below
     are (in essence) all known.   For example, the optimal recovery rates for the Besov model classes of this section have all been given in the 
      papers \cite{DNSI,dahlke2006optimal,vybiral2007sampling} save for one small discrepancy on optimal recovery in $H^{-1}(\Omega)$  (explained below) and for the fact that
      they do not directly consider
     recovery on manifolds like $ \partial\Omega$. 
     The main distinction between the results in \cite{vybiral2007sampling} and ours given below is in how the optimal recovery is obtained/proven. The optimal recovery method in \cite{vybiral2007sampling} uses linear combinations of certain \lq bump functions\rq, whereas we use  continuous piecewise polynomial Lagrange interpolation.
     One can argue that the fact that optimal recovery can be obtained by interpolation is also known since it is a prominent method of approximation in Finite Element Methods (see for instance \cite{ern2021finite}).  However, the latter community usually considers Sobolev classes, rather than the more general Besov classes, and also typically
     restricts $p\geq 1$. 
     In summary, we could not find a direct reference which proves the results given below based on 
     continuous piecewise polynomial interpolation.  In any case, those not familiar
     with optimal recovery may find the exposition below useful and the OR community can just do a very light reading
     of this section.  The statements in this section are all proved in detail
     in the appendix.

 We shall assume that  
 \be 
 \label{condf}
 f\in \cF:=U(\cB),\ \ \cB=B_q^s(L_p(\Omega)), \quad 0<q, p\le\infty, \quad s>d/p,
 \ee 
 where $U(\cB)$ is the unit ball of $\cB$ and
 the restriction on $s$  ensures that $\cF$ compactly embeds  into 
 $C(\overline\Omega)$. 
 
 The function $g$ is defined on the manifold $\partial \Omega$.  There are   two   ways to describe smoothness conditions on such functions $g$. The first  is through the trace operator $Tr:=Tr_{\partial\Omega}$.  The second is to  place smoothness conditions directly on  $g$ as a function from
 $C(\partial\Omega)$.  The latter is usually referred to as an intrinsic definition.    We choose in this paper to define the set $\cG$ via the trace operator.  Namely, we define the model class $\cG$  via the trace of functions from the unit ball of a Besov class
 $B_{\overline q}^{\overline s}(L_{\overline p}(\Omega))$ with 
 $\overline s> d/\overline p$ to ensure continuity.  We also place the restrictions $0<\bar p\le 2$ in order to simplify the presentation that follows.  In other words, our model class assumptions on the boundary function $g$ take the form
 \be 
 \label{condg}
 g\in \cG:= Tr(U(\overline\cB)),\ \ \hbox{where}\ \ \overline\cB:=B_{\overline q}^{\overline s}(L_{\overline p}(\Omega)),\quad \hbox{with}\ \ \bar s>d/\bar p, \quad 0<\bar p\le 2.
 \ee 
Then $\cG$ is a compact subset of both $H^{1/2}(\partial\Omega)$ and $C(\partial \Omega)$. The latter follows because $\overline\cB$ embeds compactly into $C(\overline\Omega)$ since $\bar s>d/\bar p$.  When taking the trace this compactness is preserved. The former embedding follows since the stated conditions also imply that $\bar{s} > 1 + (d/\bar{p} - d/2)$ as $d \geq 2$, which implies that $\overline\cB$ embeds into $H^1(\Omega)$. Taking the trace will then be compact in $H^{1/2}(\partial \Omega)$. Note that $\bar s>1$.
  
 Note that, in general, the parameters $s,p,q$ used in defining $\cF$ are different from $\overline s,\overline p, \overline q$ used in defining $\cG$.  
These assumptions imply that
 the function $u$ we want to numerically approximate is an element of  the set
  \be
 \label{defcU}
 \cU :=\{\widetilde u\in C(\overline \Omega):\ \widetilde u \ {\rm satisfies }\ \eref{main} \ {\rm with} \ f\in \cF,\ g\in \cG\}.
 \ee 

 Given data $
 {\bf f}:=(f_1, \ldots,f_{\widetilde m})$, the totality of information we have about $f$ is that it belongs to the set
 $$
 \cF_{\rm data}:=  \cF_{{\rm data}(f)}:=\{\widetilde f\in\cF: \ \widetilde f_i=f_i,\ i=1,\dots, \widetilde m \}.
 $$
 Similarly, if the data $
 {\bf g}:=(g_1, \ldots,g_{\overline{m}})$
  is coming from  a $g\in \cG$, then, 
 what we know about  $g$ is that in lies in the set
 \be 
 \label{cGdata}
 \cG_{\rm data}:=\cG_{{\rm data}(g)}:= \{\widetilde g\in\cG: \ \widetilde g_i=g_i,\ i=1,\dots,\overline m\}.
 \ee 
 Notice that    $g_i=u(\bz_i)$, $i=1,\dots,\overline m$.
Finally, the totality of information we have about the sought after $u$ is that it is in the set
 $$
 \cU_{\rm data}:={\cU_{{\rm data}(u)}:=} \{\widetilde u\in\cU{:\,-\Delta \widetilde u (\bx_i)=f_i,\ i=1,\dots,\widetilde m,\quad \widetilde u(\bz_i)=g_i,\ \ i=1,\dots,\overline m}\},
 $$
or equivalently
  $$
  \cU_{\rm data}:=\cU_{{\rm data}(u)}:=\{\widetilde u: \widetilde u \ \hbox{is a solution to \eref{main} with } \, f\in \cF_{\rm data}, \quad  g\in \cG_{\rm data}\}.
  $$
    We are interested in knowing to what extent the information that $u\in\cU_{\rm data}$ 
  identifies $u$. 

  There is a simple theoretical answer to such questions.   Given a compact set $K$ of a Banach space $X$, we let $B(K):=B(K)_X$
  denote a ball with the smallest radius  in $X$ which contains $K$, called a {\it Chebyshev ball} of $K$ in $X$. If we wish to provide
  an element from $X$ that simultaneously approximates all  elements in $K$ (in the norm of $X$), then the center of this ball is the best we can do and its radius 
  \be 
  \label{ORrate}
  R(K)_X:= {\rm rad}(B(K))_X
  \ee 
  is the {\it optimal error} we can obtain. It is called {\it the error of optimal recovery} of $K$. 

    If  $\cU_{\rm data}$ is a compact subset of $X$, let $B=B(\cU_{\rm data})_X$ be  a Chebyshev ball
  of this set.   Since all we know about $u$  is that it is in  $\cU_{\rm data}$, 
 any element from $B(\cU_{\rm data})_X$ will give a near best approximation to $u$ in $\|\cdot\|_X$.
The radius of $B$ 
  \be 
  \label{OReffor}
  R(\cU_{{\rm data}(u)})_X:= {\rm rad}(B(\cU_{{\rm data}(u)}))_X
  \ee 
is then  the optimal recovery error. We are interested in bounds for $R(\cU_{{\rm data}(u)})_X$, where $\|\cdot\|_X$ is the norm in which we wish to measure accuracy. We introduce similar notation for the recovery of $f$ and $g$.

  The recovery rate \eref{OReffor} will not only depend on the data sites $\cX$, and $\cZ$,  but also on the values $({\bf f}, {\bf g})$
  assigned at these data sites.  In order to obtain
  uniform estimates, we fix the data sites $(\cX,\cZ)$ and introduce
  \be 
  \label{uOR}
  R^*(\cU,\cX,\cZ)_X:= \sup_{u\in \cU} R(\cU_{{\rm data}(u)})_X,
  \ee
  where the data values come from any $u\in\cU$.  This is measuring the worst possible performance over the class $\cU$ and is called the {\it uniform optimal recovery rate} at the fixed data sites 
  $(\cX,\cZ)$.
 If we prescribe a budget for the number $m=|\cX|+|\cZ|=\widetilde m+\overline m$ of data
  sites, we can ask for the optimal error we can achieve under such a budget restriction. Accordingly, we define
 $$
 R_m^*(\cU)_X:=\inf_{
   \cX\subset\overline\Omega, \,\cZ\subset \partial \Omega,\, |\cX|+|\cZ|=m} R^*(\cU,\cX,\cZ)_X,\quad m\geq 2.
 $$
  Similarly, 
 $$
  R^*(\cF,\cX)_X:= \sup_{f\in \cF} R(\cF_{{\rm data}(f)})_X \quad  \hbox{and}\quad R_{\widetilde m}^*(\cF)_X:= \inf_{
   \cX\subset\overline\Omega, |\cX|=\widetilde m} R^*(\cF,\cX)_X,\quad \widetilde m\geq 1,
$$
 where the latter is  the minimum error we can achieve over all the possible choices of data sites $ \cX=\{\bx_1,\dots,\bx_{\widetilde m}\}$.  
    We can similarly ask the question of how well can we  can  recover
  $g$ from the data ${\bf g}$.  Each of these questions has an associated uniform optimal recovery error which is dependent on
  which norm $\|\cdot\|_X$   we use to measure the error of recovery.   

  \subsection{Optimal recovery of $f$}
  \label{SS:Of}

  Let $\cF:=U(\cB)$ where  $\cB=B_q^s(L_p(\Omega))$ with $0\le q,p\le \infty$ and $s>\frac{d}{p}$. 
Recall, see \eref{Gd},  the tensor product grid $G_{k,r}$ of $\overline\Omega=[0,1]^d$ for 
 $r\geq 2$,
   with spacing $h=2^{-k}(r-1)^{-1}$ (the role of $r$ will be made clear below).
We will prove that, see Theorem \ref{T:ORf}, 
 $$
R^*(\cF,G_{k,r})_X\lesssim \widetilde m^{-\alpha_X}, \quad \widetilde m=(r2^k)^d.
  $$ 
 Moreover, we show that  for any choice of data sites $\cX\subset \overline\Omega$ with 
  $|\cX|=\widetilde m$, we have
  $$
  R^*(\cF,\cX)_X\gtrsim \widetilde m^{-\alpha_X},
  $$
  and therefore 
  $R_{\widetilde m}^*(\cF)_X\asymp \widetilde m^{-\alpha_X}$.  
The proof of this theorem is presented in the Appendix, \S\ref{T3.1}.

  \begin{theorem}
      \label{T:ORf}
      Let $\Omega=(0,1)^d$ and  $\cF:=U(\cB)$, where $\cB$ is the Besov space $B_q^s(L_p(\Omega))$ with $s>d/p$, and  $0<p,q\leq \infty$.
      Then the following holds for the optimal recovery rate   $R_{\widetilde m}^*(\cF)_X$ in the norm of $X$:
      \vskip .1in
      \noindent 
      {\rm (i)} If $X=C(\Omega)$, then
      \be 
      \label{TF1}
      R_{\widetilde m}^*(\cF)_X \asymp {\widetilde m}^{-\alpha_C},\ \widetilde m\ge 1,\quad {\rm where} \quad \alpha_C:= \frac{s}{d}-\frac{1}{p},
      \ee 
      with constants of equivalence  independent of $\widetilde m$.

      \vskip .1in
      \noindent 
      {\rm (ii)} If $X=L_\tau(\Omega)$, $\tau>0$, then
      \be 
      \label{TF2}
      R_{\widetilde m}^*(\cF)_X\asymp \widetilde m^{-\alpha_\tau},\ \widetilde m\ge1, \quad {\rm where}\quad \alpha_\tau:= \frac{s}{d}-\left [\frac{1}{p}-\frac{1}{\tau}\right ]_+,
      \ee 
      with constants of equivalence independent of $\widetilde m$.

      \vskip .1in
      \noindent 
      {\rm (iii)} If $X=H^{1}(\Omega)$,  $0<p\le 2$, then 
      \be 
      \label{TF3}
      R_{\widetilde m}^*(\cF)_X\asymp \widetilde m^{-\alpha_H}, \ \widetilde m\ge 1,\quad {\rm where}\quad \alpha_H:= \frac{s-1}{d}-\left[\frac{1}{p}-\frac{1}{2}\right] ,
      \ee 
      with constants of equivalence independent of $\widetilde m$.

      \vskip .1in
      \noindent 
      {\rm (iv)} If $X=H^{-1}(\Omega)$, then 
      \begin{itemize}
          \item if $d \geq 3$  and $0<p\leq \infty$, or $d = 2$ and $1<p\leq \infty$, we have
      \be 
      \label{TF4}
      R_{\widetilde m}^*(\cF)_X\asymp \widetilde m^{-  \alpha_{-1}},\ \widetilde m\ge 1 \quad {\rm where} \quad \alpha_{-1}:= \frac{s}{d}-\left[\frac{1}{p}-\frac{1}{\delta}\right]_+ \ {\rm and} \ \frac{1}{\delta}:=\frac{1}{2}+\frac{1}{d},
      \ee 
      with constants of equivalence independent of $\widetilde m$. 
      \item if $d = 2$ and $0<p \leq 1$, we have
    \be 
      \label{TF3-d2}
       \widetilde m^{-  \alpha_{-1}}\lesssim R_{\widetilde m}^*(\cF)_X\lesssim \log(\widetilde m)\widetilde m^{-  \alpha_{-1}},\qquad  
       \qquad \widetilde m\ge 1,
      \ee       
      with $\alpha_{-1}$ and $\delta$ as in \eref{TF4}.
    \end{itemize}  
      Moreover, when $\widetilde m=(r2^k)^d$,   the upper bounds in each of these estimates are obtained when we take the data sites   $\cX=G_{k,r}$, provided $r>\max(s,1)$. 
      In this case, the approximant to $f$ is given by the continuous extension of the  piecewise polynomial 
$$
S_k^*(f)=\sum_{T\in\cT_{k}}L_T(f)\chi_T,
$$
described in \eref{pwpinter1}, where $\cT_{k}=\cT_{k}(\Omega)=\bigcup T$ 
 is the collection of all simplicies arising from the (Kuhn-Tucker)
decomposition of the
dyadic cubes in $\cD_k(\Omega)$, 
the $\chi_T$'s are the characteristic functions of $T$, and  the  $L_T(f)$'s are the polynomials of order $r$ ($L_T(f)\in\cP^d_r$) gotten by interpolating $f$ at    the data points  $\overline T\cap G_{k,r}$.
\end{theorem}

We remark that in the case $d = 2$ and $0<p\le 1$, the logarithm appears because of the failure  of the Sobolev embedding to provide $L_1(\Omega)\subset H^{-1}(\Omega)$.  However, it is actually true that the Hardy space $\cH^1(\Omega)$ embeds into $H^{-1}(\Omega)$. This could potentially be used to improve our later analysis, but for the sake of simplicity we do not pursue this here.

We make some further observations that will help explain the above
theorem, especially what is happening for recovery in $H^{-1}(\Omega)$.  Notice
that under the model class assumption $f\in \cF=U(B_q^s(L_p(\Omega)))$ with $s>d/p$,
the OR for this class is the same whether we measure error in $H^{-1}(\Omega)$
or in $L_{\delta}(\Omega)$ 
when $d\geq 3$ or $d=2$, $1<p\leq \infty$.  We also have that
\be 
\label{H-1embed}
\|f\|_{H^{-1}(\Omega)}\le 
\begin{cases}
  C(d)\|f\|_{L_{\delta}(\Omega)},
  \quad d\geq 3,\quad \hbox{where}\,\,\delta=\frac{2d}{d+2},\\ \\
  \frac{C}{\tau-1}\|f\|_{L_\tau(\Omega)}, \quad \,\, d=2, \quad \tau>1,
\end{cases}
\ee 
with $C$ an absolute constant, as  discussed in Lemma \ref{L:Lqembed} of the appendix .  These two facts explain the form of the loss function $\cL^*$ which we  discuss in detail later.

 \subsection{Optimal recovery of $g$}
\label{SS:ORg}

In this section, we study the optimal recovery rate of   functions $g\in \cG$ in the norm of $H^{1/2}(\partial\Omega)$.  Let $Tr=Tr_{\partial \Omega}$ be the trace operator onto $\partial\Omega$ and let $\overline \cB:=B_{\bar q}^{\bar s}(L_{\bar p}(\Omega))$, where we fix $\bar s,\bar p,\bar q$ and assume that $ \overline s> d/\overline p$ so that we are sure that
the functions in $\overline \cB$ are continuous on $\overline\Omega$.  We consider the case $0<\bar p\le 2$,
 which together with the restriction $d\ge 2$ implies that $\bar s-d(1/\bar p-1/2)>1$, and thus the unit ball of $\overline \cB$ compactly embeds  into $H^1(\Omega)$.

 We define the model class 
\be 
\label{defcG}
\cG:=\{g:\,g=Tr(v):\ \|v\|_{\overline \cB}\le 1\},
\ee 
and use the trace norm definition of
$\|\cdot\|_{H^{1/2}(\partial\Omega)}$ throughout this section.

Recall the tensor product
grid $G_{k,r}$ for $\overline\Omega=[0,1]^d$, see \eref{Gd}.  We let 
\be\label{Gb}
\overline G_{k,r}:=\{\bz_i:\ i=1,\dots,\overline{m}\}=G_{k,r}\cap \partial \Omega
\ee
denote the set of those grid points of $G_{k,r}$ that are on 
the boundary $\partial\Omega$.  We will use the points 
$\cZ=\overline G_{k,r}$ to recover $g$.  Note that the number $ \overline m$ of data sites in $\overline G_{k,r}$ is
\be 
 \overline{m}=\#(\overline  G_{k,r})  =2d[r2^k]^{d-1}\asymp 2^{k(d-1)},
\ee 
with constants of equivalency depending only on $r$ and $d$.

Given $g\in \cG$, 
we define   the continuous piecewise polynomial
\be 
\label{defSkbar}
\overline S_k(g):= \sum_{i=1}^{\overline{m}} g(\bz_i) \overline\phi_i,
\ee
where each $\overline\phi_i$ is the trace  of the Lagrange element 
$\phi_i$ centered at $\bz_i\in \overline G_{k,r}$.
Note that if  $v\in \overline \cB$ is any function whose trace on $\partial\Omega$ is $g$, and 
 $S_k^*(v)$
is its  continuous piecewise polynomial Lagrange interpolant at the grid points $G_{k,r}$, see \eref{pwpinter1},
then we have 
\be
\label{need}
Tr(S_k^*(v))=\overline S_k(g).
\ee
This follows from the fact that  each Lagrange interpolant 
$\phi_i$ centered at $\bx_i\in G_{k,r}\setminus \overline G_{k,r}$ vanishes at the faces of the simplex from the corresponding Kuhn-Tucker decomposition   that contains $\bx_i$. Note that $\overline S_k(g)$ does not depend on $v$.   According to Theorem \ref{T:SkH1}, $S_k^*(v)\in H^1(\Omega)$ 
since $v$ belongs to unit ball of $\overline \cB$. Then, \eref{need}  gives that  $S_k(g)\in H^{1/2}(\partial \Omega)$.

The proof of the following result is provided in the Appendix, \S\ref{T3.2}.

\begin{theorem}
\label{T:ORg}
Let $\overline \cB=B_{\overline q}^{\overline s}(L_{\overline p}(\Omega))$ with $\overline s> d/\bar p$, $0<\overline p\le 2$, $0<\overline q\le \infty$, and  let $\cG=\{Tr(v):\|v\|_{\overline\cB}\le 1\}$.   Then  the optimal recovery rates of the model class $\cG$ in  $H^{1/2}(\partial\Omega)$ is 
\be 
\label{ORGrate}
R_{\overline m}^*(\cG)_{H^{1/2}(\partial\Omega)}\asymp \overline{m}^{- \beta}, \quad \overline{m}\ge 1, \quad \hbox{where}\quad 
\beta:=
\frac{\overline s-1}{d-1}-\frac{d}{d-1}\left [\frac{1}{\bar p}-\frac{1}{2}\right],
\ee
with constants of equivalency independent of $\overline m$.  
Moreover,  when $\overline m=\#(\overline G_{k,r})$, the upper bound in \eref{ORGrate} is achieved  when we take the data sites  
$\cZ=\overline G_{k,r}$, provided $r>\max(s,1)$  and the approximant to $g$ is given by
 the function $\overline S_k(g)$.
\end{theorem}

 To understand the exponent $\beta$ in \eref{ORGrate} better, let $v$ be a function in $ U(B_q^{\bar s}(L_{\bar p}(\Omega)))$ whose trace is $g$, and rewrite  $\beta$ as
 $$
 \beta=\frac{\overline s-1-(\frac{d}{\bar p}-\frac{d}{2})}{d-1}.
 $$
Here the numerator is the excess regularity of $v$ in $H^1(\Omega)$ and thereby corresponds to the excess regularity of $g$ in $H^{1/2}(\partial\Omega)$, and the denominator is $d-1$ since 
 $\partial\Omega$ has dimension $d-1$.

\subsection{Optimal recovery of $u$}
\label{SS:Ou}
 We turn now to determining the optimal $H^1(\Omega)$ recovery rate for  the functions $u$ in the  model class $\cU$ from data
 \be\label{dataU}
 -\Delta u(\bx_i)=f_i, \quad \bx_i\in\cX,\ i=1,\dots, \widetilde m; \quad u(\bz_i)= g_i, \quad \bz_i\in \cZ,\ i=1,\dots \overline m, 
 \ee
 under a given total budget $m=\widetilde m+\overline m$ of data observations.
Recall that $\cF:=U(B_q^s(L_p(\Omega)))$,  where $s>d/p$
 and $0<q,p\le\infty$,  $\cG:= Tr(U(B_{\overline q}^{\overline s}(L_{\overline p}(\Omega)))$, where $\overline s> \frac{d}{\bar p}\geq 1$ since $d\geq 2$ and $0<\overline p\le 2$, 
 $0<\overline q\leq \infty$, and
 \be 
\nonumber
 \cU:= \{u \in C(\overline \Omega)\ : \  u \ {\rm satisfies} \ \eref{main}\ {\rm with} \ f\in\cF,g\in\cG\}.
 \ee 
    The following theorem
 gives the optimal recovery rate for $\cU$.

 \begin{theorem}
     \label{T:ORu}
Let $\Omega=(0,1)^d$  and let $\cF$, $\cG$,  $\cU$ be as above.  Then the following holds for the optimal recovery rate $R^*_m(\cU)_{H^1(\Omega)}$:
  \begin{itemize}
          \item if $d \geq 3$, or $d = 2$ and $p > 1$, we have
      \be
\label{TORu}
R^*_m(\cU)_{H^1(\Omega)}\asymp m^{-\min(\alpha,\beta)}, \quad m\ge 2,
\ee 
      \item if $d = 2$ and $0<p\leq 1$, we have
    \be 
      \label{TORu-d2}
 m^{-\min(\alpha,\beta)}\lesssim R^*_m(\cU)_{H^1(\Omega)}\lesssim \log(m)m^{-  \alpha}+ m^{-  \beta},\quad \ m\ge 2,
      \ee       
    \end{itemize}  
where 
\be\label{ab}
\alpha= \frac{s}{d}-\left[\frac{1}{p}-
\left (\frac{1}{2}+\frac{1}{d}\right)\right]_+ ,  \quad  \beta=\frac{\overline s-1-(\frac{d}{\bar p}-\frac{d}{2})}{d-1},
\ee 
and all constants of equivalence are independent of $m$.
\end{theorem}
\begin{proof}
    Let us fix the data cites $\cX,\cZ$ with $|\cX|=\widetilde m$,
 $|\cZ|=\overline m$, $m=\overline m+\widetilde m$, and let $u_1,u_2\in \cU_{\rm data}$ with corresponding $f_1,f_2\in \cF_{\rm data}$ and $g_1,g_2\in \cG_{\rm data}$. It follows from \eref{perturb} that 
 $$
 \|u_1-u_2\|_{H^1(\Omega)}\asymp 
 \|f_1-f_2\|_{H^{-1}(\Omega)}+\|g_1-g_2\|
_{H^{1/2}(\Omega)},
 $$
 and therefore
\be\label{eq}
 R^*(\cU,\cX,\cZ)_{H^1(\Omega)}\asymp R^*(\cF,\cX)_{H^{-1}(\Omega)}+
 R^*(\cG,\cZ)_{H^{1/2}(\Omega)}.
\ee

 Clearly, if $\overline m=\widetilde m\asymp m/2$ and using  the rates from 
 Theorem \ref{T:ORf} and Theorem \ref{T:ORg}, we have that 
 $$
 R^*_m(\cU)_{H^1(\Omega)}\lesssim 
 \begin{cases}
     m^{-\alpha}+m^{-\beta}\lesssim m^{-\min(\alpha,\beta)},\quad d\geq 3, \,\hbox{or}\,\,d=2,\,p>1,\\ \\
\log(m)m^{-\alpha}+m^{-\beta}, \,\,\quad\quad\quad \quad d=2,\,0<p\leq 1.
     \end{cases}
 $$
 On the other hand, using the same theorems and \eref{eq}, we obtain
\begin{eqnarray*}
 m^{-\min(\alpha,\beta)}\lesssim\inf_{m=\overline m+\widetilde m}(\widetilde m^{-\alpha}+\overline m^{-\beta})&\lesssim& \inf_{m=\overline m+\widetilde m}\left (R_{\widetilde m}^*(\cF)_{H^{-1}(\Omega)}+
 R_{\overline m}^*(\cG)_{H^{1/2}(\Omega)} \right )\\
 &\lesssim &
 \inf_{m=|\cX|+|\cZ|}
 R^*(\cF,\cX)_{H^{-1}(\Omega)}+
 R^*(\cG,\cZ)_{H^{1/2}(\Omega)}
 \\
 &\lesssim &\inf_{m=|\cX|+|\cZ|} R^*(\cU,\cX,\cZ)_{H^1(\Omega)}=R^*_m(\cU)_{H^1(\Omega)},
\end{eqnarray*}
 and the proof is completed.
 \end{proof}

 \begin{remark}
 \label{RR}
Consider  $\Omega=(0,1)^d$  and  $\cF$, $\cG$,  $\cU$  as above.  
Let us assume that we have a fixed budget of $\widetilde m$ data sites $\cX$ from $\overline \Omega$ and a fixed number $\overline m$ of data sites $\cZ$ on the boundary $\partial \Omega$ and  denote by 
 $$
R_{\widetilde m,\overline m}^*(\cU)_{H^1(\Omega)}:=\inf_{
   \cX\subset\overline\Omega, \,\cZ\subset \partial \Omega,\, |\cX|=\widetilde m, \,|\cZ|=\overline m} R^*(\cU,\cX,\cZ)_{H^1(\Omega)}.
   $$  
 It follows from \eref{eq} that 
 \be
 \label{that}
R_{\widetilde m}^*(\cF)_{H^{-1}(\Omega)}+
R^*_{\overline m}(\cG)_{H^{1/2}(\partial \Omega)}
\asymp
R_{\widetilde m,\overline m}^*(\cU)_{H^1(\Omega)}.
\ee
Theorem \ref{T:ORf} and Theorem \ref{T:ORg} then yield
\be
\label{mainco}
\widetilde m^{-\alpha}+\overline m^{-\beta}
\lesssim R_{\widetilde m,\overline m}^*(\cU)_{H^1(\Omega)}\lesssim
\begin{cases}
\log(\widetilde m)\widetilde m^{-\alpha}+\overline m^{-\beta}, \quad d=2, \,0<p\leq 1,\\
\widetilde m^{-\alpha}+\overline m^{-\beta}, \quad \quad\quad\quad d=2, \,1<p\leq \infty\quad \hbox{or}\quad d\geq 3,
    \end{cases}
\ee
where $\alpha,\beta$ are as in \eref{ab}.

Now suppose instead that we are given a fixed total budget $m$ and $d\geq 2$. Then we can obtain the  optimal rate $R^*_m(\cU)_{H^1(\Omega)}$  by choosing $\widetilde m\asymp \overline m\asymp m/2$. 
\end{remark}

 \subsection{Final observations on optimal recovery}
 \label{SS:finalobservation}
 This section gave the optimal recovery rates for the model classes $\cF,\cG,\cU$ from 
 $\widetilde m,\overline m,m$, data observations, respectively.  While we began by considering 
  general model class assumptions $\cF$ and $\cG$ for $f$ and
 $g$ to be the unit balls of Besov spaces, we found
 that the optimal recovery rate is the same for many of these model classes.  This means that some of the model classes are superflous in that they are contained in a larger model class with the same recovery rate.   It makes sense only to use these largest model classes in going forward.

   Consider, for example,
 the model classes $\cF=U(\cB)$ with $\cB=B_q^s(L_p(\Omega))$ for $f$.  We found that 
  in the case $d\geq 3$ and $0<p\leq \infty$,  or $d=2$ and $1<p\leq \infty$, the optimal recovery rate of $\cF$ in $H^{-1}(\Omega)$
 is the same as its optimal recovery rate in $L_{\delta}(\Omega)$ with ${\delta}=\frac{2d}{d+2}$.  Moreover,   the model class $\cF$  has an optimal recovery rate $m^{-\alpha/d}$ in $H^{-1}(\Omega)$  if and only if it is contained in $U(B_\infty^{\alpha}(L_\delta(\Omega)))$ and, in addition, $\cF$  embeds into $C(\overline\Omega)$.

 \vskip .1in
 \noindent
 {\bf Largest Model Classes for $f$:} 
 
In the case 
$d\geq 3$, $0<p\leq \infty$ or $d=2$, 
$1<p\leq \infty$ the class
 \be 
 \label{onlyF}
 \cF:= U(\cB),\quad \cB=B_\infty^s(L_p(\Omega)), \ {\rm with} \  p\ge {\delta}  \ {\rm and}\  s> \frac{d}{p}
 \ee 
 has the optimal recovery rate 
 \be
 \label{ORspecialF}
 R_{m}^*(\cF)_{H^{-1}(\Omega)}\asymp { m}^{-s/d}, \quad  m\ge 1,
 \ee
 and any Besov model class that gives the optimal recovery rate $O(m^{-s/d})$ is contained in one of these  model classes.  In each of these cases, the number $s$ represents the   smoothness of $f\in\cF$ in $L_{\delta}(\Omega)$.  These largest model classes are pictured in Figure \ref{F:devore-diagram}.

In the case $d=2$, $0<p\leq 1$, the largest class
 \be 
 \label{onlyF1}
 \cF:= U(\cB),\quad \cB=B_\infty^s(L_1(\Omega)), \ {\rm with} \    s> 2,
 \ee 
 has  recovery rate 
 \be
 \label{ORspecialF1}
 { m}^{-s/2}\lesssim R_{m}^*(\cF)_{H^{-1}(\Omega)}\lesssim (1+\log(m)){ m}^{-s/2}, \quad  m\ge 1.
 \ee
 \vskip .1in
 \noindent 
 {\bf Largest  Model Classes for $g$:} We have considered the  model classes $\cG=Tr(B^{\bar s}_\infty(L_{\bar p}(\Omega))$ with $0<\bar p\leq 2$, $d\geq 2$, and $\bar s> \max \left(\frac{d}{\bar p},1\right)$.  Thus, the model classes
 which give a  rate $O(\overline  m^{-\alpha})$ will all be contained in one model class
 \be 
 \label{largestG}
 \cG:=Tr(U(\overline \cB)), \quad \overline\cB=B_\infty^{\bar s}(L_{2}(\Omega)),   \quad   \ {\rm with} \ \bar s> d/2.    
 \ee
 This class has optimal recovery rate  in $H^{1/2}(\partial\Omega)$ 
 \be
 \label{ORspecialG}
 R_{m}^*(\cG)_{H^{1/2}( \partial \Omega)}\asymp m^{-(\bar s-1)/(d-1)}, \quad m\ge 1.
 \ee
Moreover, any Besov model class that gives this recovery rate is contained in one of
these largest model classes.  For each of these largest model classes $\bar s -1$ is the excess smoothness of $v$ in $H^1(\Omega)$ and also the excess smoothness of $g=Tr(v)$ in $H^{1/2}(\Omega)$.  
 \begin{remark}
     \label{R:largest}
Because of the above remarks, in going further, we always take 
      $\cF$ as in \eref{onlyF} and $\cG$  as in \eref{largestG}.
 \end{remark}

\begin{figure}[ht!]
\begin{center}
\begin{tikzpicture}[scale =0.9]


\draw [thick,->]plot coordinates{(-0.1,2) (5.5,2)};
\draw [thick,->]plot coordinates{(0,1.9) (0,6.3)};

\draw  plot coordinates{(2,4)(4,6)};
\draw [very thick,dashed] plot coordinates{(0,2) (2,4)};
\draw [very thick] plot coordinates{(2,4) (2,6)};
\draw [very thick, loosely dotted] plot coordinates{(2,2) (2,6)};

\draw [->]plot coordinates{(0,2.6) (0.4,2.6)};
\draw [->]plot coordinates{(0,3.1) (0.9,3.1)};
\draw [->]plot coordinates{(0,3.6) (1.4,3.6)};
\draw [->]plot coordinates{(0,4.2) (1.9,4.2)};
\draw [->]plot coordinates{(0,4.7) (1.9,4.7)};
\draw [->]plot coordinates{(0,5.2) (1.9,5.2)};
\draw [->]plot coordinates{(0,5.7) (1.9,5.7)};

\draw [->]plot coordinates{(3.9,6) (2.1,4.2)};
\draw [->]plot coordinates{(3.4,6) (2.1,4.7)};
\draw [->]plot coordinates{(2.9,6) (2.1,5.2)};
\draw [->]plot coordinates{(2.4,6) (2.1,5.7)};

\coordinate [label= {s}] (z1) at (-0.5,6);
\coordinate [label= $1/\delta$] (z1) at (2,1.2);
\coordinate [label= $1/p$] (z1) at (3+2,1.2);
\coordinate [label= $0$] (z1) at (-0.25,1.75);
\coordinate [label= $s{=}d/p$] (z1) at (4.8,5.8);

\draw[fill=white](2,4) circle (0.1);

\draw [thick,->]plot coordinates{(-0.1+8,2) (5.5+8,2)};
\draw [thick,->]plot coordinates{(0+8,1.9) (0+8,6.3)};

\draw  plot coordinates{(2+8,4)(4+8,6)};
\draw [very thick] plot coordinates{(2+8,4) (2+8,6)};

\draw [->]plot coordinates{(3.9+8,6) (2.1+8,4.2)};
\draw [->]plot coordinates{(3.4+8,6) (2.1+8,4.7)};
\draw [->]plot coordinates{(2.9+8,6) (2.1+8,5.2)};
\draw [->]plot coordinates{(2.4+8,6) (2.1+8,5.7)};

\draw [very thick, loosely dotted] plot coordinates{(2+8,2) (2+8,4)};
\draw [very thick, loosely dotted] plot coordinates{(0+8,4) (2+8,4)};

\coordinate [label= {s}] (z1) at (-0.5+8,6);
\coordinate [label= $1/2$] (z1) at (2+8,1.2);
\coordinate [label= $1/p$] (z1) at (3+2+8,1.2);
\coordinate [label= $0$] (z1) at (-0.25+8,1.75);
\coordinate [label= $\frac d 2$] (z1) at (-0.35+8,3.6);
\coordinate [label= $s{=}d/p$] (z1) at (4.8+8,5.8);

\draw[fill=white](2+8,4) circle (0.1);

\end{tikzpicture}
\caption{The spaces $B^{s}_\infty(L_p(\Omega))$ are represented by the point $(1/p,s)$ in the first quadrant. The arrows indicate embedding with a  slope of $d$ for the oblique lines. In particular, the spaces above the line $s=d/p$ embed into $C(\overline{\Omega})$. (Left) The largest model classes for $f$ are the unit balls of spaces just above the thick and dashed segment or on the on the thick vertical half-line (excluding the point $(1/{\delta},d/{\delta})$). (Right) The largest model classes for $g$ are the traces of the unit balls of spaces on the thick vertical half-line (excluding the point $(1/2,d/2)$).}\label{F:devore-diagram}
\end{center}
\end{figure}

 \vskip .1in
 \noindent 
 {\bf  Model Classes for $u$:}
If we use one of the model classes 
 from \eref{onlyF} for $\cF$ and one of the 
model classes from \eref{largestG}  for $\cG$, we obtain a model class $\cU$ for $u$. The optimal
recovery rate for this model class is
$$
R_m^*(\cU)_{H^1(\Omega)}\asymp m^{-\min\{s/d,(\bar s-1)/(d-1)\}},\quad m\ge 2,
$$
which could be  obtained by assigning $\widetilde  m\asymp m/2$ points in $\overline \Omega$ and  $\overline m\asymp m/2$ points  on the boundary $\partial \Omega$.

\section{Numerical algorithms  based on linear approximation}
\label{S:NOR2}

Recall that we consider the case when  $\Omega=(0,1)^d$, $d\geq 2$,  $\cF:=U(\cB)$, where  $\cB=B_q^s(L_p(\Omega))$ with $0\le q,p\le \infty$ and $s>\frac{d}{p}$, and 
$\cG=\{Tr(v):\|v\|_{\overline\cB}\le 1\}$, where
$\overline \cB=B_{\overline q}^{\overline s}(L_{\overline p}(\Omega))$ with $\overline s> d/\bar p$, $0<\overline p\le 2$, $0<\overline q\le \infty$. For these right hand sides and boundary values, we consider the solution set  
\be 
\nonumber
 \cU:= \{u \in C(\overline \Omega)\ : \  u \ {\rm satisfies} \ \eref{main}\ {\rm with} \ f\in\cF,g\in\cG\}.
 \ee

The above analysis of optimal recovery tells us to what extent the information $({\bf f}, {\bf g})$ determines $u$. That is, the optimal recovery analysis  tells us that if two functions $u,u'\in \cU$ and  satisfy the same data, namely $u,u'\in \cU_{\rm data}$, then they are close in the $H^1(\Omega)$ norm.
More precisely, if $u^*$ is the center of a  Chebyshev ball $B(\cU_{\rm data})$, 
then
\be
\label{not}
\|u-u'\|_{H^1(\Omega)}\leq \|u-u^*\|_{H^1(\Omega)}+\|u^*-u'\|_{H^1(\Omega)}\leq 2R_m^*(\cU)_{H^1(\Omega)}.
\ee
However, our analysis does not give a numerical algorithm that takes the data and creates an approximation $\widehat u$ to $u$  with accuracy of the optimal recovery rate.  In this section, we discuss two numerical methods which would accomplish the latter task.  The methods discussed in this section are not in the form of PINNs. The latter will be discussed in
\S \ref{S:implications}.

We assume in this section that the data sites form tensor product grids, that is,
 $\cX=G_{k,r}$ and 
$\cZ=\overline G_{\bar k,r}:=G_{\bar k,r}\cap \partial \Omega$.  We know such data sites
provide a near optimal recovery.
We have shown that if we use the data ${\bf f}$ at 
$G_{k,r}$, then  we can numerically construct  a continuous piecewise polynomial interpolant $\widehat f$  to $f$ at the data sites which is a near optimal (in terms of the budget 
$\widetilde m=(r2^k)^d$)
approximation to $f$ in the $H^{-1}(\Omega)$ norm.  Similarly, we can numerically construct
a $\widehat g$ from the data  ${\bf g}$ which is a near optimal recovery of $g$ in the $H^{1/2}(\partial\Omega)$ norm
in terms of the budget $\overline{m}\asymp 2^{\bar k(d-1)}$.
The following lemma holds. 
\begin {lemma}
\label{Lgeneral}
Let $\widetilde u$ be
 the solution  to \eref{main} with right hand side
$\widehat f:=S^*_k(f)$ and boundary value 
$\widehat g:=\overline S_k(g)$. Then 
\be  
\label{utilde}
\|u-\widetilde u\|_{H^1(\Omega)}\lesssim
R_{\widetilde m,\overline m}^*(\cU)_{H^1(\Omega)}, \quad \hbox{where}\quad \widetilde m\asymp 2^{kd},
\quad \overline m\asymp 2^{\bar k(d-1)}.
\ee
If instead, we are given a fixed total budget $m$, then 
\be  
\label{utilde1}
\|u-\widetilde u\|_{H^1(\Omega)}\lesssim
R_{m}^*(\cU)_{H^1(\Omega)}, \quad 
\hbox{provided}\quad \widetilde m\asymp \overline m\asymp m/2.
\ee
\end{lemma}
\begin{proof}
We have that  $\widehat f:=S^*_k(f)\in C(\overline \Omega)\cap H^1(\Omega)$, see Theorem \ref{T:ORf},  and $\widehat g:=\overline S_k(g)\in H^{1/2}(\partial \Omega)\cap C(\partial \Omega)$, see Theorem \ref{T:ORg}.
Note that we cannot use \eref{not} to bound $\|u-\tilde u\|_{H^1(\Omega)}$ since $\hat f\notin\cF$ and $\hat g\notin\cG$ (they do not have the necessary smoothness) and therefore $\tilde u\notin\cU$. In this case, in view of \eref{equiv} with $v=\widetilde u$, we have
$$
\|u-\widetilde u\|_{H^1(\Omega)}\lesssim
\|f-\hat f\|_{H^{-1}(\Omega)}+\|g-\hat g\|_{H^{1/2}(\partial \Omega)}\lesssim 
R_{\widetilde m}^*(\cF)_{H^{-1}(\Omega)}+
R^*_{\overline m}(\cG)_{H^{1/2}(\partial \Omega)}
\lesssim
R_{\widetilde m,\overline m}^*(\cU)_{H^1(\Omega)},
$$
where we have used \eref{that}. This proves \eref{utilde}.
  In view of Remark~\ref{RR},  given $m$ data sites, the solution $\widetilde u$ that corresponds to the choice of $\widetilde m\asymp m/2$ points on $\overline \Omega$ and $\overline m\asymp m/2$ points on $\partial \Omega$ satisfies \eref{utilde1}
  and thus   is a near optimal recovery of $u$ in the $H^1(\Omega)$ norm.
    \end{proof}


Note that  the constructions of $\widehat f$ and $\widehat g$ are numerical and we can control their complexity. In order to do that for $\widetilde u$, we  use an  existing numerical method for solving
\eref{main}. This will, of course, incur a numerical error depending on the algorithm we choose. We mention two natural possibilities.

\subsection{Using FEMs or AFEMs}

  In order to streamline the discussion using FEM, we make the simplifying  but unnecessary assumption that the data tensor product grids $G_{k,r}$ and $\overline{G}_{\overline k,r}$ are matching, i.e. $k=\overline k$. For $n\geq k$, we let $\mathcal T_n$ be the Kuhn-Tucker subdivision of $\Omega$ associated with the tensor grid $G_{n,r}$ and $V_n$ be the spaces of continuous piecewise polynomials of degree $r$ subordinate to $\mathcal T_n$. We use Lemma \ref{Lgeneral} to derive the following statement about FEMs.
  
  \begin{lemma}
      \label{LFEM}
  Let $\widetilde u$ be the solution to \eref{main} with right hand side $\widehat f:=S^*_k(f)$ and boundary value 
$\widehat g:=\overline S_k(g)$. For $n\geq k$, let  $\widehat u_n$  be the Galerkin projection of $\widetilde u$ onto $V_n$. Then, we have
\be 
  \label{FEM}
  \|u-\widehat u_n\|_{H^1(\Omega)}
  \lesssim R_{\widetilde m,\overline m}^*(\cU)_{H^1(\Omega)}+\e_n(\cU)_{H^1(\Omega)}, \quad \hbox{where}\quad \e_n(\cU)_{H^1(\Omega)}:=\dist(\cU,V_n)_{H^1(\Omega)}.
  \ee 
  \end{lemma}
  \begin{proof}
  It follows from \eref{utilde} that 
  \be
\label{FEM1}
  \|u-\widehat u_n\|_{H^1(\Omega)}
  \leq \|u-\widetilde u\|_{H^1(\Omega)}+\|\widetilde u-\widehat u_n\|_{H^1(\Omega)}\lesssim R_{\widetilde m,\overline m}^*(\cU)_{H^1(\Omega)}+\|\widetilde u-\widehat u_n\|_{H^1(\Omega)}.
  \ee
 Note that  since $n\geq k$, we have that $\widehat g \in Tr(V_n)$.
 In particular, $\widehat u_n$ and $\widetilde u$ have the same boundary data (and are defined with the same right hand side $\widehat f$).
 Consequently, $\widehat u_n$ is a quasi-best approximation to $\widetilde u$ in $V_n$,
\be\label{edno}
  \| \widetilde u - \widehat u_n \|_{H^1(\Omega)} \lesssim \inf_{v \in V_n} \| \widetilde u
  - v\|_{H^1(\Omega)}.
\ee
  Invoking \eref{utilde}, we find that for every $v\in V_n$,
  $$
  \| \widetilde u
  - v\|_{H^1(\Omega)}\leq \|\widetilde u-u\|_{H^1(\Omega)}+\| u - v\|_{H^1(\Omega)}\lesssim R_{\widetilde m,\overline m}^*(\cU)_{H^1(\Omega)}+\| u - v\|_{H^1(\Omega)},
  $$
   and hence
  $$
  \inf_{v \in V_n} \| \widetilde u
  - v\|_{H^1(\Omega)}\lesssim R_{\widetilde m,\overline m}^*(\cU)_{H^1(\Omega)}+\inf_{v \in V_n} \|u
  - v\|_{H^1(\Omega)}\leq R_{\widetilde m,\overline m}^*(\cU)_{H^1(\Omega)}+
  \e_n(\cU)_{H^1(\Omega)}.
$$
This, together with \eref{edno}, gives
\be\label{dve}
\| \widetilde u - \widehat u_n \|_{H^1(\Omega)} \lesssim R_{\widetilde m,\overline m}^*(\cU)_{H^1(\Omega)}+
  \e_n(\cU)_{H^1(\Omega)}.
\ee
Combining \eref{FEM1} and \eref{dve} yields the desired estimate \eref{FEM}.
\end{proof}

   Let us now discuss the statement of Lemma \ref{LFEM}. We have determined the optimal recovery rates  
  $R_{\widetilde m,\overline m}^*(\cU)_{H^1(\Omega)}$, see \eref{mainco} for the class $\cU$.
  The term $\e_n$ will tend to zero as $n\to\infty$ and we would obtain the optimal recovery error on the right side of \eref{FEM} if $n$ is chosen suitably large. 
  The  actual decay rate of
  $\e_n$ depends on the regularity of the model class $\cU$ in $H^1(\Omega)$. Therefore, the issue becomes what do our model class assumptions
  $f\in\cF$, $g\in\cG$, say about the regularity of $u$ in the $H^{1}(\Omega)$ norm.   If we use linear spaces such as standard finite element spaces, then we would want to know the regularity of $u$
  in the scale of $H^t(\Omega)$ spaces, $t>1$. 
  There are several theorems that obtain such regularity results (see e.g. \cite{dahlke1997besov,jerison1995inhomogeneous, mitrea2010boundary} and the references therein) and thereby obtain concrete bounds on the decay rate $\e_n$.  We do not elaborate on this further but refer the interested reader to the papers cited above.

  If we use nonlinear numerical methods such as AFEMs for solving \eref{main}  with right side $\widehat f$, and boundary value $\widehat g$, then the 
   decay rate for $\e_n(\cU)$ can be improved (see \cite{bonito2024adaptive} for a summary of such results).

  \subsection{Reduced models}
  \label{SS:RM}
  If one is faced with solving several PDEs with the same data sites $(\cX,\cZ)$ but with different 
  data $({\bf f},{\bf g})$,
  one can build  a numerical algorithm as follows.  In the case when   
  $\cX=G_{k,r}$, $|G_{k,r}|=\widetilde m= [r2^k]^d$, see \eref{Gd} and 
 $\cZ=\overline G_{\bar k,r}$, $|\overline G_{\bar k,r}|=\overline m=2d[r2^{\bar k}]^{d-1}$, see \eref{Gb},
we 
know that the simplicial interpolation operator $S_k^*$ defined in
  \eref{pwpinter1} provides a near optimal recovery for $\cF$.  Namely,
  \be 
  \label{orcF}
  \|f-S_k^*(f)\|_{H^{-1}(\Omega)}\lesssim R_{{\widetilde m}}^*(\cF)_{H^{-1}(\Omega)}, \quad \widetilde m= [r2^k]^d.
  \ee
  We  write
  \be 
  \label{orcF1}
   S_k^*(f)=\sum_{i=1}^{\widetilde m} f(\bx_i)\phi_i,
  \ee
  where the $\phi_i$ are the local Lagrange functions with $\phi_i(\bx_i)=1$
  and $\phi_i(\bx_j)=0$ when $j\neq i$, supported on the simplex containing the point $\bx_i$.  Let $\phi_i^*$ be the solution to
  \eref{main} with right side $\phi_i$ and zero boundary conditions.

  Similarly, the interpolant  
  $\overline S_{\bar k}(g)$ to $g\in \cG$ used for optimal recovery, see \eref{defSkbar}, can be
  written as
  \be 
  \label{orcG1}
\overline S_{\bar k}(g)=\sum_{i=1}^{\overline  m} g(\bz_i)\psi_i, \quad \overline m=2d[r 2^{\bar k}]^{d-1}.
  \ee
  Let $\psi_i^*$ be the solution to \eref{main} with $g=\psi_i$ and zero right side.  
  
  Given any data $({\bf f}, {\bf g})$, the function
  \be
  \label{rb}
  \widetilde u:=\sum_{i=1}^{\widetilde m}f(\bx_i)\phi_i^*+ \sum_{i=1}^{\overline m}g(\bz_i)\psi_i^*
  \ee 
  is a near optimal recovery of the solution $u$ to \eref{main} with right side $f$ and boundary values $g$.  
  Note that $\widetilde u$ is the solution to \eref{main} with right side $S_k^*(f)$ and boundary value $\overline S_{\bar k}(g)$.
  
  Given a budget $m$,   one   computes offline to a sufficiently high accuracy 
  suitable approximations $\widehat\phi_i$ and $\widehat \psi_i$ in $H^1(\Omega)$ to the $\phi_i^*$ and $\psi_i^*$, respectively. Then the function
  \be
  \label{rb1}
  \widehat u:=\sum_{i=1}^{\widetilde m}f(\bx_i)\widehat\phi_i+ \sum_{i=1}^{\overline m}g(\bz_i) \widehat \psi_i, \quad  \hbox{with}\quad \widetilde  m\asymp\overline m\asymp m/2,
  \ee 
is a near optimal recovery for  the solution $u$ to \eref{main} with $f\in\cF$, $g\in \cG$, where $\restr{f}{\cX}={\bf f}$ and 
$\restr{g}{\cZ}={\bf g}$.

Note also that the near optimality can be guaranteed for all model class assumptions
  provided $s$ and $\overline s$ are all less than a fixed number $s_0$.

\section{Numerical methods based on optimization}
\label{S:NOpt}
Another method of finding a near optimal approximation to $u$ from the given data is through optimization, such as the one  used in   PINNs.  
The starting point of this approach is the fact that  the unique solution $u\in H^1(\Omega)$ to \eref{main} is the unique minimizer of the theoretical loss  $\cL_T(u)$, see \eref{firstmin}. However, 
the minimization in \eref{firstmin} is over all $v\in H^1(\Omega)$.  To make this  approach  more amenable to numerical implementation, we take this minimization over a set $\Sigma=\Sigma_n$, where
$\Sigma$ is either a linear space of dimension $n$ or a nonlinear manifold of order $n$ (i.e.,
depending on $n$ parameters), and study any function
\be 
\label{minimize2}
S_\Sigma\in \{\argmin_{S\in\Sigma} \cL_T(S)\}.
\ee
We want to understand how close such an $S_\Sigma$ is to $u$ with error measured in the norm of $H^1(\Omega)$.  Therefore, we need to require that
$\Sigma\subset H^1(\Omega)$.

Let us fix our model class   $\cU$ as given in \eref{defcU}. If we want $S_\Sigma$ to be a good approximation to $u$, then we need  $\Sigma$ to be good at
approximating the elements of $\cU$.  We define the error
$$
E(\cU,\Sigma):= \sup_{v\in\cU}\inf_{S\in  \Sigma}\|v-S\|_{H^1(\Omega)}.
$$
If we return to our specific $u\in\cU$ determined by $f$ and $g$, then for any $S_\Sigma$ from \eref{minimize2} and any $S\in \Sigma$, we have
from \eref{equiv} that 
\be 
\label{H1err1}
\|u-S_\Sigma\|_{H^1(\Omega)}\lesssim \cL_T(S_\Sigma)\lesssim \cL_T(S)\lesssim \|u-S\|_{H^1(\Omega)}.
\ee 
Since \eref{H1err1} holds for all $S\in\Sigma$, we have that
\be 
\label{H1err}
\|u-S_\Sigma\|_{H^1(\Omega)}\lesssim 
\inf_{S\in  \Sigma}\|u-S\|_{H^1(\Omega)}\lesssim E(\cU,\Sigma).
\ee

The discussion that we have just given shows that if we solve the continuous minimization problem
\eref{firstmin} over $\Sigma$ instead of all of $H^1(\Omega)$, then the solution will get as close to
$u$ as the  efficiency of $\Sigma$ in approximating $\cU$, i.e., $E(\cU,\Sigma)$.  This error  can be made as small as we wish by taking finer spaces for $\Sigma$, i.e., letting $n$ increase. The error we incur is then proportional to the error in
$E(\cU,\Sigma)$ which tell us that the best candidates for $\Sigma$ are those  that approximate our model class $\cU$ well.

Several issues arise that prevent the direct implementation of the above loss in our setting.
The first of these is that we do not know neither $f$ nor $g$.  We only know these functions through the given
data.  The remedy for this is to introduce a surrogate for the $H^{-1}(\Omega)$ and $H^{1/2}(\partial\Omega)$ norms that  use only the data information $({\bf f},{\bf g})$ we have about $f$ and $g$.    The next section addresses this issue.

 \section{Discretizing norms} 
\label{S:discrete}
The optimization procedure of the previous section is not a numerical algorithm since it does not incorporate
numerical methods for estimating the   norms appearing  in the loss $\cL_T$, namely   $\|\cdot\|_{H^{-1}(\Omega)}$ and $\|\cdot\|_{H^{1/2}(\partial\Omega)}$. 
In this section, we address this issue by introducing discrete
norms which involve only the values of functions at the data sites.  Replacing the norms in 
$\cL_T$ by these discrete norms leads to a loss  $\cL^*$ which can be numerically
computed. The equivalence between the norms in $\cL_T$ and their discrete counterparts holds (modulo corresponding optimal recovery rates) for   functions whose Laplacians are uniformly bounded in  $\cal B$ and whose traces are uniformly bounded in $\overline{\cal B}$.  

As it will be seen in \S\ref{S:dloss}, the new loss $\cL^*$ is based on a  discretization of  the 
$\|\cdot\|_{L_\gamma(\Omega)}$ (for a special value of $\gamma$) 
and $\|\cdot\|_{H^{1/2}(\partial\Omega)}$ norms. These discretizations  have to satisfy the following two requirements:
\begin{itemize}
    \item 
use the given information $({\bf f}, {\bf g})$ at the grids
(data sites)   $G_{k,r}\subset \overline \Omega$ and 
$\overline G_{\bar k,r}=G_{\bar k,r}\cap \partial \Omega$, respectively.
\item the errors incurred for functions in $\cF$ and $\cG$ be comparable with the optimal recovery rates for these classes.
\end{itemize}
We remark that the discretization of integral norms has been studied in many other papers. For example, in \cite{siegel2023greedy} the discretization of the variational form of the Poisson equation for functions in the Barron space (or variation space corresponding to shallow ReLU$^k$ networks) is considered. This analysis is done for both random data sites and a fixed grid. There are also papers that prove existence of data sites $\cX$ for which the  continuous and discrete $L_\tau$ norms are equivalent for elements of certain linear spaces, see \cite{dai2021sampling,dai2023some}.  To the best of our knowledge, there are no results in the literature which treat the discretization of the $H^{1/2}$-norm that are comparable to what we do in \S\ref{H12}.

\subsection{A discrete $L_\tau$ norm}
\label{SS:dLtau}
In this section, we study
the discretization of the $L_\tau$ norm for any $1\le \tau\le\infty$.  Let $\cF=U(B_q^s(L_p(\Omega)))$, $s>d/p,$
   be  a model class assumption on $f$. 
     We consider the uniform grid $G_{k,r}\subset \overline\Omega$,  see \eref{Gd}, consisting of  
   $\widetilde m$ points, $\widetilde m=|G_{k,r}|=[r2^k]^d$, $k\ge 0$
and $r>\max(s,1)$.
   For any continuous function $f$ and any $1\le \tau < \infty$, we define
\be 
\label{disLgamma}
\|f\|^*_{L_\tau}:=\left[\frac{1}{\widetilde m}
\sum_{j=1}^{\widetilde m}|f({\rm\bf x}_j)|^\tau\right ]^{1/\tau}, \quad \hbox{where}\quad \bx_j\in G_{k,r}.
\ee 
When $\tau=\infty$ we make the usual modification to define the $\|f\|^*_{L_\infty}$ norm.

The following lemma shows that the discrete $L_\tau$ norm is close to the true
$L_\tau$ norm for functions in the model class $\cF$.
\begin{lemma}
\label{L:Ldiscrete}
Let $\cB=B_q^s(L_p(\Omega))$, $\Omega=(0,1)^d$,  be a Besov space with $s>d/p$.  
If $f\in \cB$, then
for any $1\le \tau\le \infty$, we have
$$
\|f\|_{L_\tau(\Omega)}\lesssim
\|f\|^*_{L_\tau}+\|f\|_{\cB}{\widetilde m}^{-\alpha_\tau}, 
\quad \hbox{and}\quad \|f\|^*_{L_\tau}\lesssim \|f\|_{L_\tau(\Omega)}  + \|f\|_{\cB}{\widetilde m}^{-\alpha_\tau}, 
$$
with 
$\alpha_\tau:=\frac{s}{d}-\left (\frac{1}{p}-\frac{1}{\tau}\right)_+$.  Recall that ${\widetilde m}^{-\alpha_\tau}$ is the uniform rate of optimal recovery of the model class $\cF$ in $L_\tau(\Omega)$.   The constant 
in   $\lesssim$ depends only on $r$ and
$d$.
\end{lemma}
\begin{proof} 
We prove the lemma for $1\le \tau<\infty$.
A simple modification of the proof handles the case $\tau=\infty$.
We use the simplicial interpolation operator
 $S^*_k$ given in \eref{pwpinter1}. Let $\cT_k:=\cT_k(\Omega)=\bigcup T$ be the set of simplices that make up the simplicial decomposition of the cubes in $\cD_k(\overline \Omega)$.
Then, from \eref{pwpinter1} we have
\be 
\label{Sknorm}
\|S_k^*(f)\|^\tau_{L_\tau(\Omega)}=\sum_{T\in\cT_k}  \|L_T(f)\|^\tau_{L_\tau(T)},
\ee
where $L_T(f)$ is the polynomial in $\cP^d_r$ that interpolates the 
data  
$\{f({\bf x}_i):\,{\bf x}_i\in \overline T\}$. By equivalence of
norms on $\cP^d_r$ (the linear space of algebraic polynomials of $d$ variables and  order $r$ (total degree $r-1$)), rescaling arguments,  and the fact that we have $d!$ simplices in the Kuhn-Tucker decomposition of a cube,
\be 
\label{dLtau}
\|L_T(f)\|_{L_\tau(T)}\asymp  \[\frac{1}{\widetilde m}\sum_{\bx_i\in 
{\overline T}}|f(\bx_i)|^\tau\]^{1/\tau},
\ee 
with   constants of equivalency depending only on $d,r$.
Summing over $T\in\cT_k$ the quantities in \eref{dLtau} gives
\be 
\label{dLtau1}
\|S_k^*(f)\|_{L_\tau(\Omega)}\asymp \[ \frac{1}{\widetilde m}\sum_{i=1}^{\widetilde m}|f(\bx_i)|^\tau\]^{1/\tau}=\|f\|^*_{L_\tau(\Omega)} .
\ee

For $f\in \cB=B_q^s(L_p(\Omega))$, the piecewise polynomial interpolant  $S_k^*(f)$ provides a  near optimal recovery of $f$ from the data ${\bf f}$, that is, 
see Theorem \ref{T:interapprox},
\be 
\label{Ld0}
 \|f -S_k^*(f)\|_{L_\tau(\Omega)} \le \begin{cases}
 C|f|_{B_p^s(\Omega)} 2^{-k(s-d/p+d/\tau)}= C\|f\|_{\cB}{\widetilde m}^{-\alpha_\tau}, \quad \quad\quad \quad\quad \quad \tau\geq p,\\\\
\|f -S_k^*(f)\|_{L_p(\Omega)}\le C|f|_{B_p^s(\Omega)} 2^{-ks}= C\|f\|_{\cB}{\widetilde m}^{-\alpha_\tau}, \quad \tau< p,
 \end{cases}
\ee 
where we have used that $\widetilde m=[r2^k]^d$.
It follows then from \eref{dLtau1} that 
\be 
\label{Ld1}
\left |\|f\|_{L_\tau(\Omega)}-\|S_k^*(f)\|_{L_\tau(\Omega)}\right |\le C\|f\|_{\cB}{\widetilde m}^{-\alpha_\tau},
\ee
and the proof is completed.
\end{proof}

\subsection{Discrete $H^{1/2}(\partial \Omega)$ norms}
\label{H12}

 In this section, we introduce discrete $H^{1/2}(\partial\Omega)$ norms and discuss their accuracy in computing the true $H^{1/2}(\partial\Omega)$ norm for functions $g$ in our model class $\cG$. The discrete norm
  we introduce only uses the values of $g$ at the data sites 
  and can therefore be computed from the given data. 
 We fix any $k\ge 1$ and let 
 $\cZ:=\overline G_{k,r}=\{\bz_j\}_{j=1}^{\overline m}=G_{k,r}\cap \partial \Omega$, see \eref{Gb},  be the grid points on the boundary $\partial\Omega$.    The number of grid points in $\cZ$ is $\overline m=2d[r2^{k}]^{d-1}$. 

We fix a model class $\cG=Tr(U(\overline\cB))$ of the form \eref{largestG}.  The optimal recovery rate for this class is
\be 
\label{ORGH}
R_{\overline m}^*(\cG)_{H^{1/2}(\partial\Omega)}\asymp {\overline m}^{-\beta},\quad { \overline m}\ge 1,\quad \beta= \frac{\bar s-1}{d-1}.
\ee
This also serves to fix $r$ which we recall is the smallest integer satisfying $r>\max(\bar s,1)$.
For any continuous function $g$ which is the trace of a $v\in\overline B$, we use the trace norm
\be 
\label{Trnorm}
\|g\|_{Tr(\overline \cB)}:=\inf_{Tr(v)=g} \|v\|_{\overline B}.
\ee
 Recall that the intrinsic semi-norm  for $H^{1/2}(\partial\Omega)$ is given by   
\be 
\label{iH1/2}
|g|_{H^{1/2}(\partial \Omega)}^2:= \int_{\partial\Omega\times \partial\Omega}\frac {|g(\bx)- g(\by)|^2}{|\bx-\by|^d}\,d\bx d\by.
\ee
We obtain the intrinsic $H^{1/2}(\partial\Omega)$ norm by adding
$\|g\|_{L_2(\partial\Omega)}$ to this semi-norm.
It is known that this definition is equivalent to the trace norm definition \eref{Trnorm} for this domain $\Omega$ (see \cite{grisvard2011elliptic}).  We shall use \eref{iH1/2} through out this section.

Let $\overline S_k$ be as defined in \eref{defSkbar}.  We have proven in Theorem \ref{T:ORg} that 
\be 
\label{errorg}
\|g-\overline S_k(g)\|_{H^{1/2}(\partial\Omega)}\leq C \|g\|_{Tr(\overline \cB)} \overline {m}^{-\beta},  
\ee 
with $C$ independent of $\overline{m}$ and $g$.
This gives the comparison
$$
\left |\|g\|_{H^{1/2}(\partial\Omega)}-
\|\overline S_k(g)\|_{H^{1/2}(\partial\Omega)}\right |\leq \|g-\overline S_k(g)\|_{H^{1/2}(\partial\Omega)}\leq C\|g\|_{Tr(\overline\cB)}\overline {m}^{-\beta}.
$$
We concentrate therefore on finding a discrete $H^{1/2}(\partial\Omega)$ norm for the  functions $\overline S_k(g)$, $g\in\cG$.  Note that $\overline S_k(g)$ is completely determined by the values of $g$ at the data sites $\overline G_{k,r}$.
For each $k\ge 0$, we define $\cD_k(\partial \Omega):=\partial \Omega\cap\cD_k(\overline \Omega)$ to be the  partition of $\partial \Omega$ into  dyadic cubes $I$ of side length $2^{-k}$.
    Given $k\ge 0$, let  $ \cT_k:=\cT_k(\partial \Omega)$ denote the collection of all $d-1$ dimensional simplicies  given by the    Kuhn-Tucker partition of   the  dyadic cubes $I\in\cD_k(\partial\Omega)$ of  the boundary $\partial\Omega$.  For each dyadic cube $I\in\cD_k(\partial\Omega)$, there are $(d-1)!$ simplices contained in $I$

    Let us further introduce the finite dimensional linear space $V^r(\cT_k)$ of functions that are  continuous on $\partial\Omega$ and piecewise polynomial (subordinate to $\cT_k$) of order $r$ on $\partial \Omega$.  The linear operator $\overline S_k$   is a projector onto $V^r(\cT_k)$.

     If $ g$  is in $C(\partial \Omega)$, we define
 \begin{equation}
 \label{dH1/2}
| g|_{H^{1/2}(\partial \Omega)}^* :=   \left[ \frac 1 {\overline{m}^{2}}\sum_{\substack{i \not =  j\\
}} \frac{|g(\bz_i)-g(\bz_j)|^{2}}{|\bz_i-\bz_j|^{d}}\right]^{1/2}
\end{equation}
and prove that the  the two  semi-norms $|\cdot |_{H^{1/2}(\partial \Omega)}^* $ and $| \cdot|_{H^{1/2}(\partial \Omega)} $ are equivalent on $V^r(\cT_k)$.
We begin with the following lemma.

\begin{lemma}\label{L:equivalenceH12}
    For each pair $T,T' \in \cT_k$ and each $S\in V^r(\cT_k)$, we have
    \begin{equation}
    \label{e:local_equiv_discrete}
    \left [\frac 1 {\overline {m}^{2}}\sum_{
    {\substack{i \not =  j \\ \bz_i \in \overline T, \ \bz_j \in \overline T'}}} \frac{|S({\bz_i})-S({\bz_j})|^{2}}{|{
    \bz_i-\bz_j}|^{d}}\right ]^{1/2}
    {\asymp} \left [ \int_{T\times T'}  \frac{|S(\bz)-S(\bz')|^2}{|\bz - \bz'|^{d}}d\bz \  d\bz' \right ]^{1/2}, 
    \end{equation}
  { where the constants in $\asymp$ depend} only on $r$ and $d$. 
   \end{lemma}
\begin{proof}  
We fix $d,\,r$,   and a pair $T,T'\in \cT_k:=\cT_k(\partial\Omega)$.
Let us first consider the case when $k\le 4$.  
Let $\cN(S):=\cN_{T,T'}(S)$ denote the expression in \eref{e:local_equiv_discrete} involving the sum, and  $\cN'(S)$ the expression involving the integral.  Both are semi-norms on the finite dimensional linear space $X=X(T,T')$ of functions from $V^r(\cT_k)$ restricted to $\overline {T}\cup \overline{T'}$.  A function $S\in X$ satisfies $\cN(S)=0$ if and only if $S$ is constant on $\overline T$, {$\overline{T'}$}  and, additionally, $S$ is constant on $\overline {T}\cup \overline{T'}$ when $\overline {T}$ and $\overline{T'}$ touch.  A similar statement holds for the semi-norm $\cN'$.   It follows that for any pair $T,T'$, we have
\be 
\label{kernel}
c\cN(S)\le \cN'({S})\le C\cN(S), \quad S\in X(T,T'),
\ee 
where the constants $c,C$ depend on $T$, $T'$,    $r$ and $d$,  but are independent of $S$.  
Since 
 there are only a finite number (depending on $d$) of pairs of simplicies in $\cT_k(\partial\Omega)$, $k\leq 4$,  there are constants $c_1,C_1$  
which depend only on 
$r,d$ for which   \eref {kernel} is valid for all pairs of simplices from 
  $\cT_k(\partial\Omega)$, $k\leq 4$. Then
\eref{kernel} holds for all such  pairs with $c=c_1$, $C=C_1$, and  we have completed the proof in this case.

We now consider the case $k>4$.  
Given a pair $T,T'$, we let
$c(T,T')$ be the largest number such that the lower inequality in \eref{kernel} holds uniformly for all $S\in X(T,T')$. Similarly $C(T,T')$ is the smallest constant so that the upper inequality of \eref{kernel} holds uniformly on $X(T,T')$. The above argument using equivalence of semi-norms shows that there always exist positive constants $c(T,T'), C(T,T')$.  We are left to show
that there are constants $0<c^*<C^*<\infty$, depending only on $d$ and $r$,  for  which $c(T,T')\ge c^*$
and $C(T,T')\le C^*$ holds for all pairs $T,T'$ such that $T,T'\in\cT_k$ and $k>4$.   We consider two cases.
\vskip .1in
\noindent 
{\bf Case 1:  $\overline T\cap \overline{T'}\neq \emptyset$.} 
\vskip .1in
In this case, there is a pair $T_0,T_0'\in \cT_4$, and  a   linear  mapping consisting of a translation and 
a dilation with factor $2^{-k+4}$ that rigidly transforms $T_0,T_0'$ to the pair $T,T'$.  If we use this linear
transformation to change variables in \eref {L:equivalenceH12} we obtain the validity of \eref {L:equivalenceH12}
for this case with the same constants $c_1,C_1$.
 
\vskip .1in
\noindent 
{\bf Case 2: $\overline{T}\cap \overline{T'}=\emptyset$.}
First note that from the $L_2$ discretization of the previous section, we have that there are constants
$c_2,C_2$ depending only on $r,d$ such that for each $S\in V^r(\cT_k(\partial\Omega))$
   \be 
   \label{firstddis}
   c_2\[ \int_{T\times T'}  |S(\bz)-S(\bz')|^2 d\bz \  d\bz' \]^{1/2}
    \le \[\frac 1 {\overline{m}^{2}}\sum_{\substack{\bz \not =  \bz' \in \mathcal Z\\ \bz \in {\overline T}, \ \bz' \in {\overline T'}}} |S(\bz)-S(\bz')|^{2} \]^{1/2}\le C_2 \[ \int_{T\times T'}  |S(\bz)-S(\bz')|^2 d\bz \  d\bz' \]^{1/2}. 
    \ee
 Now, on $T\times T'$,   the expressions $|\bz-\bz'|$  that appear in \eref  {e:local_equiv_discrete} are all comparable with absolute constants.  Therefore, we obtain
 \eref  {e:local_equiv_discrete} holds in this case as well with constants depending only on $r,d$.  
\end{proof}

We can now prove that the discrete semi-norm and the true semi-norm of functions in $\cG$ are comparable.

\begin{theorem} 
\label{T:dH121}
For any $g\in\cG=Tr(U(\overline B))$, where  
$\overline B=B^{\bar s}_\infty(L_2(\Omega))$  with $\bar s>d/2$, and any $\overline m\geq 1$, we have
\be 
\label{comparedH1/2}
|g|_{H^{1/2}(\partial\Omega)}\lesssim |g|^*_{H^{1/2}(\partial\Omega)}+{\|g\|_{Tr(\overline \cB)} }\overline  m^{-\beta}, \quad \hbox{and}\quad 
|g|^*_{H^{1/2}(\partial\Omega)}\lesssim |g|_{H^{1/2}(\partial\Omega)}+{\|g\|_{Tr(\overline \cB)}}\overline  m^{-\beta},
\ee
with 
\be
\label{defbeta}
\beta=\frac{\bar s-1}{d-1},
\ee 
and constants in $\lesssim$ depending only on   $r,d$.
Note that the optimal recovery rate of $\cG$ in the $H^{1/2}(\partial\Omega)$ is $\overline m^{-\beta}$.
\end{theorem}
\begin{proof}
 For any $S\in V^r(\cT_k)$, we have
$$
| S |_{H^{1/2}(\partial \Omega)} = \left[\sum_{T,T' \in \cT_k}  \int_T \int_{T'} \frac{| S(\bz)-S(\bz') |^2}{|\bz -\bz'|^{d}} \ d\bz \ d\bz'\right]^{1/2},
$$
and
$$
| S |_{H^{1/2}(\partial \Omega)}^* = \left[\sum_{\bz \not = \bz' \in \overline G_{k,r}}  \frac{1}{\overline m^2} \frac{| S(\bz)-S(\bz') |^2}{|\bz -\bz'|^{d}} \right[^{1/2} \asymp \left[\sum_{T,T' \in \cT_k} \sum_{\substack{\bz \not =  \bz' \in \overline G_{k,r}\\ \bz \in {\overline T}, \ \bz' \in {\overline T'}}}  \frac{1}{\overline m^2} \frac{| S(\bz)-S(\bz') |^2}{|\bz -\bz'|^{d}} \right]^{1/2}.
$$
Here, in the second expression, we only have equivalence since a given data site may be used for more than one simplex.
Note, however that the constant in the second equivalence depends only on $d$ and $r$.

 In view of Lemma \ref{L:equivalenceH12}, we know that there exist constants $0<c<C<\infty$ that depend only on $d$ and $r$ such that for each $S\in V^r(\cT_k)$, we have
 \be 
 \label{knowH12}
c| S |_{H^{1/2}(\partial \Omega)}\le | S |_{H^{1/2}(\partial \Omega)}^*\le C| S |_{H^{1/2}(\partial \Omega)}.
 \ee 
In particular this holds for $S=\overline S_k(g)$ whenever $g\in \cG$.  If we combine \eref{knowH12} with \eref{errorg}
and observe that $|g|^*_{H^{1/2}(\partial\Omega)}=|\overline S_k(g)|^*_{H^{1/2}(\partial\Omega)}$,  we obtain the theorem.
\end{proof}

 In order to define a discrete norm for $H^{1/2}(\partial\Omega)$, we will also need a discrete $L_2(\partial \Omega)$ norm for the functions in our model class $\cG$.  Using the same $\overline m$  data sites $\cZ:=\overline G_{k,r}$, we define
\be 
\label{dL2boundary}
\|g\|^*_{L_2(\partial\Omega)}:=\left[\frac{1}{\overline m}\sum_{j=1}^{\overline m} |g(\bz_j)|^2\right]^{1/2},\quad 
g\in\cG, \quad \bz_j\in\overline G_{k,r}.
\ee
Arguing as we have done in \S\ref{SS:dLtau}, we have for $\alpha= \frac{\bar s-1/2}{d-1} >\beta$ and any $g\in Tr(\overline \cB)$,
\be\label{discr}
\|g\|_{L_2(\partial\Omega)}\lesssim
\|g\|^*_{L_2(\partial\Omega)}+\|g\|_{Tr(\overline \cB)} \overline m^{-\alpha},  \quad\hbox{and}\quad 
 \|g\|^*_{L_2(\partial\Omega)}\lesssim\|g\|_{L_2(\partial \Omega)}  + \|g\|_{Tr(\overline \cB)} \overline m^{-\alpha}.
 \ee

 We now define for any continuous function $g$ the discrete $H^{1/2}(\partial\Omega)$ norm
 \be 
 \label{dH12norm}
 \|g\|^*_{H^{1/2}(\partial\Omega)}:=\|g\|^*_{L_2(\partial\Omega)}+|g|^*_{H^{1/2}(\partial\Omega)}.
 \ee
 
 \begin{theorem}
 \label{T:dH12norm}
 For any  $g\in\cG=Tr(U(\overline B))$, where  
$\overline \cB=B^{\bar s}_\infty(L_2(\Omega))$
with $\bar s>d/2$, and any $\overline m\geq 1$, we have
 \be 
\label{TcomparedH1/2}
\|g\|_{H^{1/2}(\partial\Omega)}\lesssim \|g\|^*_{H^{1/2}(\partial\Omega)}+ \|g\|_{Tr(\overline \cB)}\overline m^{-\beta}, \quad \hbox{and}\quad 
\|g\|^*_{H^{1/2}(\partial\Omega)}\lesssim \|g\|_{H^{1/2}(\partial\Omega)}+ \|g\|_{Tr(\overline \cB)}\overline m^{-\beta},
\ee
where $\beta$  is given in \eref{defbeta}, and the constants in $\lesssim$ depend only on $r$ and $d$.
 \end{theorem}
\begin{proof} 
The theorem follows from the comparisons 
\eref{comparedH1/2}, 
\eref{discr},  and the fact that $\alpha>\beta$. 
\end{proof}

 \section{A discrete loss function with error control}
 \label{S:dloss}
In section \S\ref{S:NOR2},  we have introduced  several numerical methods for finding a near optimal recovery to the solution $u\in \cU_{\rm data}$  from  the given data. In this section, we analyse an alternative  approach that is based on optimization (such as PINNs). We consider a discrete loss function
$\cL^*$ which is a surrogate for the 
theoretical loss function $\cL_T$, see \eref{Tloss}.  The advantage of   $\cL^*$ is
that it can be computed directly from the data 
$({\bf f},{\bf g})$. 
Let $k,r$ be fixed and let $\cX=G_{k,r}\subset \overline \Omega$ and $\cZ=\overline G_{\bar k,r}$
be the data sites of $\partial\Omega$.  Note that $\bar k$ could be different from $k$. This fixes $\widetilde m=\#(G_{k,r})$ and $\overline m:=\#(\overline G_{\bar k,r})$ 
for the remainder of this section.

  \begin{remark}
    \label{R:convention}
    In this and the following two sections, the results and their proofs  are simplest in the case $d\ge 3$.  When $d=2$,  log  factors appear.
    For this reason, we state the results in correct form (including the log factors) for all cases $d\ge 2$, but give the proofs only for $d\ge 3$.  The proofs for the case $d=2$ are given  in the Appendix, \S\ref{d2}.
\end{remark}

 Let $\|\cdot\|^*_{L_\gamma}$
be the discrete norm, see \eref{dLtau}, introduced in \S\ref{SS:dLtau}  
  for the following choices of $\gamma$,
 \begin{equation}\label{gamma-choice-consistent-pinns}
 \gamma=\begin{cases}
\frac{2d}{d+2}, \quad \quad &d\geq 3, \\ \\
1+[\log (\widetilde m)]^{-1}, \,\,\,\quad &d=2,
 \end{cases}
 \end{equation}
 and let $\|\cdot\|^*_{H^{1/2}}$ be the discrete $H^{1/2}(\partial\Omega)$ introduced in \S\ref{H12}, see \eref{dH1/2}, \eref{dL2boundary}, and  \eref{dH12norm}.   Given the data vectors $({\bf f},{\bf g})$, we introduce  the discrete loss function
\be 
\label{dloss}
\cL^*(v):= \begin{cases}
    \| f+\Delta v\|^*_{L_\gamma(\Omega)}+\|g-v\|^*_{H^{1/2}(\partial \Omega)},\quad \quad \quad \quad &d\geq 3,\\\\
    [1+\log (\widetilde m)]\| f+\Delta v\|^*_{L_\gamma(\Omega)}+\|g-v\|^*_{H^{1/2}(\partial \Omega)}, \quad &d=2,
    \end{cases}
\ee 
which is defined  whenever
$v$ and $\Delta v$ are continuous on 
{$\overline\Omega$}.  

We want to show that making $\cL^*(v)$ small guarantees that $v$ is a good approximation to $u$ in the $H^1(\Omega)$ norm.  As in \cite{binev2024optimal}, such a result requires model class assumptions on $f$ and $g$.  Keeping in mind the remarks of
\S \ref{SS:finalobservation}, we make the model class assumptions   $\cF=U(\cB)$ of \eref{onlyF}
and  $\cG=Tr(U(\overline \cB))$ of \eref{largestG}.

Recall that,  in the case $d\ge 3$  and all $0<p\leq \infty$, or $d=2$ and $1<p\leq \infty$, for a model class $\cF$ of the form \eref{onlyF}, the uniform optimal recovery rate for $\cF$ in $H^{-1}(\Omega)$ is   $\asymp \widetilde m^{-s/d}$, 
$\widetilde m\ge 1$ 
where $s$ is the excess regularity of $\cF$ in $L_{ \delta}(\Omega)$. 
 For a model class  $\cG$ of the form \eref{largestG}, its uniform optimal recovery rate in $H^{1/2}(\partial\Omega)$ is $ \asymp\overline m^{-(\bar s-1)/(d-1)}$ ,where $\bar s-1$ is the excess
 regularity in $H^{1/2}(\partial\Omega)$.  
These model class assumptions on $f,g$ imply that the solution $u$ to \eref{main} is in a model class $\cU$  which has a uniform  optimal recovery rate 
$\max\{\widetilde m^{-\frac{s}{d}}, \overline m^{-(\bar s-1)/(d-1)}\}$.
 We introduce the following notation for a function $v$
\be 
  \label{defnormU}
  \|v\|_{\cU}:= \max\{\|\Delta v\|_{\cB},\|Tr(v)\|_{Tr(\overline \cB)}\}.
  \ee

The following theorem bounds the error $\|u-v\|_{H^1(\Omega)}$ in terms of the discrete loss
$\cL^*(v)$,  provided $\Delta v\in\cB$ and   $v\in \overline\cB$.

\begin{theorem} 
\label{T:dloss}
 Let $u$ be the solution to \eref{main}
 with $f\in \cF=U(\cB)$ of \eref{onlyF}
and  $g\in \cG=Tr(U(\overline \cB))$ of \eref{largestG}.
Given the data $({\bf f}, {\bf g})$ of $f$ and $g$ at grid data sites $(G_{k,r},\overline G_{\bar k,r})$ with $|G_{k,r}|=\widetilde m$ and $|\overline G_{\bar k,r}|=\overline m$, consider the functional $\cL^*$, defined in \eref{dloss}. 
If $v$ is any continuous function
in $H^1(\Omega)$, then
\be 
\label{Tdloss}
\|u-v\|_{H^1(\Omega)}\lesssim \cL^*(v)+\left [1+\| v\|_{\cU}\right] \cR_{\cU}(\widetilde{m}, \overline{m}),
\ee
with the constants in $\lesssim$ independent of $u,v$, $\widetilde m$ and $\overline m$, where
\be
\label{ORnotation}
\cR_{\cU}(\widetilde{m}, \overline{m}) := \begin{cases}
     \max\{ {\widetilde m}^{-\frac{s}{d}},\overline m^{-\frac{\bar s-1}{d-1}}\}, & d \geq 3,\\
    \max\{ \log({\widetilde m})
    {\widetilde m}^{-\frac{s}{2}},\overline m^{-(\bar s-1)}\},  & d = 2.
\end{cases} 
\ee
\end{theorem}
\begin{proof}  We only prove this for $d \geq 3$.  From \eref{equiv}, we have
\begin{eqnarray*} 
  \|u-v\|_{H^1(\Omega)}&\lesssim& \|f+\Delta v\|_{H^{-1}(\Omega)} +\|g-Tr(v)\|_{H^{1/2}(\partial\Omega)}
\lesssim \|f+\Delta v\|_{L_\gamma(\Omega)}+
\|g-Tr(v)\|_{H^{1/2}(\partial\Omega)}\nonumber\\
&\lesssim& \left [\|f+\Delta v\|^*_{L_\gamma(\Omega)}+
\|g-Tr(v)\|^*_{H^{1/2}(\Omega)}\right ]+\left[\|f+\Delta v\|_{\cB} \widetilde  m^{-\frac{s}{d}}+\|g-Tr(v)\|_{Tr(\overline \cB)} \overline m^{-\frac{(\bar s-1)}{d-1}}\right],
\end{eqnarray*}
where in the second inequality we used the continuous embedding of $L_\gamma(\Omega)$ into $H^{-1}(\Omega)$ and  in the third inequality we used
the comparisons between continuous and discrete norms (see Lemma \ref{L:Ldiscrete} and Theorem \ref{T:dH121}).   Finally, if we use 
the facts that $\|f\|_{\cB}\le 1$ and
$\|g\|_{Tr(\overline\cB)}\le 1$, we complete the proof of the theorem.
\end{proof}

Let us make some remarks on this theorem. 

\begin{remark}
 If we knew in advance the smoothness class of $f$, or at least the value of $p$, and if this value were $1<p\leq \infty$, then in the case $d=2$ we could modify appropriately the  the  loss $\cL^*$ so that it is independent of $\widetilde m$. This choice would lead to the actual optimal recovery rates 
 $\max\{{\widetilde m}^{-\frac{s}{2}},\overline m^{-(\bar s-1)}\}$ for this class (i.e the logarithm will not be present), see the Appendix, \S\ref{d2}.
\end{remark}

\begin{remark} It is natural that  $\|v\|_{\cU}$ enters into the error bounds  since we need some control on $v$ away from the data.   This means that when we try to obtain a good approximation to $u$ (in $H^1(\Omega)$) through optimization, then
either the norm $\|\cdot\|_{\cU}$ will have to participate in the loss function  as a regularization term,  see \cite{binev2024optimal}, or the approximant has to come from a set with good  restricted approximation properties, see Theorem \ref{T:NOR} and Remark \ref{Rnew}. In both of these situations there does not appear to exist an optimization procedure with a priori guarantees when using deep neural networks, although using shallow networks a priori bounds can be achieved using a greedy algorithm \cite{siegel2023greedy}. In practice, however, once an approximant  $v$ is obtained (say, as a result of a minimization of a loss that does not involve $\|v\|_{\cU}$), we can in principle use \eref{Tdloss} to compute an upper bound for the quality of approximation 
$\|u-v\|_{H^1(\Omega)}$, provided that we can compute an (a posteriori) estimate on $\|v\|_{\cU}$. In fact, similar methods for certifying PINNs solutions have been studied before (see for instance \cite{eirasefficient} and the references therein).
\end{remark}
\begin{remark}
    This error estimator can be used
to monitor the error in any optimization scheme that attempts to minimize the loss $\cL^*(v)$.  For example, suppose that 
  $\Sigma=\Sigma_n$ is a linear space of dimension $n$ or a nonlinear manifold determined by $n$ parameters such that each
$S\in\Sigma$ is continuous and has a continuous Laplacian $\Delta S$. If at any stage of the optimization
procedure we have an $S\in\Sigma_n$, then we can use \eref{Tdloss} to bound the error of $\|u-S\|_{H^1(\Omega)}$.
Of course, a good estimate requires a bound for  $\|S\|_\cU$.
\end{remark}

\begin{remark}\label{remark-section-7-2824}
    Note that for fixed values of $\widetilde m$ and $\overline{m}$, the losses $\cL$ and $\cL^*$ are equivalent up to a constant depending upon the parameters $\widetilde m$ and $\overline{m}$. This means that if the original loss $\cL$ is driven to $0$, i.e. we are perfectly interpolating the data, then both losses give the same error control. However, in the situation where the data is not interpolated, the new loss $\cL^*$ gives independent on $m$ control on the error according to Theorem \ref{T:dloss}, while the original PINNs loss $\cL$ does not.
\end{remark}

 In view of the above remarks, we see that the effectiveness of  using $\Sigma_n$ together with the discrete loss $\cL^*$ to provide an approximation to $u$
 is not simply governed by 
  the error
 $E(\cU,\Sigma_n)_{H^1(\Omega)} $ 
  of approximating the class $\cU$ 
 by $\Sigma_n$ in the $H^1(\Omega)$
 but rather  by a form of restricted approximation
 which  involves $\|\cdot\|_\cU$ \cite{binev2024optimal}.  

\section{Implications for Loss Function Design}
\label{S:implications}

In this section, we discuss the implications of Theorem \ref{T:dloss} on the design of loss functions for the solution of elliptic PDEs using PINNs. Theorem \ref{T:dloss} implies that the loss function $\cL^*$ given in \eqref{dloss} gives an a posteriori control on the error up to the optimal recovery rate. However, there are numerous factors that need to be taken into account when designing an appropriate loss function, including the ease of training and implementation. For this reason, directly optimizing the loss $\cL^*$ may not be the best choice in practice.

We observe that any loss function $\cL$ which upper bounds $\cL^*$ also gives a posteriori error control via Theorem \ref{T:dloss}. With this in mind, we consider the following variants of PINNs for elliptic PDEs. 
We consider the case $d\geq 3$. Similar discussion holds when $d=2$, taking into account the form of $\cL^*$ in \eqref{dloss}:
\begin{itemize}
    \item The original PINNs formulation with an adjustable weight $\lambda > 0$ on the boundary terms, 
\begin{equation}\label{original-weighted-pinns}
\cL_{sq,\lambda}(v) := \left[\frac{1}{\widetilde m}\sum_{i=1}^{\widetilde m} [\Delta v (\bx_i)+f(\bx_i)]^2\right]+\left[\frac{\lambda}{\overline{m}} \sum_{i=1}^{\overline m} [v(\bz_i)-g(\bz_i)]^2\right].
\end{equation}
Note that the case $\lambda = 1$ corresponds to the original formulation of PINNs \cite{raissi2019physics}  and is equivalent to the (square of the) loss $\cL$, see \eref {square-root-loss-definition}. Numerous heuristic strategies for choosing the weight $\lambda$ have been explored in a variety of works, see \cite{de2023operator,wang2021understanding,wang2022and}.
\item The loss function $\cL^*$ in \eqref{dloss} with exponent $\tau \geq 1$
\begin{equation}\label{consistent-pinns-formulation}
\begin{split}
    \cL_{sq,\tau}^*(v) := \!\!\left [\frac{1}{\widetilde m}\sum_{i=1}^{\widetilde m} |f(\bx_i)+\Delta v(\bx_i)|^\tau \right ]^{2/\tau}\!\!\!\!&+
 \left [\frac{1}{\overline m^2}\sum_{\substack{i,j=1 \\ i\neq j}}^{\overline m}
 \frac{|[g-v](\bz_i)-[g-v](\bz_j)|^2}{|\bz_i-\bz_j|^2}\right ]\\
 &+\left[\frac{1}{\overline m}\sum_{j=1}^{\overline m}|g(\bz_j)-v(\bz_j)|^2\right].
\end{split}
\end{equation}
When $\tau = \gamma$ given in \eqref{gamma-choice-consistent-pinns}, this 
is equivalent to the (square of the) loss $\cL^*$.
\end{itemize}
Note that all of these loss functions are squared, which is typically done to allow for easier optimization (the notation $sq$ here is a shorthand for squared).

We first observe that the $L_\gamma$-norm is bounded by the $L_\tau$-norm for any $\tau \geq \gamma$ on a finite measure space, such as the bounded domain $\Omega$, or the normalized empirical measure at the data points $\cX$. This means that 
\be
\label{hat}
(\cL^*)^2 \lesssim\cL_{sq,\tau}^*\quad \hbox{for any}\quad \tau \geq \gamma,
\ee
and thus
Theorem \ref{T:dloss} gives an a posteriori bound on the solution error in terms of $\cL_{sq,\tau}^*$. This means that we can switch to a different norm $L_\tau$ with $\tau \geq \gamma$ on the domain term in the loss $\cL^*$ and retain an a posteriori error guarantee.

The original PINNs loss $\cL_{sq,\lambda}$ with fixed, 
 independent on $\overline m$,
$\lambda$ is not an upper bound (with constant independent of $\overline{m}$ or $\widetilde{m}$) on the loss $(\cL^*)^2$, and thus does not give an a posteriori error estimate via Theorem \ref{T:dloss}. This is due to the fact that the discrete $L_2$-norm is not an  upper bound for the discrete $H^{1/2}$-norm on the boundary used in $\cL^*$. However, by letting $\lambda$ depend upon the number  $\overline{m}$ of the boundary collocation points, one can correct this. Indeed, 
 it follows from \eref{hat} with $\tau=2>\gamma=\frac{2d}{d+2}$ that 
$$
(\cL^*)^2 
\lesssim\cL^*_{sq,2}\lesssim  \cL_{sq,\lambda(\overline{m})}, \quad \hbox{where}\quad 
\lambda(\overline{m}) := \overline{m}^{1/(d-1)},
$$
and the 
constants in $\lesssim$ are  independent of $\overline{m}$. The last inequality  is easily proved by observing that
\begin{equation}
    \frac{1}{\overline m^2}\sum_{\substack{i,j=1 \\ i\neq j}}^{\overline m}
 \frac{|u(\bz_i)-u(\bz_j)|^2}{|\bz_i-\bz_j|^d} \leq \frac{2}{\overline m^2}\sum_{\substack{i,j=1 \\ i\neq j}}^{\overline m}
 \frac{|u(\bz_i)|^2 + |u(\bz_j)|^2}{|\bz_i-\bz_j|^d} \lesssim \overline{m}^{-\frac{d-2}{d-1}}\sum_{i=1}^{\overline{m}} |u(\bz_i)|^2.
\end{equation}
\begin{remark}\label{r:PINNs-weighting}
The 
value   $(\cL_{sq,\lambda(\overline{m})}(v))^{1/2}$ can be used in Theorem \ref{T:dloss} instead of $\cL^*(v)$. This indicates how to \textbf{correctly weight} the boundary and domain terms depending on the number of collocation points in the original PINNs formulation, i.e. to use the loss

\be\label{Lnew}
\cL_{sq,\lambda(\overline m)}(v) = \left[\frac{1}{\widetilde m}\sum_{i=1}^{\widetilde m} [\Delta v (\bx_i)+f(\bx_i)]^2\right]+\left[\overline{m}^{(2-d)/(d-1)}\sum_{i=1}^{\overline m} [v(\bz_i)-g(\bz_i)]^2\right].
\ee
\end{remark}

All of these alternative loss functions provide an upper bound for the loss $\cL^*_{sq,\gamma}$. Consequently, if the (restricted) optimization problem were solved exactly, the loss $\cL^*_{sq,\gamma}$ would reach a lower value than these alternatives and would give the best a posteriori error bound up to OR rates. However, in practice, the optimization of neural networks introduces numerous complications, and it is not a priori clear which loss function is best. In \S\ref{numerical-section}, we present numerical experiments comparing the practical performance of these loss functions.

Finally, we remark that our analysis of loss functions is performed with respect to the $H^1$-error. This same analysis can be performed for different error norms as well (see for instance \cite{zeinhofer2023unified} for the $H^{1/2}$-norm), but we leave this for future work.


  
\section{A numerical optimal recovery algorithm}
 \label{S:NOR}
 
Let $\cU$ be the model class \eref{defcU}, determined by
$\cF$ and $\cG$,   where $\cF$ is a maximal model class \eref{onlyF},  \eref{onlyF1} and $\cG$ is a maximal model class \eref{largestG}.   
 To build a numerical algorithm for the recovery of $u\in\cU_{\rm data}$, we want to  minimize the discrete loss $\cL^*$ over a suitable set $\Sigma_n$.  Let  $M\ge 1$ and let $\widetilde u$ be such that 
\be\label{dL*}
\widetilde u\in \{ \argmin_{\|v\|_{\cU}\le M}\cL^*(v)\}.
\ee
Clearly,  $u$ is a solution to this problem since $\|u\|_{\cU}\leq 1$ and $\cL^*(u)=0$. Thus, 
$\cL^*(\widetilde u)=0$ and it follows from \eref{Tdloss} that 
\be 
\label{L*satisfies}
\|u-\widetilde  u\|_{H^1(\Omega)}\lesssim M \cR_{\cU}(\tilde m, \bar m),\quad m\ge 1,
\ee
with the implied constant depending  only on $r$ and $d$.  In other words, $\widetilde u$ provides a near optimal recovery when $d\ge 3$ and a near optimal recovery up to a $\log m$ factor when $d=2$.
However, the minimization problem
\eref{dL*} is taken over too large of a set to be numerically viable and so we would like to utilize minimization
over a smaller set. 

Let $\Sigma_n \subset H^1(\Omega)$   be an approximating set.  Here, the primary examples for $\Sigma_n$  are linear spaces of finite dimension $n$ or a nonlinear manifolds depending on $n$ parameters such as NNs.  For $M\ge 1$ 
(typically taken as some fixed number not dependent on $n$), we define the set 
\be 
\label{defStilde1}
  \Sigma_n(M):=\{S\in\Sigma_n: \|S\|_{\cU} \leq M\}\subset \Sigma_n,
\ee 
 and consider the solution $\widehat S$ to the minimization problem
\be 
\label{dL*min1}
\widehat S\in\{ \argmin_{S\in\Sigma_n(M)} \cL^*(S)\}.
\ee 
In order to present our results, we denote by 
\be
\label{restapprox}
E(v, \Sigma_n(M)):= \begin{cases}
    \inf_{S\in  \Sigma_n(M)}   
\left (\|\Delta v-\Delta S\|_{L_\gamma(\Omega)}+\|Tr(v-S)\|_{H^{1/2}(\partial\Omega)}\right ), & d \geq 3,\\ \\
\inf_{S\in  \Sigma_n(M)}   
\left ((1+\log(\widetilde{m}))\|\Delta v-\Delta S\|_{L_\gamma(\Omega)}+\|Tr(v-S)\|_{H^{1/2}(\partial\Omega)}\right ), & d = 2,
\end{cases}
\ee 
the error of simultaneously approximating $v$ and $\Delta v$. We remark that in the case $d = 2$ and $1 < p\leq \infty$, this analysis can be modified to remove the logarithm. Then, for  the  model class $\cU$ of $H^1(\Omega)$, we denote by 
$$
E(\cU, \Sigma_n(M)):= \sup_{v\in\cU}E(v,\Sigma_n(M)),
$$
the {\it error of restricted approximation} of the class $\cU$ by elements from $\Sigma_n(M)$.
We have the following theorem.
\begin{theorem}
    \label{T:NOR}
Let $\Sigma_n$ be a set in $H^1(\Omega)$   and let $\Sigma_n(M)$ be defined by  \eref{defStilde1}for a fixed value  $M\ge 1$. If  $\widehat S$ is any function from  \eref{dL*min1} then
\be
\label{minimize3}
\|u- \widehat S\|_{H^1(\Omega)}\lesssim M \cR_{\cU}(\widetilde m,\overline m)+  E(u, \Sigma_n(M)),
\ee
with the constant in $\lesssim$  depending only on $r,d$.
Moreover, if the approximation set $ \Sigma_n(M)$ is such that  $E(\cU,\Sigma_n(M))\le C\cR_{\cU}(\widetilde m,\overline m)$  for some
$C$ depending only on $r,d$,
then $\widehat S$ provides a uniform near optimal recovery (up to  logarithmic factors in the case $d = 2$)    for the model class $\cU$.
 
\end{theorem}
\begin{proof}  
We only prove this in the case $d \geq 3$. In the case $d = 2$ the same modifications can be made as before. We use    \eref{Tdloss} for $\widehat S$  to obtain
\be\label{eq31}
 \|u-\widehat S\|_{H^1(\Omega)}\lesssim    \cL^*( \widehat S)+ (1+M) \cR_{\cU}(\widetilde m,\overline m).
\ee
 For  all  $S\in  \Sigma_n(M)$,    Lemma \ref{L:Ldiscrete} and Theorem \ref{T:dH12norm} give the bound
$$
\cL^* (\widehat S)\le   \cL^*(S)\lesssim \|\Delta u-\Delta S\|_{L_\gamma(\Omega)}+ \|Tr(u-S)\|_{H^{1/2}(\partial\Omega)}+ CM \cR_{\cU}(\widetilde m,\overline m),
$$
with the constant in $\lesssim$  depending only on $r,d$.
Taking an infimum over all $S\in  \Sigma_M$ gives  
\be\label{eq5}
  \cL^* (\widehat S)\lesssim E(u, \Sigma_n(M))+ M\cR_{\cU}(\widetilde m,\overline m),
\ee
which together with \eref{eq31} proves the theorem. 
\end{proof}
\begin{remark}
\label{Rnew}
 Theorem \ref{T:NOR}
guarantees that $\widehat S$ is a near optimal recovery of $u$ provided that for the model class $\cU$ we can make
 $E(\cU, \Sigma_n(M))$ small   by taking $n$ large.   Note that the quantity $E(\cU, \Sigma_n(M))$ determines how well the elements of the  set $\Sigma_n(M)$ simultaneously approximate the Laplacians  (in the $L_\gamma(\Omega)$ norm) and the traces  (in $H^{1/2}(\partial \Omega)$ norm) of the elements in $\cU$. Thus, the question is whether
$\Sigma_n$ has favorable restricted  approximation properties.
 There is a large literature of approximation theoretic results concerning how efficiently deep neural networks can approximate various classes of functions (see \cite{devore2021neural} and the references therein).
  When error is measured in the $L_p$-norm, optimal rates have been obtained \cite{siegel2023optimal,yarotsky2018optimal,shen2022optimal} in the case of approximating functions from Besov spaces using deep ReLU neural networks. There has also been recent work on the problem when error is measured in another Besov space, although in this case a smoother activation function must be taken \cite{yang2024nearly,yang2023nearly}. 
On the other hand, the question of restricted approximation that is used in the present paper has seemingly not been studied.
 We conjecture that the techniques developed in \cite{siegel2023optimal,yarotsky2018optimal,shen2022optimal} can be appropriately modified to give the same rates   for restricted approximation with a fixed suitably large value for $M$ (independent of $n$).
\end{remark}

\section{Numerical Experiments}\label{numerical-section}
 In this section, we present numerical experiments testing the loss functions described in \S\ref{S:implications} for the practical solution of elliptic PDEs. In our experiments, we solve the two-dimensional Poisson equation
 \be\label{poisson-equation-experiments}
    -\Delta u = f~\text{in $\Omega$},~u = g~\text{on $\partial\Omega$},
\ee
on $\Omega = (0,1)^2$ for various right-hand sides $f$ and boundary values $g$. We take a uniform grid of collocation points on $\Omega$, where $\widetilde m$ is the total number of collocation points (where we penalize the residual of the equation), and $\overline m$ is the number of collocation points which lie on the boundary (where we penalize the residual of the boundary values). We then test the following four loss functions, described in \S\ref{S:implications}:
\begin{itemize}
    \item The original (unweighted) PINNs loss: $\cL_{sq,1}$;
    \item A properly weighted PINNs loss: $\cL_{sq,\lambda(\overline{m})}$, see \eref{Lnew};
    \item The loss $\cL^*_{sq,\gamma}$ from \eqref{dloss} (since $d = 2$ we choose here $\gamma = 1.1$);
    \item The loss $\cL^*_{sq,2}$ (the loss function \eqref{dloss} with $L_2$ on the domain).
\end{itemize}
Since we are in the case $d=2$, the latter two loss functions involve a logarithm which is ignored in the numerical experiments.
For all losses, we use the same neural network and training algorithm. We report both the final loss values achieved and the final relative error in the $H^1(\Omega)$-norm (estimated using a large number of collocation points).

Let us remark on the relative cost of implementing each of these four loss functions. In all our experiments, the efficiency of optimizing each of them was roughly the same. The reason for this is that although the discretization of the $H^{1/2}$-norm on the boundary used in the losses $\cL^*_{sq,\gamma}$ and $\cL^*_{sq,2}$ is more expensive to compute and differentiate, the training process is dominated by the cost of evaluating the network and its derivatives at all of the collocation points rather than evaluating the loss based upon these values. For this reason, using losses which discretize the $H^{1/2}$-norm doesn't significantly slow down the network training.
\subsection{Network Architecture and Training Algorithm}
When training PINNs, it is known that the optimization of the neural network can be particularly sensitive and tricky \cite{rathorechallenges,krishnapriyan2021characterizing,wang2021understanding,wang2022and,de2023operator}, and much research has been devoted to the development of improved training algorithms for PINNs \cite{mcclenny2023self,davi2022pso,kopanivcakova2024enhancing,muller2023achieving,zhang2023element,wang2023expert}. In this section, we describe our network architecture, initialization, and training algorithm in detail.

In our experiments, we use a deep neural network with  $L$ hidden layers of width $W$. The first layer uses a $\tanh$ activation function, while the remaining layers are residual layers \cite{he2016deep} with the ReLU$^3$ activation function. The network is initialized as follows:
\begin{itemize}
\item The parameters in the first layer are randomly initialized from a normal distribution with variance $1$.
\item For the residual layers with ReLU$^3$ activation function, the parameters are randomly initialized from a normal distribution with variance $$
    \sqrt{2}\left(\frac{2}{15}\right)^{1/6}\frac{1}{\sqrt{LW}}.
$$
In analogy with Xavier and He initialization \cite{glorot2010understanding,he2015delving}, this ensures that the variance of the output of each layer gets multiplied by a factor of $(1 + 2/L)$ in each residual layer, which guarantees that the expected magnitude of the neural network function at initialization is bounded.
\item In the final layer,  the weights are initialized randomly from a normal distribution with variance $1/\sqrt{W}$, which ensures that the network output has the same variance as the output of the final hidden layer.
\end{itemize}

We train this network applying  the energy natural gradient descent algorithm described in \cite{muller2023achieving}, using a line search to find the step size in each iteration. The full code and all experimental details can be found at \href{https://github.com/jwsiegel2510/consistent-PINNs}{https://github.com/jwsiegel2510/consistent-PINNs}.

\subsection{Results}
 {\bf Experiment 1:} In our first experiment, we choose $f$ and $g$ so that the exact solution is given by
$$
u(x,y)= e^x \cos(y), \qquad (x,y) \in \Omega = (0,1)^2.
$$
The function is harmonic, i.e $f \equiv 0$. The function $u$ is very smooth and so for this problem it suffices to use a relatively small number of collocation points and a relatively small network. We test a grid of collocation points with $5,10,15$ and $20$ points in each direction. We use a small network ($L = 3$, $W = 5$) and iterate for $500$ steps. We present the results in Table \ref{harmonic-experimental-results}.

\begin{table}[ht!]
    \centering
    \begin{tabular}{c|c||c|c||c|c||c|c||c|c}
      \multirow{2}{*}{$\widetilde{m}$} & \multirow{2}{*}{$\overline{m}$} & $\cL_{sq,1}$ & $\cL_{sq,1}$ & $\cL_{sq,\lambda(\overline{m})}$ & $\cL_{sq,\lambda(\overline{m})}$ &$\cL^*_{sq,\gamma}$ & $\cL^*_{sq,\gamma}$ & $\cL^*_{sq,2}$ & $\cL^*_{sq,2}$\\
      & & Rel Error & Loss & Rel error & Loss &Rel error & Loss & Rel error & Loss \\
      \hline
      25 & 16 & $0.80\%$ & $2.8\cdot 10^{-5}$ & $0.68\%$ & $1.1\cdot 10^{-4}$ & $0.33\%$ & $2.5\cdot 10^{-5}$ & $0.30\%$ & $3.1\cdot 10^{-5}$\\
      \hline
      100 & 36 & $0.43\%$ & $1.7\cdot 10^{-5}$ & $0.21\%$ & $2.3\cdot 10^{-5}$ & $0.21\%$ & $2.0\cdot 10^{-5}$ & $0.25\%$ & $2.2\cdot 10^{-5}$\\
      \hline
      225 & 56 & $0.41\%$ & $1.5\cdot 10^{-5}$ & $0.15\%$ & $2.1\cdot 10^{-5}$ & $0.23\%$ & $1.7\cdot 10^{-5}$ & $0.27\%$ & $2.0\cdot 10^{-5}$\\
      \hline
      400 & 76 & $0.41\%$ & $1.7\cdot 10^{-5}$ & $0.086\%$ & $7.18\cdot 10^{-5}$ & $0.22\%$ & $1.9\cdot 10^{-5}$ & $0.33\%$ & $3.6\cdot 10^{-5}$\\
      \end{tabular}
    \caption{Experimental results on the harmonic problem. All relative errors are in the $H^1(\Omega)$-norm, computed using a grid with $500$ points in each direction. When the number of grid points is large, the properly weighted PINNs loss $\cL_{sq,\lambda(\overline{m})}$ results in significantly lower error than the other loss functions.} 
    \label{harmonic-experimental-results}
\end{table}

\noindent
{\bf Experiment 2:} In our second experiment, we choose $f$ and $g$ so that the exact solution is given by
$$
u(x,y)= \frac{e^{2(x+y)}\cos(2\pi(y-x))}{1+8x^2 + y^2}, \qquad (x,y) \in \Omega = (0,1)^2.
$$
In this case, both $f$ and $g$ are non-zero and  the solution is significantly less smooth. To capture this, we use a slightly larger network ($L = 3$ and $W = 10$). All other parameters are the same as in the first experiment. The results are presented in Table \ref{smooth-experimental-results}.

\begin{table}[ht!]
    \centering
    \begin{tabular}{c|c||c|c||c|c||c|c||c|c}
      \multirow{2}{*}{$\widetilde{m}$} & \multirow{2}{*}{$\overline{m}$} & $\cL_{sq,1}$ & $\cL_{sq,1}$ & $\cL_{sq,\lambda(\overline{m})}$ & $\cL_{sq,\lambda(\overline{m})}$ &$\cL^*_{sq,\gamma}$ & $\cL^*_{sq,\gamma}$ & $\cL^*_{sq,2}$ & $\cL^*_{sq,2}$\\
      & & Rel Error & Loss & Rel error & Loss &Rel error & Loss & Rel error & Loss \\
      \hline
      25 & 16 & $26\%$ & $1.2\cdot 10^{-7}$ & $24\%$ & $2.2\cdot 10^{-6}$ & $24\%$ & $3.3\cdot 10^{-7}$ & $22\%$ & $1.6\cdot 10^{-7}$\\
      \hline
      100 & 36 & $1.5\%$ & $3.7\cdot 10^{-3}$ & $2.9\%$ & $3.7\cdot 10^0$ & $1.1\%$ & $1.6\cdot 10^{-2}$ & $0.69\%$ & $2.0\cdot 10^{-2}$\\
      \hline
      225 & 56 & $1.4\%$ & $5.5\cdot 10^{-3}$ & $1.1\%$ & $4.1\cdot 10^{-2}$ & $1.6\%$ & $1.1\cdot 10^{-1}$ & $0.90\%$ & $1.1\cdot 10^{-2}$\\
      \hline
      400 & 76 & $2.3\%$ & $1.7\cdot 10^{-2}$ & $0.88\%$ & $3.2\cdot 10^{-2}$ & $1.5\%$ & $4.7\cdot 10^{-2}$ & $0.75\%$ & $2.2\cdot 10^{-2}$\\
      \end{tabular}
    \caption{Experimental results on the smooth problem. All relative errors are in the $H^1(\Omega)$-norm, computed using a grid with $500$ points in each direction. Aside from large errors when the number of collocation points is relatively small, the consistent loss with $L_2$-penalty on the domain, $\cL^*_{sq,2}$, performs best on this experiment. } 
    \label{smooth-experimental-results}
\end{table}

\noindent{\bf Experiment 3:} Finally, for our last experiment, we choose $f$ and $g$ so that
the exact solution is given by
$$
u(x,y)= 1000x(1-x)y(1-y) \left ( (x-0.5)^2+(y-0.5)^2 \right)^{9/4}, \qquad (x,y) \in (0,1)^2.
$$
The solution $u$ vanishes on the boundary (i.e. $g = 0$) and  is much less smooth than the preceding two experiments. For this reason, we use a significantly larger network ($L = 3$, $W = 15$), more collocation points ($10$, $20$, $30$ and $40$ in each direction), and more training iterations ($1000$ steps). The results are presented in Table \ref{non-smooth-experimental-results}. 

\begin{table}[ht!]
    \centering
    \begin{tabular}{c|c||c|c||c|c||c|c||c|c}
      \multirow{2}{*}{$\widetilde{m}$} & \multirow{2}{*}{$\overline{m}$} & $\cL_{sq,1}$ & $\cL_{sq,1}$ & $\cL_{sq,\lambda(\overline{m})}$ & $\cL_{sq,\lambda(\overline{m})}$ &$\cL^*_{sq,\gamma}$ & $\cL^*_{sq,\gamma}$ & $\cL^*_{sq,2}$ & $\cL^*_{sq,2}$\\
      & & Rel Error & Loss & Rel error & Loss &Rel error & Loss & Rel error & Loss \\
      \hline
      100 & 36 & $8.4\%$ & $9.0\cdot 10^{-4}$ & $3.8\%$ & $3.3\cdot 10^{-3}$ & $10.4\%$ & $1.0\cdot 10^{-2}$ & $18.9\%$ & $6.1\cdot 10^{-3}$\\
      \hline
      400 & 76 & $18.8\%$ & $5.2\cdot 10^{-2}$ & $1.8\%$ & $4.4\cdot 10^{-2}$ & $5.1\%$ & $2.2\cdot 10^{-2}$ & $5.3\%$ & $2.9\cdot 10^{-2}$\\
      \hline
      900 & 116 & $10.9\%$ & $1.1\cdot 10^{-2}$ & $2.4\%$ & $3.9\cdot 10^{-2}$ & $3.1\%$ & $1.9\cdot 10^{-2}$ & $2.5\%$ & $2.1\cdot 10^{-2}$\\
      \hline
      1600 & 156 & $6.1\%$ & $6.3\cdot 10^{-3}$ & $1.2\%$ & $7.5\cdot 10^{-3}$ & $11.1\%$ & $6.4\cdot 10^{-2}$ & $1.4\%$ & $1.1\cdot 10^{-2}$\\
      \end{tabular}
    \caption{Experimental results on the non-smooth problem. All relative errors are in the $H^1(\Omega)$-norm, computed using a grid with $500$ points in each direction. We see that the properly weighted PINNs loss, $\cL_{sq,\lambda(\overline{m})}$, and the consistent PINNs with $L_2$-penalty on the domain, $\cL^*_{sq,2}$, perform significantly better than both other loss functions. Interestingly, the performance of the consistent loss $\cL^*_{sq,\gamma}$ deteriorates significantly with the largest number of collocation points.} 
    \label{non-smooth-experimental-results}
\end{table}

Our numerical experiments seem to imply that the optimization method plays a critical role in the final accuracy achieved when using any of these loss functions. Indeed, there is a fairly wide variation in the computed final loss values. 
Accounting for this, in the first two tests, it seems that the correctly weighted PINNs loss $\cL_{sq,\lambda(\overline{m})}$ performs the best. The loss \eqref{dloss} with $L_2$-penalty on the residual, $\cL^*_{sq,2}$, is performing similarly well. However, 
more experiments are needed to confirm our observations.

The results of  Experiment 3 are more conclusive. The solution in this case is much less smooth and more difficult to fit using neural networks than in the first two experiments. In this experiment, we observe  that optimizing the correctly weighted PINNs loss $\cL_{sq,\lambda(\overline{m})}$ and the loss \eqref{dloss} with $L_2$-penalty on the domain, $\cL^*_{sq,2}$, result in significantly lower solution error than the other two loss functions. Interestingly, the loss $\cL^*_{sq,\gamma}$ performs nearly as well up until the final experiment where the collocation grid has $40$ points in each direction.

Our interpretation of these results  is the following. The (unweighted) original PINNs loss does not adequately fit the boundary values, while the loss \eqref{dloss} results in a less smooth numerical solution due to the optimization of the non-smooth $L_\gamma$-norm. Both the losses $\cL_{sq,\lambda(\overline{m})}$ and $\cL^*_{sq,2}$ overcome these issues by correctly weighting the boundary error and also resulting in neural network solutions which are smooth. This is corroborated by the solution plots in each case, which are shown in Figures \ref{original-consistent-PINNs-plots} and \ref{weighted-consistent-l2-plots}. Observe  that the solutions minimizing both of the losses $\cL_{sq,1}$ and $\cL^*_{sq,\gamma}$ fail to accurately fit the boundary values. However, the solution for $\cL_{sq,1}$ is relatively smooth, while the solution for $\cL^*_{sq,\gamma}$ is significantly less smooth (especially near the boundary). The solutions for both $\cL_{sq,\lambda(\overline{m})}$ and $\cL^*_{sq,2}$ are fairly smooth and fit the boundary values well.

\begin{figure}[ht!]
    \includegraphics[scale = 0.5]{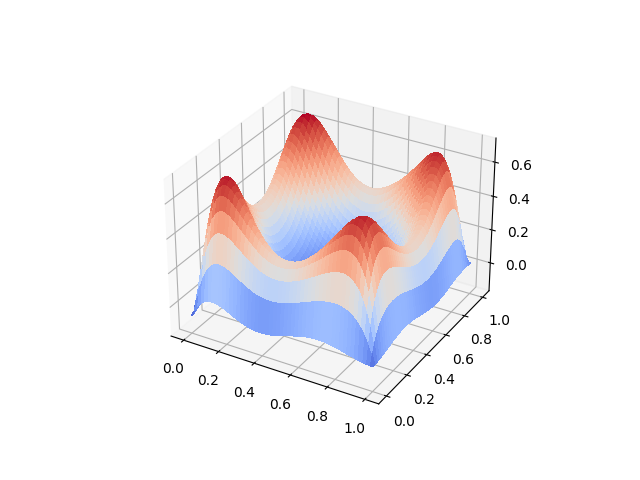}
    \includegraphics[scale = 0.5]{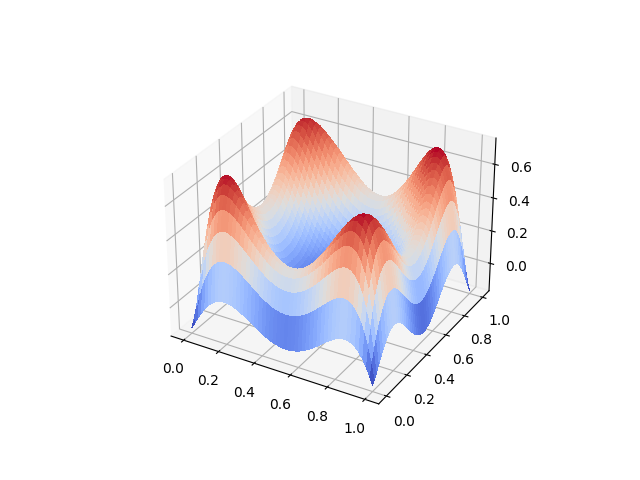}
    \label{original-consistent-PINNs-plots}
    \caption{Experiment 3: Solution plots for the (unweighted) original PINNs loss $\cL_{sq,1}$ (left) and the consistent PINNs loss $\cL^*_{sq,\gamma}$ (right). We see that both fail to accurately fit the boundary values. Note that  the consistent PINNs solution is much less smooth near the boundary.}
\end{figure}

\begin{figure}[ht!]
    \includegraphics[scale = 0.5]{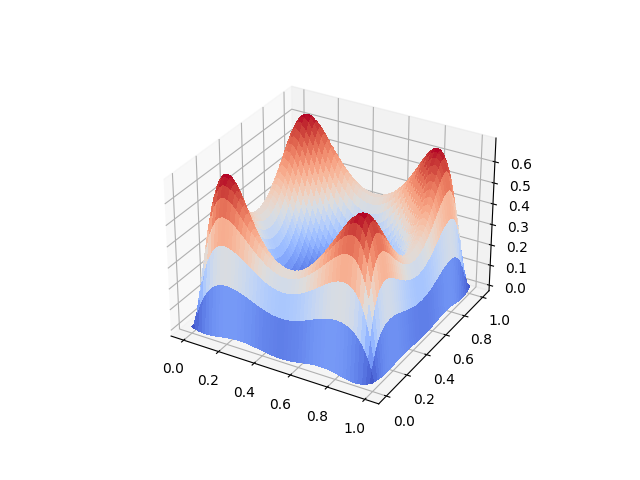}
    \includegraphics[scale = 0.5]{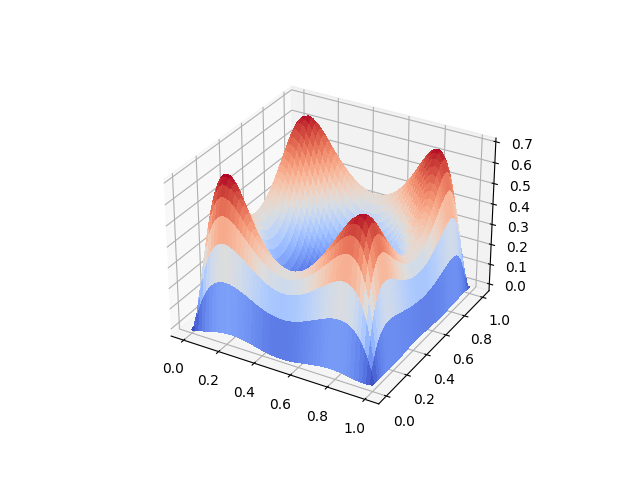}
    \label{weighted-consistent-l2-plots}
    \caption{Experiment 3: Solution plots for the correctly weighted PINNs loss $\cL_{sq,\lambda(\overline{m})}$ (left) and the consistent PINNs loss with $L_2$-penalty on the domain $\cL^*_{sq,2}$ (right). We see that both accurately fit the boundary values and remain fairly smooth.}
\end{figure}

Based upon our numerical tests, we recommend using either the correctly weighted PINNs loss $\cL_{sq,\lambda(\overline{m})}$, see \eref{Lnew},  or the loss \eqref{dloss} with $L_2$-penalty on the domain, $\cL^*_{sq,2}$, especially for difficult problems with limited regularity. However, our observations are naturally quite limited and significant additional experiments are needed to draw stronger conclusions.

\section{Concluding Remarks}
\label{S:CR}
We have investigated collocation methods for solving  elliptic PDEs of the form \eref{main}.    
Under model class assumptions of the form $f\in\cF$
and $g\in \cG$, which are both unit balls of Besov spaces, we showed that there is an optimal
error (called the error of optimal recovery) in recovering $u$ in the $H^1(\Omega)$ norm. We determined this optimal recovery error for the various model classes. Our theoretical analysis was performed for $\Omega = (0,1)^d$,  but it is well-known that the results  carry over to any Lipschitz domain (see \cite{NT}).

We then turned to the study of theoretical
algorithms which would yield an optimal recovery.  The most prominent of these methods is to
minimize a certain loss function $\cL$ that measures how well we have fit the data observations.  While the typical loss function used in PINNs does not provide an a posteriori bound on the error, we have introduced a variety of modified loss functions which do provide such a bound. We call the use of any one of these theoretically justified loss functions consistent PINNs. Finally, we perform several  numerical experiments which compare these different loss functions and show their favorable performance 
in comparison to the typical loss used in PINNs.


We now wish to make some further comments that one can view as serious caveats to
the above theory.  The first and most serious objection to the above development is
that we do not provide a numerical algorithm for finding a minimizer or near minimizer
of the suggested losses over $\Sigma_n(M)$.  This is the same objection that can be
made for almost all learning problems and certainly for PINNs.  The usual approach
in PINNs 
is to use gradient descent or some modification of gradient descent to solve the minimization of the loss.  Unfortunately, there is no proof of convergence or accuracy when applying gradient descent to non-convex losses such as $\cL$ or $\cL^*$.  This has not impeded the use of gradient descent in practice with much empirical success. We have utilized such techniques as well in our numerical examples.

We do not  enter into a discussion of the art of using gradient descent in optimization.  However, we want to address two issues that are relevant to this paper.   The first is the fact that we must optimize the loss over $\Sigma_n(M)$ rather than $\Sigma_n$ itself to gain the proof of optimality.  This is a serious numerical burden.  We have ignored this issue in our numerical experiments much like one ignores the lack of a provable
gradient descent algorithm.  One could implement the restricted approximation 
by adding a penalty to the consistent losses to guarantee that at each step of the optimization one remains in the constrained set $\Sigma_n(M)$ (see \cite{binev2024optimal}) but this does not ease the computational burden and one is still faced with a lack of convergence of the numerical optimization.

The second issue is understood by practitioners but not usually pointed out in the theory. Neural network spaces are a very unstable manifold in the sense  that changing parameters very slightly can have a huge
effect on the output of the neural network.  That is, the mapping from parameters to output in neural networks is very unstable.  This instability appears more in deep networks than in shallow networks.  Avoiding this lack of stability is an art in practice that manifests itself in judiciously choosing the starting parameters
and the learning rates (step size in gradient descent).  The theoretical aspects of
this lack of stability, and how to avoid the instability theoretically, is addressed in a series of recent papers \cite{CDPW, PW, PW1,PW2}.  However, there remains a serious gap between theoretical algorithms for learning and their numerical implementation.

\section{Appendix}

In this appendix, we provide proof of the results stated in sections $\S\ref{S:Besov}$
and \S\ref{S:OR}.

\subsection{Local approximation by polynomials}
\label{SS:polyapprox}

There are many important results on the approximation of functions in Besov classes. We will  use approximation by  piecewise polynomials.  We begin by describing local
polynomial approximation.  

For any integer $r\ge 1$, we let $\cP_r^d := \cP_r$ denote the linear space of algebraic polynomials of order $r$ (total degree $r-1$), namely, 
$$
\cP_r:=\left\{\sum_{|{\bk}|_1<r}a_{\bk}\bx^{\bk}, \,\,a_{\bk}\in\R\right\}, 
\quad \hbox{where}\quad \bx^{\bk}:=x_1^{k_1}\cdots x_d^{k_d}, \quad \bk:=(k_1,\dots,k_d), \quad k_j\geq 0,
\quad |\bk|_1:=\sum_{j=1}^d k_j.
$$
Note that $\omega_r(P,t)_{p}=0$, $t\ge 0$, for all $P\in\cP_r$.
If $I$ is any cube in $\R^d$ and $f\in L_p(I)$, $0<p\le \infty$,  we denote by 
\be 
\label{Er}
E_r(f,I)_p:=\inf_{P\in\cP_r}\|f-P\|_{L_p(I)},
\ee 
 the error of approximation of $f$ on $I$  in the $\|\cdot\|_{L_p(I)}$ norm by the elements of $\cP_r$.  A well known result in approximation theory, commonly referred to as Whitney's theorem, says
that for any $f\in L_p(I)$ with $I$ a cube with sidelength $\ell_I$, we have
\be 
\label{WTheorem}
cE_r(f,I)_p\le \omega_r(f,\ell_I)_{L_p(I)}\le CE_r(f,I)_p,
\ee 
with the constants $c,C$ depending only on $r,d$ and also $p$  if  $p$  is close to $0$.  Whitney's theorem usually only refers to the lower inequality in \eref{WTheorem}.  However, the upper inequality follows trivially since 
$$\omega_r(f,\ell_I)_{L_p(I)}=\omega_r(f-P,\ell_I)_{L_p(I)}\le C\|f-P\|_{L_p(I)},
$$
holds for any polynomial $P\in\cP_r$.

We also will  use the following modified form of Whitney's theorem.  For any cube $I\subset \R^d$,
any $0<p<\infty$, and $f\in L_p(I)$,  we define 
\be 
\label{avemod}
\widetilde w_r(f,t)^p_{L_p(I)}:=t^{-d}\int_{{\rm \bf h}\in[-t,t]^d}\int_{I_{{\color{blue} r}\rm\bf h}} |\Delta_{\rm\bf h}^r(f,\bx)|^p\, d{\rm\bf x}\, d{\rm \bf  h},
\ee 
where $I_{ r{\rm \bf h}}:=\{{\rm\bf x}: [{\rm\bf x},{\rm\bf x}+{r}{\rm\bf h}]\subset I\}$.  This is called {\it the averaged modulus of smoothness of $f$}. It is   known that $w_r$ is equivalent to  $\widetilde w_r$ (see \S 2 of \cite{DS}),
\be 
\label{ave,pd1}
c\widetilde w_r(f,t)_{L_p(I)}\le \omega_r(f,t)_{L_p(I)}\le C\widetilde w_r(f,t)_{L_p(I)},\quad 0<t\leq 1,
\ee 
where again the constants $c,C$ depend only on $r$ and $p$ and can be taken absolute when $r$ is fixed and $0<p_0\le p\le \infty$ with $p_0$ fixed. Thus,   Whitney's theorem holds with $\omega_r$ replaced by $\widetilde w_r$
\be 
\label{WT1}
cE_r(f,I)_p\le \widetilde w_r(f,\ell_I)_{L_p(I)}\le CE_r(f,I)_p.
\ee 

Before proceeding further, let us remark on why we introduce the averaged modulus of smoothness $\widetilde w_r$.  
The advantage of $\widetilde w_r$ over $\omega_r$ is that $\widetilde w_r^p$ it is set subadditive. We shall use this in the following form.  Let $\Omega=(0,1)^d$ and let $\cI$ be a collection of subcubes of $\Omega$ which form a partition of $\Omega$.  Then, from the definition of $\widetilde w_r$. we have 
    \be
    \label{wsubadd}
\sum_{I\in\cI} \widetilde w_r(f,t)^p_{L_p(I)}\le \widetilde w_r(f,t)^p_{L_p(\Omega)},\quad t>0.
\ee 
This same subadditivity holds when  $\cI$ is replaced by
a set $\cT$ of simplices which form a partition   of $\Omega$.

If $I\subset\Omega$ is a cube, we  say that $P_I\in\cP_r$ is a near best $L_p(I)$ approximation to $f$ with constant $\lambda\ge 1$ if
\be 
\|f-P_I\|_{L_p(I)}\le \lambda E_r(f,I)_p.
\ee 
It is shown in Lemma 3.2 of \cite{DP} that if $P_I\in\cP_r$ is near best in $L_{ p}(I)$ with constant $\lambda$, then it is also near best in $L_{\overline p}(I)$ whenever $\overline p\ge p$,
i.e.,
\be 
\|f-P_I\|_{L_{\overline p}(I)}\le C\lambda E_r(f,I)_{\overline p},
\ee 
with the constant $C$ depending only on $r,d$ and $p$. 
 This constant does not depend on $I$ or $\bar p$.

Another important remark is that any near best approximation $P_I$ with constant $\lambda$ is near best on any larger cube $J$ which contains $I$ in the sense that
\be 
\|f-P_I\|_{L_{ p}(J)}\le C\lambda E_r(f,J)_{ p},
\ee 
where now $C$ depends additionally on $|J|/|I|$,  see Lemma 3.3 in \cite{DP}. Note that even though Lemma 3.2 and Lemma 3.3 in \cite{DP} are stated for polynomials of coordinate degree $<r$, they also hold for polynomials of total degree $<r$.   

In summary, a near best $L_p(I)$
approximation is also near best on large cubes $J$ containing $I$  and larger values $\overline p\ge p$.
We shall use these facts going forward.

\subsection{Polynomial norms and inequalities}
\label{SS:polynorms}
We need good bounds on the constants that appear when comparing norms.  For this, we recall the following equivalences, see  (3.2) of
\cite{DS}   or Lemma 3.1 and Lemma 3.2 in \cite{DSmono}.  For any cube $I$ in $\R^d$ and any function $g\in L_p(I)$, we introduce the normalized $L_p$ norm  
 \be
\label{npnorm}
\|g\|^\#_{L_p(I)}:= |I|^{-1/p}\|g\|_{L_p(I)}.
\ee
For any cube $I\subset \R^d$ and any polynomial $P\in\cP_r$, and any $0<p,q\le \infty$,
we have
\be 
\label{eqPrnorm1}
  \|P\|^{\#}_{L_q( I)}\asymp  \|P\|^{\#}_{L_p(I)} ,
 \ee 
 with absolute constants of equivalency provided   $  q,p\ge p_0$ with $p_0>0$ fixed.

 One can also compare Besov norms of polynomials.  For example, we will use the fact that if   $P\in \cP_\ell$ and $I$ is a cube in $\R^d$,  then for every $s>0$, we have
 \be 
 \label{Bpoly}
 |P|_{B_q^s(L_p(I))} \le C \ell_I^{-s}\|P\|_{L_p(I)},
 \ee
 with a constant independent of $I$. This is proved  by dilation (see e.g. Corollary 5.2 of \cite{DP}).

\subsection{Besov spaces and piecewise polynomial approximation}
\label{Bpwp}

We now recall how membership in $ B_p^s(\Omega)$ guarantees a rate of approximation by piecewise polynomials. In actuality, the membership of a function in a Besov space can be characterized by its rate of approximation by piecewise polynomials.  However, we only need the results that prove that a function in a Besov space can be approximated by piecewise polynomials with a certain accuracy
and therefore only concentrate on proving results of this type.

  Recall that for  $k\ge 0$, we denoted by $\cD_k$ the  partition of $\Omega$ into  dyadic cubes $I$ of side length $2^{-k}$, where the dyadic cube $I\in\cD_k$  is the tensor product
  of the dyadic intervals $[(j-1)2^{-k},j2^{-k})$, $1\le  j\leq 2^k$.
  We   then define 
$$
\cS_k=\cS_k(r)
$$ 
to be the space of all piecewise polynomials of order $r$ that are subordinate to the partition $\cD_k$. In other words, a  function $S\in \cS_k$ if and only if the restriction of $S$ on each $I\in \cD_k$ belongs to 
$\cP_r$.   

If $I\in\cD_k$,   let $P_I\in\cP_r$ be the polynomial of best $L_p(I)$ approximation to $f$  and let us define 
 \be 
 \label{defSk1}
 S_k:=S_k(f):=\sum_{I\in\cD_k}P_I\chi_I\in \cS_k,
 \ee 
 where $\chi_I$ is the characteristic function of $I$.  Notice that $S_k(x)$ is defined pointwise for each
 $x\in\Omega$.
We extend $S_k$ to all of $\overline\Omega$ by continuity.

The following lemma holds.

\begin{lemma}
    \label{approx}
 Let $0<p\le \infty$, $s>0$, and let $r\geq 2$ be any fixed integer strictly larger than $s$. If a  function 
 $f\in B_p^s=B_\infty^s(L_p(\Omega))$, with $\Omega=(0,1)^d$, then 
 \be 
   \label{Tbesovpwpchar}
   \dist(f,\cS_k(r))_{L_p(\Omega)}\le C|f|_{B_p^s(\Omega)}2^{-ks},\quad k\ge 0,
   \ee 
 with $C$ depending only on $p$ and $s$. 
 \end{lemma}
 \begin{proof}
 Consider $S_k$ defined in \eref{defSk1}. 
 From Whitney's theorem, we have
  \be 
 \label{pwp1}
 \|f-P_I\|_{L_p(I)}\le C\omega_r(f,2^{-k})_{L_p(I)}\le C\widetilde w_r(f,2^{-k})_{L_p(I)},
 \ee 
 where $\widetilde w_r$ is the averaged modulus of smoothness.  
 
 If $p<\infty$, we raise both
 sides of \eref{pwp1} to the power $p$ and then sum over $I\in \cD_k$ to obtain
 \be 
 \label{pwp2}
 \|f-S_k\|^p_{L_p(\Omega)}\lesssim \sum_{I\in\cD_k}\widetilde w_r(f,2^{-k})^p_{L_p(I)}\lesssim \widetilde w_r(f,2^{-k})^p_{L_p(\Omega)}\lesssim w_r(f,2^{-k})^p_{L_p(\Omega)}\lesssim |f|^p_{B_p^s(\Omega)}2^{-ksp},
 \ee 
 where we used the subadditivity \eref{wsubadd} and the fact that $\cD_k$ is a partition of $\Omega$. 
Here we have also used the equivalence \eref{ave,pd1} of $\widetilde w_r$ and $w_r$ and \eref{member1}. This proves  \eref{Tbesovpwpchar} in the case
 $p<\infty$.  
  When $p=\infty$, these inequalities   follow directly from \eref{pwp1}   and the fact that   $\|\cdot\|_{L_\infty(I)}\le \|\cdot\|_{L_\infty(\Omega)} $ for each $I\subset \Omega$.
\end{proof}

 \subsection{Piecewise polynomial approximation in $L_\tau(\Omega)$}
\label{SS:pwpL2}
 
All constants $C$ appearing in this section depend at most on $s$ and $p$
and may change at each occurrence.

\begin{theorem}
    \label{T:SkLtau} Let $S_k$ be defined as in {\rm \eref{defSk1}}.
    If $s>d/p$ and $p\le \tau\le \infty$, then  we have
    \be 
    \label{TSkLtau}
    \|f-S_k\|_{L_\tau(\Omega)}\le C |f|_{B_p^s(\Omega)} 2^{-k(s-d/p+d/\tau)},\quad f\in B_p^s(\Omega).
    \ee
\end{theorem}
\begin{proof} 
Let us fix any $f\in B_p^s(\Omega)$ and consider the 
corresponding $S_k=S_k(f)$, see   \eref{defSk1}.   It was proven in Lemma \ref{approx}  that
$$
\|f-S_k\|_{L_p(\Omega)}\le C|f|_{B_p^s(\Omega)}2^{-ks},\quad k\ge 0.
$$
Let $R_0:=S_0$ and for each $\bx\in\overline\Omega$, we define
\be 
\label{defRk}
R_k(\bx):=S_k(\bx)-S_{k-1}(\bx), \quad k\geq 1. 
\ee 
The functions $R_k$ are defined for all $\bx\in\overline\Omega$ and are  in $\cS_k$, $k\ge 0$,  
\be 
\label{Tk}
\|R_k\|_{L_p(\Omega)}\le \|f-S_k\|_{L_p(\Omega)}+\|f-S_{k-1}\|_{L_p(\Omega)}\le C|f|_{B_p^s(\Omega)}2^{-ks},\quad k\ge 1,
\ee
and of course $\|R_0\|_{L_p(\Omega)}\le C\|f\|_{L_p(\Omega)}$.
It follows that
\be 
\label{repf1} 
f=\sum_{k=0}^\infty R_k
\ee
with the series converging in $L_p(\Omega)$.
We now consider the following cases for $\tau$.

\noindent
{\bf Case 1:} $\tau=\infty$.
   On each dyadic cube $I\in\cD_k$, we have 
  \be 
\nonumber
  R_k(\bx)= Q_I(\bx):=P_I(\bx)-P_{\overline I}(\bx),\quad \bx\in I,
  \ee 
  where $\overline I\in\cD_{k-1}$ is the parent of 
  $I$.   From Whitney's theorem and \eref{eqPrnorm1},
we have for every $I\in \cD_k$, 
$k\geq 1$,
\begin{eqnarray}
\label{Tk0}
\|Q_I\|_{C(I)}&\le& C|I|^{-1/p}\|Q_I\|_{L_p(I)}
\le C2^{kd/p}[\|f-P_I\|_{L_p(I)}+ \|f-P_{\overline I}\|_{L_p(I)}]
\nonumber\\
&\le& C 2^{kd/p} \omega_r(f,2^{-k+1})_{L_p(\overline I)}.
\end{eqnarray} 
This  gives that for each $\bx\in\overline\Omega$, we have
\be
\label{Tk1}
|R_k(\bx)| \le C 2^{kd/p} \omega_r(f,2^{-k+1})_{L_p(\Omega)} \le C|f|_{B_p^s(\Omega)}2^{-k(s-d/p)} , \quad   k\ge 1.
\ee 
Thus the series
\eref{repf1} converges in $L_\infty(\overline\Omega)$ and also pointwise to a limit function $\widetilde f $ and for each $\bx\in \overline \Omega$, we have
\be
\label{Tk2}
|\widetilde f(\bx)-S_k(\bx)| \le \sum_{j>k}|R_j(\bx)| \le C|f|_{B_p^s(\Omega)}2^{-k(s-d/p)} ,\quad  \ k\ge 0.
\ee
Since the same bound holds for $\| f-S_k\|_{L_\infty(\overline\Omega)}$, we have $f=\widetilde f$, a.e. on $\overline\Omega$, 
and we have proven the theorem in the case $\tau=\infty$.

In going further, we refer to $\widetilde f$, which is defined pointwise on $\overline\Omega$ as the representer of $f$.
We shall see in Theorem \ref{L:cont} that $\widetilde f$  is continuous and in Lip $\alpha$, $\alpha=s-d/p$.

\noindent
{\bf Case 2:} $p\leq \tau<\infty$.
Similarly to \eref{Tk0} and using the comparison of polynomial norms of \eref{eqPrnorm1}, we find that
\begin{eqnarray} 
\nonumber
\int_\Omega |R_k|^\tau &=&\sum_{I\in\cD_k}\int_I |Q_I|^\tau 
 \le C^\tau 2^{kd(\tau/p-1)} \sum_{I\in\cD_k} \|Q_I\|_{L_p(I)}^\tau
\le C^\tau 2^{kd(\tau/p-1)} \sum_{I\in\cD_k}  \widetilde w_r(f,2^{-k+1})_{L_p(\overline I)}^\tau \nonumber\\
&\le& C^\tau 2^{kd(\tau/p-1)} \left [\sum_{I\in\cD_k}  \widetilde w_r(f,2^{-k+1})_{L_p({\overline I})}^{p}\right ]^{\tau/p}
 \le 2^{d\tau/p}C^\tau 2^{kd(\tau/p-1)} \left [\sum_{\overline I\in\cD_{k-1}}  \widetilde w_r(f,2^{-k+1})_{L_p(  \overline I)}^{p}\right ]^{\tau/p}
\nonumber\\
&\le&  2^{d\tau/p}C^\tau 2^{kd(\tau/p-1)}\widetilde w_r(f,2^{-k+1})^\tau_{L_p(\Omega)} , 
\nonumber
\end{eqnarray} 
where we used the subadditivity \eref{wsubadd}.  In other words, for any $\tau\ge p$, we have
\be
\label{Tk4}
\|R_k\|_{L_\tau(\Omega)}\le C 2^{kd(1/p-1/\tau)} |f|_{B_p^s(\Omega)}2^{-ks},\quad 
 k\ge 1.
\ee
Since $f-S_k=\sum_{j> k}R_j$, when $\tau\ge 1$, we can add these estimates to arrive at
\be 
\nonumber
\|f-S_k\|_{L_\tau(\Omega)}\le \sum_{j> k}\|R_j\|_{L_\tau(\Omega)}\le  C|f|_{B_p^s(\Omega)}2^{-k(s-d/p+d/\tau)},
\ee 
which is the desired inequality.  When $p<\tau<1$, we use \eref{Tk4} to obtain
$$
\|f-S_k\|^\tau_{L_\tau(\Omega)}\leq 
\sum_{j>k}\|R_j\|^\tau_{L_\tau(\Omega)}\leq C^\tau|f|^\tau_{B^s_p(\Omega)}\sum_{j>k}2^{-j\tau(s-d/p+d/\tau)}\leq C^{\tau}|f|^\tau_{B^s_p(\Omega)}2^{-k\tau(s-d/p+d/\tau)},
$$
 which completes the proof in this case.
\end{proof}

 \begin{remark} 
 \label{R:repf1}
 Theorem {\rm \ref{T:SkLtau}} is valid for more general  $S_k\in \cS_k$, for example, for $S_k=\sum_{I\in\cD_k}P_I\chi_I$, where 
 \be 
 \label{provided}
 \|f-P_I\|_{ L_\tau(I)}\le C2^{kd(1/p-1/\tau)}  \widetilde w_r(f,2^{-k})_{L_p(I)}, \quad I\in\cD_k(\Omega),
\ee
 where $C$ does not depend on $I$, see 
  \eref{Tk0}.
 \end{remark}

 \subsubsection{Proof of Theorem \ref{L:cont}}
 \label{T2.3}
For any $f\in B_p^s(\Omega)$, consider the corresponding $S_k(f)$, defined in \eref{defSk1} and  the function $\widetilde f$, see \eref{Tk2}. As shown  in Theorem \ref{T:SkLtau}, $f=\widetilde f$ a.e..
We want to observe that the function $\widetilde f$
is continuous on $\Omega$ and moreover,
\eref{fcont} holds.
  Indeed, it is enough to consider $0<t\le 1$.    Let $\bx\in\Omega$ and $|\bh|\le t$ be such that $[\bx,\bx+r\bh]\subset \Omega$.  Let $J$ be the smallest   cube that contains $[\bx,\bx+r\bh]$. Then, we have   $ \ell_J\le rt$.  We argue as in
the proof of \eref{Tk1} and \eref{Tk2}, replacing $\Omega$ by $J$ to show that there is a polynomial $P_J\in\cP_r$ such that
$$\sup_{{\bf x}\in J}|\widetilde f({\bf x})-P_J({\bf x})| \le 
C|f|_{B_p^s(J)}\ell_J^{s-d/p}\le C|f|_{B_p^s(\Omega)}t^{s-d/p}. 
$$
  It follows that
  $$|\Delta_{r\bh}(\widetilde f,{\bf x})|\le C|f|_{B_p^s(\Omega)}t^{s-d/p},$$
  uniformly for ${\bf x}\in\Omega$. This proves that $\widetilde f$ is continuous and \eref{fcont} holds.  
  \hfill $\Box$

 \subsection{Polynomial interpolation}
 \label{SS:PI}
In this section, we prove the results for piecewise polynomial approximation stated in \S \ref{S:Besov}.
\subsubsection{Proof of Theorem \ref{T:interapprox}}
\label{T2.1}

 In view of Theorem \ref{T:SkLtau}, it is enough to show that for any fixed $k\ge 0$, we have
\be
\label{enough0}
\|S_k(f)-S_k^*(f)\|_{L_\tau(\Omega)}\le C|f|_{B_p^s(\Omega)} 2^{-k(s-d/p+d/\tau)},
\ee 
with $C$ here and for the remainder of the proof always denoting a constant not depending on $f$ or $k$,   but depending on $r$ and $d$.
If $I\in\cD_k$,  we let $P_I$ be the best $L_p(I)$ approximation to $f$ by elements of $\cP_r$ and
similarly let $P_I^*$ be the best $C(I)$ approximation to $f$ from $\cP_r$. Recall that by definition $S_k(f)=\sum_{I\in\cD_k}P_I\chi_I$. It follows that if $T\in \cT_k$ and $T\subset I$, then we have
\be 
\label{follows0}
\|S_k^*(f)-S_k(f)\|_{L_\infty(T)} \le\|L_T(f-P_I)\|_{L_\infty(T)}\le C\|f-P_I\|_{L_\infty(T)}\leq C\|f-P_I\|_{L_\infty(I)},
\ee
and therefore 
 \be 
\label{follows1}
\|S_k^*(f)-S_k(f)\|_{L_\infty(I)} \le  C\|f-P_I\|_{L_\infty(I)}.
\ee
We know from Theorem \ref{T:SkLtau}  that $\|f-P_I\|_{L_\infty(I)}\le C|f|_{B_p^s(\Omega)}2^{-k(s-d/p)}$ for each $I\in\cD_k$, and thus  
\be 
\label{enough1}
\|S_k(f)-S_k^*(f)\|_{L_\infty(\Omega)}\le C |f|_{B_p^s(\Omega)} 2^{-k(s-d/p)},
\ee 
which completes the  proof of the theorem when 
when $\tau=\infty$.

To handle the case $p\le \tau<\infty$, we first  give an improved bound for $\|f-P_I\|_{L_\infty(I)}$ when $I\in\cD_k$.  Let $\bx\in I$ be any fixed point.   For each index $j\ge k$, let  $J_j\in\cD_j$ be the dyadic cube that contains $\bx$ and $Q_j:=  P_{J_{j+1}}-P_{J_j}$, $j\ge k$,  $I=J_k$.  Then, arguing as in \eref{Tk0}, we have
\be 
\label{follows2}
| f(\bx)-P_I(\bx)|\le \sum_{j\ge k}\|Q_j\|_{L_\infty(J_{j+1})}\le C\sum_{j\ge k}|J_{j+1}|^{-1/p}\|Q_j\|_{L_p(J_{j+1})}\le C\sum_{j\geq k} 2^{jd/p}\|f-P_{J_j}\|_{L_p(J_j)}.
\ee
Since $\|f-P_{J_j}\|_{L_p(J_j)}\le \widetilde w_r(f,2^{-j})_{L_p(I)}$, this gives
\be 
\label{follows3}
\|f-P_I\|_{L_\infty(I)} \le   C\sum_{j\ge k} 2^{jd/p}\widetilde w_r(f,2^{-j})_{L_p(I)},
\ee
which is the improved bound we want.

We now consider two cases.  
\vskip .1in
\noindent {\bf Case $p\le 1$:}
From \eref{follows1} and \eref{follows3} we derive  
\begin{eqnarray}
\|S_k(f)-S_k^*(f)\|_{L_\infty(I)}&\le& C\|f-P_I\|_{L_\infty(I)} \le   C\sum_{j\ge k} 2^{jd/p}\widetilde w_r(f,2^{-j})_{L_p(I)}\nonumber \\
&\le&  C\left [\sum_{j\ge k} 2^{jd}\widetilde w_r(f,2^{-j})^p_{L_p(I)}\right ]^{1/p}.
\end{eqnarray}
The set subadditivity \eref{wsubadd} of $\widetilde w_r$ then gives
\begin{eqnarray} 
\label{follows4}
\|S_k(f)-S_k^*(f)\|^p_{L_p(\Omega)}&\le &  C2^{-kd} \sum_{I\in\cD_k} \|S_k(f)-S_k^*(f)\|^p_{L_\infty(I)}  
\le   C2^{-kd} \sum_{j\ge k} 2^{jd}\sum_{I\in\cD_k} \widetilde w_r(f,2^{-j})^p_{L_p(I)}\nonumber\\ 
&\le& C2^{-kd} \sum_{j\ge k} 2^{jd}\widetilde w_r(f,2^{-j})^p_{L_p(\Omega)}
\le C|f|^p_{B_p^s(\Omega)}2^{-kd} \sum_{j\ge k} 2^{jd}2^{-jsp}\le C|f|^p_{B_p^s(\Omega)}2^{-ksp}.
\nonumber
\end{eqnarray}
This proves \eref{enough0} for the case $\tau=p$.  For general $\tau \in(p,\infty)$, we use \eref{enough1} and obtain
\be 
\label{follows5}
 |S_k(f)-S_k^*(f)|^\tau\le   |S_k(f)-S_k^*(f)|^{\tau-p} |S_k(f)-S_k^*(f)|^p \le C|f|^{\tau-p}_{B_p^s(\Omega)}2^{-k(s-d/p)(\tau-p)} |S_k(f)-S_k^*(f)|^p.
\ee
We now integrate both sides of this inequality and find that
\be 
\label{follows51}
 \|S_k(f)-S_k^*(f)\|^\tau_{L_\tau(\Omega)}\le   C|f|^\tau_{B_p^s(\Omega)} 2^{-k(s-d/p)(\tau-p)}2^{-ksp}\le   C|f|^\tau_{B_p^s(\Omega)} 2^{-k(s-d/p+d/\tau)\tau}.
\ee
as desired.

\vskip .1in

\noindent
{\bf Case $p\ge 1$:}  From \eref{follows3}, we find for any $\beta\in (0,s-d/p)$ that
\be 
\label{follows6}
 \|f-P_I\|_{L_\infty(I)} \le   C\{\sum_{j\ge k}  2^{-j\beta p'}\}^{1/p'} \{\sum_{j\ge k}  [2^{j(d/p+\beta)}\widetilde w_r(f,2^{-j})^p_{L_p(I}]^p\}^{1/p}, 
\ee
where $p'$ is the conjugate index to $p$, i.e., $1/p+1/p'=1$.  Arguing as in \eref{follows4}, we arrive at

\be 
\label{follows7}
 \|S_k(f)-S_k^*(f)\|_{L_p(\Omega)} \le C|f|_{B_p^s(\Omega)} 2^{-ks}.
\ee
This proves \eref{enough0} for $\tau=p$. In other words, we know \eref{enough0} for $\tau=p$ and $\tau=\infty$. 
 We complete the proof  for general $\tau\in [p,\infty)$ as in \eref{follows51}.  This completes the proof of the theorem.\hfill $\Box$

\subsubsection{Proof of Theorem \ref{T:SkH1}}
\label{T2.2}

    Since $B_q^s(L_p(\Omega))$ is embedded in $B_p^s(\Omega):=B_\infty^s(L_p(\Omega))$, it is enough to prove the theorem when $f\in B_p^s(\Omega)$, i.e., when $q=\infty$.   We know from Theorem \ref{T:interapprox} that
   \be\label{knowH1}
   \|f-S^*_k(f)\|_{L_2(\Omega)}\le C |f|_{B_p^s(\Omega)}2^{-k(s-d/p+d/2)}, \quad k\ge 0.
   \ee 
   We define 
   \be 
   \label{defTk*}
   R_k^*(f):=S_k^*(f)-S_{k-1}^*(f), \quad   k\ge 1.
   \ee
   It follows from \eref{knowH1} that
   \be 
   \label{Tk*}
   \|R_k^*(f)\|_{L_2(\Omega)}\le C |f|_{B_p^s(\Omega)}2^{-k(s-d/p+d/2)}, \quad   k\ge 1.
   \ee 
   Since on each simplex $T\in \cT_{k}$,
   the function $R_k^*(f)$ is a polynomial $Q_T$ from $\cP_r$, we have
   \be\label{knowH11}
   \|Q_T\|_{H^1(T)}\le C2^k\|Q_T\|_{L_2(T)}, \quad T\in\cT_{k}.
   \ee 
   Since the function $R_k^*(f)$ is continuous on $\Omega$, we have (see Theorem 2.1.2 in \cite{ciarlet2002finite})
   $$
   \|R_k^*(f)\|_{H^1(\Omega)}^2\le  C\sum_{T\in\cT_k} \|R_k^*(f)\|_{H^1(T)}^2\le C2^{2k}\sum_{T\in\cT_k}\|R_k^*(f)\|^2_{L_2(T)}= C2^{2k}\|R_k^*(f)\|^2_{L_2(\Omega)}.
   $$
   This gives
   \be
   \label{knowH121}
   \|R_k^*(f)\|_{H^1(\Omega)}\le C2^k\|R_k^*(f)\|_{L_2(\Omega)}\le C |f|_{B_p^s(\Omega)}2^{-k(s-1-d/p+d/2)}, \quad {\quad k\ge 1}.
   \ee 
   Writing, 
   $$
   f-S_k^*(f)=\sum_{j>k}R_j^*(f),
   $$
   and then applying $H^1$ norms gives the theorem because
   $s-d/p+d/2>1$. 
   \hfill $\Box$ 
   
  \subsection{The proofs of the Theorems on optimal recovery} 
   
 This subsection will be devoted to the proof of  the theorems on optimal recovery stated in \S\ref{S:OR}. 

 \subsubsection{The proof of Theorem \ref{T:ORf}}
 \label{T3.1}
 
 It will be convenient to consider only the case $d\ge 2$ to avoid changes in notation.  We leave the case $d=1$ to the reader. The proofs are divided into
  two parts: a proof of the upper bounds and then a proof of the lower bounds.  
\vskip .1in
\noindent 
{\bf Proof of the upper bounds in Theorem \ref{T:ORf}:}  A proof of (i) and (ii) can already be found in \cite{NT}, but  for completeness, we give the following   proof using the results of
\S \ref{S:Besov}.  Recall our  notation $B_p^s(\Omega):=B_\infty^s(L_p(\Omega))$ and the fact that all other 
Besov spaces $B_q^s(L_p(\Omega))$, with the same values for $p$ and $s$, are contained in $B_p^s(\Omega)$.  Therefore, it is enough to prove the
upper bounds when $\cB=B_p^s(\Omega)$.  To prove these upper bounds, we need to exhibit a set of $\widetilde m$ data sites at which we achieve the claimed bound.  We take these data sites as the tensor product grid points $G_{k,r}$ of $\overline\Omega$ where $r > \max(s,1)$.  This will yield the upper bounds when $\widetilde m=(2^kr)^d$.  The upper bounds for the other values of $\widetilde m$ follow from the fact that $R_{\widetilde m}^*$ is monotonically nonincreasing as $\widetilde m\to \infty$ 
\vskip .1in
\noindent{\bf The cases (i) and (ii):} Let $\cT_{k}=\cT_{k}(\Omega)$ be the simplicial partition of $\Omega$ into simplices $T$ as in \S\ref{SS:PI}.  We have shown in    that section, that from the data $(f_i)_{i=1}^{\widetilde m}$ we can create a continuous piecewise polynomial 
\be  
\label{defS2}
S=S_k^*(f)=\sum_{T\in\cT_{k}}L_T(f)\chi_T,
\ee 
where $L_T(f)$ is a polynomial of order $r$ ($L_T(f)\in\cP_r$), gotten by interpolating $f$ at    the data points $T\cap G_{k,r}$.  We have shown (see \eref{Tpwpinter}) that if $\tau\ge p$ then
$$\|f-S\|_{L_\tau(\Omega)}\preceq |f|_{B_p^s(\Omega)}  2^{-kd\alpha_\tau}\preceq |f|_{B_p^s(\Omega)} m^{-\alpha_\tau}. $$ 
This bound also holds for $\tau<p$ since then $B_p^s(\Omega)$ is continuously embedded in $B_\tau^s(\Omega)$.
The function $S$ only depends on the data $(f_i)_{i=1}^{\widetilde m}$.    This shows that every function $\widetilde f\in \cF_{\rm data}$  is within an $L_\tau(\Omega)$ distance $Cm^{-\alpha_\tau}$ of $S$ and hence   the proof of the upper bound in the cases (i) and (ii) follow.
  
\vskip .1in
\noindent{\bf The case (iii):} This is similar to case  (ii) except now we use Theorem \ref{T:SkH1}.

\vskip .1in
\noindent{\bf The case (iv):}  We shall need the following lemma.
   \begin{lemma}
       \label{L:Lqembed}
      Let $\Omega=(0,1)^d$, $d\ge 3$, and let $\gamma=\gamma(d)$ be defined as in \eref{TF4}.  Then, the space $L_\gamma(\Omega)$ is embedded into $H^{-1}(\Omega)$ and we have
       \be 
       \label{Lqembed}
       \|\widetilde f\|_{H^{-1}(\Omega)}\le C\|\widetilde f\|_{L_\gamma(\Omega)},\quad \widetilde f\in L_\gamma(\Omega),
       \ee
       with $C$ depending only on $d$. In the case $d = 2$, then for any $\gamma(2) = 1 < \tau \leq \infty$, the space $L_\tau(\Omega)$ embeds into $H^{-1}(\Omega)$ and we have the bound
       \be 
       \label{Lqembed-d-2}
       \|\widetilde f\|_{H^{-1}(\Omega)}\le C\frac{\tau}{\tau - 1}\|\widetilde f\|_{L_\tau(\Omega)},\quad \widetilde f\in L_\tau(\Omega),
       \ee
       where $C$ is an absolute constant.
   \end{lemma} 
 \begin{proof} 
 Consider first the case $d \geq 3$, in which case $\gamma(d) > 1$. Let $\gamma'$ be the dual index of $\gamma$, that is, $\frac{1}{\gamma'}=\frac{1}{2}-\frac{1}{d}$, so that $\gamma' < \infty$. If $v\in U(H_0^1(\Omega))$ , then by the Sobolev embedding theorem we have
\be 
\label{Lqembed1}
\|v\|_{L_{\gamma'}(\Omega)}\lesssim \|v\|_{H^1(\Omega)}\lesssim |v|_{H_0^1(\Omega)} := \| | \nabla v | \|_{L_2(\Omega)},\quad v\in H_0^1(\Omega),
\ee
where the last inequality is Poincare's inequality.
From H\"older's inequality, it follows that for any $v\in H_0^1(\Omega)$ and any $\widetilde f\in L_\gamma(\Omega)$, we have 
\be 
\label{LLq1}
 \int_\Omega \widetilde f v \le \|\widetilde f\|_{L_\gamma(\Omega)}\|v\|_{L_{\gamma'}(\Omega)}\lesssim \|\widetilde f\|_{L_\gamma(\Omega)}|v|_{H_0^1(\Omega)}.
 \ee 
  This gives
  \be 
\label{LLq2}
 \|\widetilde f\|_{H^{-1}(\Omega)}:= \sup_{v\in U(H_0^1(\Omega))}\int_\Omega \widetilde fv     \lesssim  \|\widetilde f\|_{L_\gamma(\Omega)},\quad \widetilde f\in L_\gamma(\Omega),
 \ee 
 which proves the lemma when $d \geq 3$.

 When $d = 2$, the Sobolev embedding fails at the endpoint since $\gamma' = \infty$. However, for any $\tau > 1$ we have that $\tau' < \infty$, where $\tau'$ is the dual index to $\tau$. In this case we have a compact Sobolev embedding
 \be\label{limiting-sobolev-embedding}
    \|v\|_{L_{\tau'}(\Omega)}\le C_{\tau'}\|v\|_{H^1(\Omega)},
 \ee
 where one can show that the constant $C_{\tau'} = C\tau' = C\tau/(\tau - 1)$ for an absolute constant $C$. For completeness, we sketch this argument. Given a function $f\in H^1(\Omega)$, we consider a multiscale decomposition of $f$
 \be
    f = \sum_{n=0}^\infty (f_n - f_{n-1}),
 \ee
 where $f_n$ denotes a piecewise linear interpolation of $f$ on a triangulation $\mathcal T_n$ made of $C2^{2n}$ triangles of diameter $c2^{-n}$ (here we set $f_{-1} = 0$). We use a standard polynomial interpolation bound and scaling argument, combined with the fact that the linear polynomials are finite dimensional (essentially the Bramble-Hilbert Lemma) to see that
 \begin{equation*}
     \begin{split}
    \|f_n - f_{n-1}\|^{\tau'}_{L_{\tau'}(\Omega)} = \sum_{T\in \cT_{n-1}} \|f_n - f_{n-1}\|^{\tau'}_{L_{\tau'}(T)} &\leq C2^{2n({\tau'}/2 - 1)}\sum_{T\in \cT_{n-1}} \|f_n - f_{n-1}\|^{\tau'}_{L_2(T)}\\
    &\leq C2^{2n({\tau'}/2 - 1)}2^{-{\tau'}n}\sum_{T\in \cT_{n-1}} \|f\|^{\tau'}_{H^1(T)}.
\end{split}
\end{equation*}
 Taking ${\tau'}$-th roots, we then get using that the $\ell_{\tau'}$-norm is bounded by the $\ell_2$-norm since ${\tau'} \geq 2$ that
 \begin{equation*}
 \begin{split}
    \|f_n - f_{n-1}\|_{L_{\tau'}(\Omega)} \leq C2^{-2n/{\tau'}}\left(\sum_{T\in \cT_{n-1}} \|f\|^{\tau'}_{H^1(T)}\right)^{1/{\tau'}} &\leq C2^{-2n/{\tau'}}\left(\sum_{T\in \cT_{n-1}} \|f\|^2_{H^1(T)}\right)^{1/2}\\
    &= C2^{-2n/{\tau'}}\|f\|_{H^1(\Omega)},
\end{split}
 \end{equation*}
where $C$ is independent of ${\tau'}$. Using the multiscale decomposition, this now implies that
 $$
 \|f\|_{L_{\tau'}(\Omega)} \leq C\|f\|_{H^1(\Omega)}\sum_{n=0}^\infty 2^{-2n/{\tau'}} = \frac{C}{1 - 2^{-2/{\tau'}}}\|f\|_{H^1(\Omega)} \leq C{\tau'}\|f\|_{H^1(\Omega)}
 $$
 for a constant $C$ independent of ${\tau'}$. Utilizing \eref{limiting-sobolev-embedding}, the same duality argument used for $d\geq 3$ now completes the proof when $d = 2$.
 \end{proof}
  
     We can now prove the upper bound in the case (iv) when error is measured in $H^{-1}(\Omega)$. Let us consider first the case where $d \geq 3$. Since the exponent $\alpha_{-1}$ does not change over the $p\ge \gamma$, it is enough to consider the case when $p\le \gamma$.     Indeed, if $p>\gamma$
     then $U(B_q^s(L_p(\Omega)))$ is contained in $U(B_\infty^s(L_\gamma(\Omega)))=:U(B_\gamma^s(\Omega))$. Therefore the upper bound  follows once we have established the case $\cF= U(B_q^s(L_p(\Omega)))$ with $p\le \gamma$.  Similarly, when $p\le \gamma$, then $U(B_q^s(L_p(\Omega)))\subset U(B_\infty^s(L_p(\Omega))$. 
     Accordingly, in going further we only  need to  consider the case when $\cF= U(B_\infty^s(L_p(\Omega))=:U(B_p^s(\Omega))$ with $p\le \gamma$.

      Let $f$ be any function in $\cF:=U(B_\infty^s(L_p(\Omega))$ with $s>d$ and $p\le \gamma$. We know that $f\in C(\Omega)$.   Given the data ${\bf f}$,
     we define    $S$ as in \eref{defS2}.   From \eref{Tpwpinter},  we have
\be 
\label{gamma1}
\|  f-S\|_{L_\gamma(\Omega)}\preceq |f|_{B_p^s(I)}  2^{-kd\alpha_{-1}}. \ee

  We now use Lemma \ref{L:Lqembed} to obtain
  \be 
  \label{gamma2}
  \|f-S\|_{H^{-1}{\Omega}}\le C\|f- S\|_{L_\gamma(\Omega)}\le C m^{- \alpha_{-1}}|f|_{B_p^s(\Omega)}.
  \ee 
  The function $S$ only depends on the data.  This shows that every element in $\cF_{\rm data}(f)$ is within an $H^{-1}(\Omega)$ distance $C m^{-\alpha_{-1}}$ of $S$ and hence proves that the Chebyshev radius
  of this set   does not exceed $Cm^{-\alpha_{-1}}$.  Since $f$ was arbitrary we obtain the same bound for $R^*(\cF)$.
  This concludes
  the proof of the upper bound in case (iv) when $d \geq 3$.

  If $d = 2$, we must slightly modify the above argument. Suppose first that $p > 1$. In this case, we have
  \be 
\label{gamma1-d2}
\|  f-S\|_{L_p(\Omega)}\preceq |f|_{B_p^s(I)}  2^{-kd\alpha_{-1}}, \ee
since the sampling numbers in $L_p$ and $L_1$ are the same in this case. Thus, Lemma \ref{L:Lqembed} implies that (setting $\tau = p$)
  \be 
  \label{gamma2-d2}
  \|f-S\|_{H^{-1}{\Omega}}\le C_p\|f- S\|_{L_p(\Omega)}\le C_pm^{- \alpha_{-1}}|f|_{B_p^s(\Omega)}.
  \ee
  On the other hand, if $p \leq 1$, we see that
  \be 
\label{gamma1-2d-pleq1}
\|  f-S\|_{L_\tau(\Omega)}\preceq |f|_{B_p^s(I)}  2^{-kd(\alpha_{-1} - 1 + 1/\tau)}. \ee
Applying Lemma \ref{L:Lqembed} we get that
  \be 
  \label{gamma2-2d-pleq1}
  \|f-S\|_{H^{-1}{\Omega}}\le C_p\|f- S\|_{L_p(\Omega)}\le \frac{C}{\tau - 1}m^{- \alpha_{-1}}m^{1 - 1/\tau}|f|_{B_p^s(\Omega)}.
  \ee
  Since this is true for all $\tau > 1$ we finally optimize in $\tau$, which gives 
  \be
    \min_{\tau > 0}\frac{m^{1 - 1/\tau}}{\tau - 1} \leq C\log(m)
  \ee
  by setting $\tau = 1 + (\log(m))^{-1}$.
  \hfill $\Box$
\vskip .1in

 \noindent
 {\bf Proof of the lower bounds in Theorem \ref{T:ORf}.}
 We shall now prove the lower bounds in Theorem \ref{T:ORf}. 
 The proofs of lower bounds all take the following form. Suppose that $\bx_i$, $i=1,\dots, m$, are any $m$ data sites. In order to prove a lower bound for $R^*(\cF)_X$, we construct a function $\eta$ in $\cF$ so that
 \vskip .1in
 \noindent 
 (a) $\eta(\bx_i)=0,\quad i=1,\dots,m$,
 \vskip .1in
 \noindent 
 (b) $\|\eta\|_X\ge cm^{-\alpha}$,  
 
 \vskip .1in
 \noindent
 where $\alpha$ is the appropriate index for $X$.
 Since both $\eta$ and the zero function  satisfy zero data the bound (b) gives the lower bound we want for $R^*(\cF)_X$.
 
 We proceed to explain how to construct an appropriate function $\eta$ for each of the    choices of $X$.   Let $\varphi$ be a smooth non-negative function on $\R^d$ which   vanishes outside
 $\Omega$ and additionally satisfies
 \be
 \label{phiinfinity}
  \|\varphi\|_{L_\infty(\R^d)}=1 \quad {\rm and}\quad \varphi(x)\ge 1/2, \quad x\in\Omega_0,
 \ee 
  where $\Omega_0:=[1/4,3/4]^d$. Of course, there are many such functions $\varphi$ but to  be more specific, and for
  use further, we assume $\varphi(\bx)=\phi(x_1)\cdots \phi(x_d) $ where $\phi$ is a univariate function with these properties (for $d=1$).

  We choose    $\varphi$ with these properties which has the smallest norm
    \be 
    \label{phinorm}
    M_{s,p,q}:=\|\varphi\|_{B_q^s(L_p(\Omega))}. 
    \ee
   So, $\varphi$  depends on $s,p,q$ and $M_{s,p,q}$ is a fixed constant since the parameters $s,p,q$ are fixed.
   
 Now consider any cube  $I\subset\Omega$ and let  $\xi_I$ be the smallest vertex  of $I$ and as usual $\ell_I$ is its side length.  We   define the function
 \be
 \label{dilate}
\varphi_I:= \varphi_{I,s,p,q}(x):= \ell_I^{ s -\frac{d}{p}} \varphi(\ell_I^{-1}(x-\xi_I)),\quad x\in\R^d.
 \ee 
This function vanishes outside $I$ and on its boundary.   Moreover, one easily checks  that for all cubes $I$, we have
 \be
 \label{dilate1}
 \|\varphi_I\|_{B_q^s(L_p(\Omega))}= M_{s,p,q}.
 \ee 
 Here the   norm is independent of $I$.

Because of the monotonicity of $R_m^*(\cF)_X$, it is enough to prove the lower bound when $m=2^{kd}$ as we now assume for $m$. Suppose that $m$ sample points $X_m := \{\bx_1,...,\bx_m\}\in \Omega$ are given. Consider the regular tensor product grid $G_{k+2, 2}$ with spacing $ 2^{-(k+2)}$. The number of cubes in this grid is   equal to  $4^dm$ by construction and thus at least $(4^d-1)m$ cubes do not contain any sample points in its interior. Denote by $\cI:=\mathcal{I}(X_m)$ the set of cubes which do not contain any sample points, so that $|\cI| \geq (4^d-1)m = C_dm$ and the volume of each cube $I\in \cI$ is $|I| \geq c_dm^{-1}$.
 For any cube $I\in \cI$, we define
 \be 
 \label{eta}
   \eta_I:= M^{-1}\varphi_I,\quad M=M_{s,p,q}.
  \ee
  It follows that each $\eta_I$ is in $  \cF$.  Also, each of the $\eta_I$ vanishes at each of the data sites.  We shall use the $\eta_I$ to prove the lower bounds of Theorem \ref{T:ORf}.

\vskip .1in
\noindent 
{\bf The lower bound in case (i):} Let us fix $s,p,q$ and
   consider the functions  $ \eta_I$ defined by \eref{eta}.  We take $I$ as any  fixed cube from $\cI$ and define $\eta:=\eta_I$. Then $\eta$ vanishes at each of the data sites. Also, it follows from the definition of $\eta$ that with $M=M_{s,p,q}$, we have
$$
\|\eta\|_{L_\infty(\Omega)}= M^{-1}\|\varphi_I\|_{L_\infty(\Omega)} =M^{-1} \ell_I^{ s -\frac{d}{p}}\ge c m^{-(\frac{s}{d} -\frac{1}{p})}=cm^{-\alpha_C}.
$$
This gives that $R^*(\cF)_{C(\Omega)} \ge cm^{-\alpha_C}$ and thereby proves the lower bound in case (i).
 \vskip .1in
\noindent 
{\bf The lower bound in case (ii):}  
 We   consider separately the cases $\tau < p$ and $\tau \geq p$. If $\tau \geq p$, then we take $\eta = \eta_I$  where $I$ is any single fixed cube in $\cI$. For $M=M_{s,p,q}$, we have from \eref{phiinfinity} that
$$
\|\eta\|_{L_\tau(\Omega)}= M^{-1}\|\varphi_I\|_{L_\tau(\Omega)} \ge c M^{-1} \ell_I^{ s -\frac{d}{p}+\frac{d}{\tau}}\ge c m^{-(\frac{s}{d} -\frac{1}{p}+\frac{1}{\tau})}=cm^{-\alpha_\tau},
$$
since $1/p - 1/\tau \geq 0$. 
 It follows that $R^*(\cF)_{L_\tau(\Omega)} \ge cm^{-\alpha_\tau}$  in the case $\tau\ge p$.

 Next consider the case  $\tau < p<\infty$.  
  Since $U(B_q^s(L_\infty(\Omega)))\subset U(B_q^s(L_p(\Omega)))$, it is sufficient to prove the lower bound in the case $p=\infty$, i.e., when $\cF=U(B_q^s(L_\infty(\Omega)))$ with $0<q<\infty$ arbitrary but fixed.   We define 
 \be\label{neweta1}
    \eta =  \kappa\sum_{I\in \cI} \eta_I,
 \ee
 with $\kappa$ a fixed constant   chosen so that
 $\eta\in U(B_q^s(L_\infty(\Omega)))$. Here, $\eta_I$ is defined as in \eref{eta} with $p=\infty$. We next derive a lower bound for $\kappa$.
 The terms $\eta_I$ in \eref{neweta1} each have $L_\infty$ norm bounded by $M2^{-ks}$ and they have disjoint supports.  Therefore, $\|\eta\|_{L_\infty(\Omega)}\le M2^{-ks}$.
 For each $r\ge 1$, these functions are in $C^{r}(\Omega)$ with $C^r(\Omega)$ norms not exceeding $M2^{-ks}2^{kr}$ where $M$ depends only on the choice of $r$. We take $r:=\lceil s\rceil +1$ and find   that the modulus of smoothness of $\eta$ satisfies
 \be 
 \label{etamo}
 \omega_r(\eta,t)_{\infty}\le M \kappa 2^{-ks}\min(1,2^{kr}t^r) .
 \ee
Therefore, breaking the integral into the integral over $[0,2^{-k}]$ and $(2^{-k},1]$ we find that 
\be 
\label{etaq}
\int_0^1 [t^{-s}\omega_r(\eta,t)_\infty]^q\frac{dt}{t}\le \kappa^q  M^q2^{-kq(s-r)}\int_0^{2^{-k}}  t^{(r-s)q}\frac  {dt}{t}+ \kappa^q M^q2^{-kqs}\int_{2^{-k}}^1 t^{-sq}  \frac {dt}{t}\le \kappa^qC^q,
\ee
with $C$ a fixed constant.
 This shows that we can take $\kappa\ge c$ where $c$ depends only on $s$.  Since $\eta\ge \kappa 2^{-ks}$
 on a set of measure $\ge 1/2$, we have that $\|\eta\|_{L_\tau(\Omega)}\ge c2^{-ks}\ge cm^{-s/d}$ which finishes the proof of the lower bound in case (ii). 
 \vskip .1in
\noindent 
{\bf The lower bound in case (iii):}  By assumption we have $p \leq 2$.  We take $\eta = \eta_I$ for any fixed $I\in \cI$. We estimate
\be
\|\eta\|_{H^1(\Omega)}= M^{-1}\|\varphi_I\|_{H^1(\Omega)} \ge c M^{-1} \ell_I^{s - 1 - \frac{d}{p}+\frac{d}{2}} \geq cm^{-(\frac{s}{d}-\frac{1}{d}-\frac{1}{p}+\frac{1}{2})},
\ee
which proves the lower bound in (iii) for  $R^*(\cF)_{H^1(\Omega)}$.

 \vskip .1in
\noindent 
{\bf The lower bound in case (iv):} We consider first the case where $p \leq \gamma$, where $\gamma$ is given by \eref{TF4}. In this case, we choose $\eta=\eta_I$
where $I$ is any cube in $\cI$. The function $\eta$ vanishes at all of the data sites. By construction, we have that $\eta\in\cF$ and the lower bound
\be 
\label{lbg}
\eta(x)\ge cm^{-\frac{s}{d}+\frac{1}{p}}, \quad x\in I_0,
\ee 
where $I_0=[\xi_I-\ell_I/4,\xi_I+1/4]^d$ is the subcube corresponding to $\Omega_0$ in $I$ (see \eref{phiinfinity}).

We want to bound the $H^{-1}(\Omega)$ norm of $\eta$ from below.  For this,
we choose a function $v$ from the unit ball of $H_0^1(\Omega)$
which  vanishes outside of $I$ and is large on $I_0$.  Namely, we take
\be 
\label{defv1}
v(x):= c\ell_I^{ 1 -\frac{d}{2}} \varphi(\ell_I^{-1}(x-\xi_I)).
\ee 
where now $c$ is chosen as a constant depending only on $d$ so that $\|v\|_{H_0^1(\Omega)}=1$.
As in the case for $\eta_I$, we have
\be 
\label{choosev}
v(x)\ge c\ell_I^{ 1 -\frac{d}{2}} \geq cm^{-\frac{1}{d}+\frac{1}{2}}, \quad x\in I_0.
\ee 
 It follows that
 \be 
 \label{H-1g}
 \|\eta\|_{H^{-1}(\Omega)}\ge \int_\Omega \eta(x)v(x)\, dx= \int_{I_0}\eta(x)v(x)\, dx\ge 
 cm^{-\frac{s}{d}+\frac{1}{p}}m^{-\frac{1}{d}+\frac{1}{2}}|I_0|.
 \ee 
 Since $|I_0|\ge cm^{-1}$ the right side of \eref{H-1g} is 
 $$
 cm^{-\frac{s}{d}+\frac{1}{p} -\frac{1}{d}-\frac{1}{2}}=cm^{-\frac{s}{d}+\frac{1}{p}-\frac{1}{\gamma}}.
 $$
 This proves that $R^*(\cF)_{H^{-1}(\Omega)} \ge cm^{-\alpha_{-1}}$ in the case $p\leq \gamma$ since $1/p-1/\gamma \geq 0$.

We finally consider the case when $p > \gamma$.  This is handled in a similar way to   case (ii).   Namely,  we can assume $\cF=U(B_q^s(L_\infty(\Omega)))$ where $s,q$ are fixed.  We take $\eta$ as in \eref{neweta1}.  We know this function is in $\cF$ and vanishes at all of the data sites.   To provide a lower bound the $H^{-1}(\Omega)$ norm of $\eta$,
\be
    v(x) = c\varphi(x),
\ee
where $c$ is a constant chosen so that $\|v\|_{H_0^1(\Omega)} = 1$. By construction, we have $v(x) \geq 0$ for all $x\in \Omega$ and $v(x) \geq c$ for $x\in \Omega_0$.  Let   $\cI_0$ be the set of $I\in\cI$ such that
 $I\subset \Omega_0$. Since $\eta$ is also non-negative, we estimate
\be\label{eta-h-1-lower-bound}
    \|\eta\|_{H^{-1}(\Omega)} \geq \int_{\Omega} \eta(x)v(x)dx \geq \int_{\Omega_0}\eta(x)v(x)dx \geq \int_{\Omega_0} \eta(x)dx\ge \kappa \sum_{I\in \cI_0}\int_I\eta_I\ge c2^{-ks} 2^{-kd} \#(\cI_0), 
\ee
 because $\kappa\ge c$ and $\eta_I\ge 2^{-ks}$
 on $I_0$.
We need to estimate the cardinality of the set $\cI_0$. Observe that the number of cubes $I$ in the original grid which are contained in $\Omega_0$ is $2^{-d}(4^dm) = 2^dm$. Since $\#(\cI) \geq (4^d - 1)m$, we see that
$$
    \#(\cI_0) \geq 2^dm - m = (2^d - 1)m.
$$
Since $m=2^{kd}$, placing this lower bound of $\#(\cI_0)$ into \eqref{eta-h-1-lower-bound}, we find that
\be
    \|\eta\|_{H^{-1}(\Omega)}  \geq cm^{-\frac{s}{d}}.
\ee
This implies that $R^*(\cF)_{H^{-1}(\Omega)} \ge cm^{-\alpha_{-1}}$ in the case $p > \gamma$ since $1/p-1/\gamma < 0$ and therefore completes the proof of the lower bound in (iv). This completes the proof of Theorem \ref{T:ORf}.\hfill $\Box$

\subsubsection{The proof of Theorem \ref{T:ORg}}
\label{T3.2}

 We first prove (ii) which also gives the upper bound in \eref{ORGrate}.  From the definition of the trace norm, we have
\be 
\label{org1}
\|g-\overline S_k(g)\|_{H^{1/2}(\partial\Omega)}\le \|v-S_k\|_{H^1(\Omega)}\le C [r2^k]^{-t+1}\le Cm^{-\beta},
\ee 
where the next to last inequality uses the estimate \eref{projapprox3} and also the fact that $m\asymp 2^{k(d-1)}$.  This proves (ii).

It follows that   
\be 
\label{org2}
R_m(g)_{H^{1/2}(\partial\Omega)}\le   Cm^{-\beta},
\ee 
for the values of $m$ that equal $2d[(r-1)2^k]^{d-1}$.  Since $g\in\cG$ was arbitrary, we obtain
\be 
\label{org3}
R^*_m(\cG)_{H^{1/2}(\partial\Omega)}\le    Cm^{-\beta},
\ee
for the above values of $m$.  From the monotonicity of $R_m^*$ we obtain \eref{org3} for all $m$.
Thus, we have proven the upper bound in \eref{ORGrate}.  

 We next prove the lower inequalities in \eref{ORGrate}.  From the monotonicity of $R_m^*(\cG)_{H^{1/2}(\partial\Omega)}$, it is enough to prove this lower bound for $m=2^{k(d-1)}-1$ whenever $k$ is any non-negative integer. The following  reasoning is the same as in the proof of the lower inequalities in Theorem \ref{T:ORf}. 

Let $\cZ:=\{\bz_1,\dots,\bz_m\}$, $m=2^{k(d-1)}-1$, be any proposed set of data sites on $\partial\Omega$.  Let $F:=\{\bx\in\Omega:\ \bx\cdot e_1=0$, $e_1=(1,0,\dots,0)\in\R^d$, be the face of $\partial \Omega$ corresponding to  the  points $\bx\in \overline\Omega$ whose first coordinate is equal to zero.   Consider the set $\cD_k(F)$ of ($d-1$ dimensional) dyadic cubes  of $F$. Since there are $2^{k(d-1)}$ cubes in
$\cD_k(F)$, it follows that there is a $\overline J\in\cD_k(F)$ such that $\overline J$ contains none of the data sites from $\cZ$ in its interior.  We will now construct an appropriate function $\eta\in \cG$ which vanishes at
each of the data sites.

Let $J$ be the $d$ dimensional cube in $\cD_k(\overline \Omega)$ which has $\overline J$ as a face and let  $\varphi_J$ be the function defined in \eref{dilate} with the parameters $\bar s, d, \bar p,\bar q$.  We know that
$M^{-1}\varphi_J\in U(B_{\bar q}^{\bar s}(L_{\bar p}(\Omega))$ when $M:=M_{\bar s,\bar p,\bar q}$.  
If  $\xi_J'$ is the 
center of  $J$, then
 the point $\xi_J'-(2^{-k-1}, 0,\dots,0)$ is the center of $\overline J$.
 
We now define
\be
v(\bx) := M^{-1} \varphi_J(\bx-(2^{-k-1}, 0,\dots,0)  ),\quad x\in\Omega,
\ee
which is also a function in $U(B_{\bar q}^{\bar s}(L_{\bar p}(\Omega))$.  Hence, the function 
\be 
\label{defetag}
\eta:=T_{\partial\Omega}v
\ee
is in our model class $\cG$ and $\eta$ vanishes at all of the data sites $\cZ$.

We now show that for the intrinsic $H^{1/2}(\partial\Omega)$ norm we have
\be 
\label{lbeta1/2}
\|\eta\|_{H^{1/2}(\partial\Omega)}\ge cm^{-\beta},
\ee 
with $c$ not depending on $m$.  This will prove the lower bound we seek. 
First consider the function $\varphi$ defined in \eref{phiinfinity}. Let $e_1=(1,0,\dots,0)\in\R^d$. Because of the tensor product structure of $\varphi$ described after \eqref{phiinfinity}, the trace $\eta_0$ of $\varphi$ onto the hyperplane $\bx\cdot e_1=1/2$ is the $d-1$ version of $\varphi$. We define
\be 
\label{tn}
\overline M:=|\eta_0|_{H^{1/2}(\partial\Omega)}>0.
\ee 
By a change of variables and the fact that $m=2^{k(d-1)}-1$, it follows that

\be
    |\eta|_{H^{1/2}(\partial\Omega)} = \overline M M^{-1} 2^{-k(\bar s-d/{\bar p})}2^{k/2} 2^{-k(d-1)/2} \ge c m^{-\beta},
\ee
  where $c>0$ does not depend on $m$.  This proves the lower bound and completes the proof of the theorem.
  \hfill $\Box$

\subsection{The proof of Theorem \ref{T:dloss} in the case $d=2$}
\label{d2}
 In this section we provide the proof of Theorem \ref{T:dloss} in when   $d=2$. We discuss two cases.
 
 {\bf Case 1:} $p=1$
 
 Let $\varepsilon>0$. According to Lemma \ref{L:Lqembed}, Lemma \ref{L:Ldiscrete} and Theorem \ref{T:dH121}, we have
 \begin{eqnarray*} 
   \|u-v\|_{H^1(\Omega)}&\lesssim& \|f+\Delta v\|_{H^{-1}(\Omega)} +\|g-Tr(v)\|_{H^{1/2}(\partial\Omega)}
  \nonumber\\
 &\lesssim &\varepsilon^{-1}\|f+\Delta v\|_{L_{1+\varepsilon}(\Omega)}+
 \|g-Tr(v)\|_{H^{1/2}(\partial\Omega)}\nonumber\\
 &\lesssim& \varepsilon^{-1}\|f+\Delta v\|^*_{L_{1+\varepsilon}(\Omega)}+
 \|g-Tr(v)\|^*_{H^{1/2}(\Omega)}\\ \nonumber
 &+&
 \left[\|f+\Delta v\|_{\cB} \varepsilon^{-1}\widetilde m^{1-1/(1+\varepsilon)}\widetilde  m^{-\frac{s}{2}}+\|g-Tr(v)\|_{Tr(\overline \cB)} \overline m^{-(\bar s-1)}\right].
 \end{eqnarray*}
 Optimizing $\varepsilon$ as in Lemma \ref{L:Lqembed}
 gives a choice of $\varepsilon=[\log(\widetilde m)]^{-1}$, for which 
 $$
 \varepsilon^{-1}\widetilde m^{1-1/(1+\varepsilon)}\lesssim \log(\widetilde m)
 $$
 and thus we have
 \begin{eqnarray*} 
   \|u-v\|_{H^1(\Omega)}
 &\lesssim& \log (\widetilde m)\|f+\Delta v\|^*_{L_{1+\varepsilon}(\Omega)}+
 \|g-Tr(v)\|^*_{H^{1/2}(\Omega)}\\ \nonumber
 &+&
 \left[\|f+\Delta v\|_{\cB} \log (\widetilde m)\widetilde  m^{-\frac{s}{2}}+\|g-Tr(v)\|_{Tr(\overline \cB)} \overline m^{-(\bar s-1)}\right]\\ \nonumber
 &\lesssim& \cL^*(v)+\left [1+\| v\|_{\cU}\right] \cR_{\cU}(\widetilde{m}, \overline{m}).
 \end{eqnarray*}

{\bf Case 2:} $1<p\leq \infty$
 
 \begin{eqnarray*} 
   \|u-v\|_{H^1(\Omega)}&\lesssim& \|f+\Delta v\|_{H^{-1}(\Omega)} +\|g-Tr(v)\|_{H^{1/2}(\partial\Omega)}
  \nonumber\\
 &\lesssim &C(p)\|f+\Delta v\|_{L_{p}(\Omega)}+
 \|g-Tr(v)\|_{H^{1/2}(\partial\Omega)}\nonumber\\
 &\lesssim& \left [C(p)\|f+\Delta v\|^*_{L_p(\Omega)}+
 \|g-Tr(v)\|^*_{H^{1/2}(\Omega)}\right ].
 \\ \nonumber
 &+&
 \left[C(p)\|f+\Delta v\|_{\cB} \widetilde  m^{-\frac{s}{2}}+
 \|g-Tr(v)\|_{Tr(\overline \cB)} \overline m^{-(\bar s-1)}\right]\\ \nonumber 
 &\lesssim& \left [C(p)\|f+\Delta v\|^*_{L_p(\Omega)}+
 \|g-Tr(v)\|^*_{H^{1/2}(\Omega)}\right ]+\left [1+\| v\|_{\cU}\right] \cR_{\cU}(\widetilde{m}, \overline{m}).
 \end{eqnarray*}
Notice that in this case we could have chosen $\cL^*$ as in the case of $d=3$ with $\gamma=p$ and obtained the  optimal recovery rate for this class. However, this choice would result in a loss functional depending on $\cal F$.

\bibliography{biblio}
\bibliographystyle{amsplain} 
\vskip .1in
\noindent
Andrea Bonito, Department of Mathematics, Texas A\&M University, College Station, TX 77843, email: bonito@tamu.edu.
\vskip .1in
\noindent
Ronald DeVore, Department of Mathematics, Texas A\&M University, College Station, TX 77843, email: rdevore@tamu.edu.
\vskip .1in
\noindent
Guergana Petrova, Department of Mathematics, Texas A\&M University, College Station, TX 77843, email: gpetrova@tamu.edu.
\vskip .1in
\noindent
Jonathan W. Siegel, Department of Mathematics, Texas A\&M University, College Station, TX 77843, email: jwsiegel@tamu.edu.

\end{document}